\documentclass[11pt,a4paper]{amsart}
\usepackage[square, numbers, comma]{natbib}
\usepackage{amssymb}
\usepackage{amsmath}
\usepackage{amsthm}
\usepackage{amsfonts}
\usepackage{dsfont}
\usepackage{a4wide}
\usepackage{pdfsync}
\usepackage{hyperref}
\usepackage{enumerate}
\usepackage{graphicx}
\usepackage{subfigure}
\usepackage{color}
\usepackage{cancel}
\usepackage{multirow}
\usepackage{tikz}
\usepackage{mathrsfs}
\usepackage{array}
\usepackage{esint}
\usepackage{enumitem}
\usepackage[justification=centering]{caption}
\usepackage{float}
\usepackage{comment}
\usepackage{caption}
\usepackage{dsfont}

\usepackage{mathtools}
\mathtoolsset{showonlyrefs}
\newtheorem{theorem}{Theorem}[section]
\newtheorem{proposition}[theorem]{Proposition}
\newtheorem{corollary}[theorem]{Corollary}
\newtheorem{lemma}[theorem]{Lemma}

\theoremstyle{remark}
\newtheorem{remark}[theorem]{Remark}

\theoremstyle{definition}

\makeatletter
\renewcommand{\tocsection}[3]{\indentlabel{\@ifnotempty{#2}{\ignorespaces #1 #2.\,}}#3}

\def\l@subsection{\@tocline{2}{0pt}{2em}{}{}}
\def\l@subsection{\@tocline{2}{0pt}{2em}{}{}}
\makeatother

\hypersetup{
	colorlinks   = true, %Colours links instead of ugly boxes
	urlcolor     = blue, %Colour for external hyperlinks
	linkcolor    = blue, %Colour of internal links
	citecolor   = red %Colour of citations
}
\definecolor{darkblue}{rgb}{0.05, .05, .9}
\definecolor{darkgreen}{rgb}{0.1, .65, .1}
\definecolor{darkred}{rgb}{0.8,0,0}

\newcommand{\Rd}{\mathbb{R}^d}
\newcommand{\lh}{\Delta_{\mathbb{H}^d}}

\newcommand{\Hd}{\mathbb{H}^d}

\newcommand{\dhh}{\mathbb{D}^{d}}
\newcommand{\uh}{\mathbb{U}^{d}}

\newcommand{\fal}{\textrm{for all }}
\newcommand{\co}{c_{0}}

\newcommand{\hs}{\mathbb{H}^d}
\newcommand{\rr}{\rho}

\newcommand{\tw}{\Phi_{c_0}}
\usepackage{etoolbox}
\newcounter{taggedeq}
\pretocmd{\equation}{\stepcounter{taggedeq}}{}{}

\begin{document}

\title[KPP equations in the hyperbolic space]{Traveling-wave behavior for Fisher-KPP equations in the hyperbolic space}

\author[M.\,M.~González]{María del Mar González}
\address[M.\,M.~González]{Departamento de Matem\'aticas, Universidad Aut\'onoma de Madrid \& Instituto de Ciencias Matem\'aticas ICMAT (CSIC-UAM-UCM-UC3M). Campus de Cantoblanco, 28049 Madrid, Spain}
\email[]{mariamar.gonzalezn\@@{}uam.es}
\urladdr{https://matematicas.uam.es/~maria.gonzalez}

\author[I.~Gonz\'alvez]{Irene Gonz\'alvez}
\address[I.~Gonz\'alvez]{ Departamento de Matemáticas,
Universidad del País Vasco / Matematika Saila, Euskal Herriko Unibertsitatea, (UPV/EHU), 48080 Bizkaia, Spain}
\email[]{irene.gonzalvez\@@{}ehu.eus}

\author[F.~Quir\'{o}s]{Fernando Quir\'{o}s}
\address[F.~Quir\'{o}s]{Departamento de Matem\'aticas, Universidad Aut\'onoma de Madrid \& Instituto de Ciencias Matem\'aticas ICMAT (CSIC-UAM-UCM-UC3M). Campus de Cantoblanco, 28049 Madrid, Spain}
\email[]{fernando.quiros\@@{}uam.es}
\urladdr{https://matematicas.uam.es/~fernando.quiros}

\begin{abstract}
    We study the Cauchy problem in the hyperbolic space for the heat equation with a Fisher-KPP type forcing term. Depending on the relative strength of diffusion, measured by the infimum of the spectrum of the Laplace-Beltrami operator, as compared to the growth due to the forcing term, solutions may propagate or vanish as time passes. We prove new results concerning this dichotomy that include the critical case where diffusion and reaction are of the same order. If the initial datum possesses some symmetry (invariance under a  cohomogeneity one subgroup of the group of  isometries of the hyperbolic space), the problem reduces to a unidimensional one. In the case of propagation, the solution to this unidimensional problem converges in shape to an Euclidean traveling wave of minimal speed in an appropriate moving frame. The choice of this frame depends on the subgroup of isometries (elliptic, hyperbolic or parabolic)  under which the initial datum is invariant. In contrast with the Euclidean case, the asymptotic spreading speed (in due coordinates) depends on the dimension, while the coefficient of the logarithmic correction in the location of the front does not, no matter the underlying isometry. 
\end{abstract}

%%%%%%%%%%%%%%

\keywords{Hyperbolic space; Fisher-KPP problem; propagation vs.\,vanishing; traveling-wave behavior; isometries}

\subjclass[2020]
{58J35,	%Heat and other parabolic equation methods for PDEs on manifolds
35K57,  %	Reaction-diffusion equations
58J90,	%Applications of PDEs on manifolds
35C07,	%Traveling-wave solutions
35B40,	%Asymptotic behavior of solutions to PDEs
58K55.	%Asymptotic behavior of solutions to equations on manifolds
}

\maketitle
\begin{center}\begin{minipage}{12cm}{\tableofcontents}\end{minipage}\end{center}

%%%%%%%%%%%%%%%%%%%%%%%%%%%%%%%%%%%%%%%%%%%%%%%%%%%%%%%%%%%%%%%%%%%%%%%%%
%%%%%%%%%%%%%%%%%%%%%%%%%%%%%%%%%%%%%%%%%%%%%%%%%%%%%%%%%%%%%%%%%%%%%%%%%
\section{Introduction and statement of results}
\setcounter{equation}{0}

We study the sharp long-time behavior of solutions to the Cauchy problem for the Fisher-KPP equation posed in the hyperbolic space $\Hd$, $d\geq 2$,
\begin{equation}\label{KPPproblem}\tag{KPP}
    u_{t}=\lh u+f(u)\quad\text{in }\Hd\times \mathbb{R}_{+},\qquad u(\cdot,0)=u_{0}\quad\text{in }\Hd.
\end{equation}
Here $\lh$ is the Laplace-Beltrami operator in $\Hd$ and $\mathbb{R}_+:=(0,\infty)$. Our motivation is to understand how the underlying hyperbolic geometry affects the evolution of solutions to~\eqref{KPPproblem}, in comparison with the Euclidean case.

Throughout the paper we always assume that $0\le u_{0}\le 1$ and that $f$ is a \emph{KPP function},
\begin{align}    
    \label{KPPfunction.1}\tag{$\textup{H}_f$1}
	&f\in C^{2}([0,1]),\quad  f(0)=f(1)=0,\quad  f'(1)<0<f'(0),\quad f(u)>0 \quad \textup{for all } u\in (0,1),    \\
\label{KPPfunction.2}\tag{$\textup{H}_f$2} 
    &f(u)\leq f'(0)u\quad \fal u\in (0,1).
\end{align}
Under these assumptions, problem~\eqref{KPPproblem} has a unique classical bounded solution, which satisfies $0\le u\le1$.
For our main convergence theorem we need to substitute~\eqref{KPPfunction.2} by the stronger condition
\begin{equation}
\label{KPPfunction.2*}\tag{$\textup{H}_f2^*$}
    f'(u)\leq f'(0)\quad\textrm{for all }u\in[0,1].
\end{equation}
The main example of a KPP function, satisfying also hypothesis~\eqref{KPPfunction.2*}, is $f(u):=u(1-u)$, $u\in[0,1]$.  Condition~\eqref{KPPfunction.2*} is  a common requirement in sharp convergence results for the Euclidean problem.

Under assumption~\eqref{KPPfunction.1} on the nonlinearity $f$, the ODE $u'=f(u)$, corresponding to spatially homogeneous solutions to~\eqref{KPPproblem},  has two stationary solutions, one stable, $u\equiv 1$, and the other unstable, $u\equiv 0$.  For general (non‑homogeneous) solutions, however, diffusion interacts with the reaction term, and the stability properties of both equilibria may change; in particular, neither the stability of $u\equiv 1$ nor the instability of $u\equiv 0$ can be taken for granted without further analysis. When the problem is posed in the Euclidean space, diffusion is never strong enough, and all nontrivial solutions converge to $1$ as $t\to\infty$, a situation known as \emph{propagation} (or spreading).  However, in the hyperbolic space, diffusion ---which can be measured in terms of $\lambda_1>0$, the infimum of the spectrum of the Laplace-Beltrami operator--- is stronger than in the Euclidean case, and it  may outweigh the repulsion of the level 0 ---measured in terms of $f'(0)$---, leading solutions to converge to 0 as $t\to\infty$.  This second possibility is known as \emph{vanishing} (or extinction). This had already been noticed in the pioneering work~\cite{Matano} by Matano, Punzo, and Tesei, who proved that, in the case of strong repulsion, i.e., $f'(0)>\lambda_1$,  there is propagation, while if $f'(0)<\lambda_1$, diffusion wins the game and solutions vanish. Here, as a first objective, we further explore  the dichotomy between propagation and extinction, covering in particular the critical case $f'(0)=\lambda_1$, which was left open in~\cite{Matano}.

Another significant difference that arises with the introduction of geometry is that, in the case of spreading,  the propagation speed in $\mathbb H^d$ is slower than in $\mathbb{R}^d$, due to the stronger diffusion  mentioned above. This effect is described as the \textit{drift from infinity} in the work~\cite{Matano}, as it  directly affects the velocity of the moving frame. 

Our second goal is to  exploit the invariances of the hyperbolic space in order to describe the long-time behavior of the solution to~\eqref{KPPproblem} in the spreading scenario. More precisely, if the initial datum has some symmetry (in terms of cohomogeneity one subgroups of isometries of the hyperbolic space), the solution itself has the same symmetry in space for any later time, which reduces the problem to a unidimensional equation.  Let us remark that, differently from the Euclidean case, where such isometries consist of rotations and translations only, the hyperbolic space has a more complex geometry with three types of this type of isometries: elliptic, parabolic, and hyperbolic. We will give a moving coordinate system that takes into account this symmetry, in which  the solution to~\eqref{KPPproblem} with a compactly supported initial data (with respect to the symmetry) converges to a specific (Euclidean) traveling-wave profile of minimal speed with a shift. It turns out that the term taking into account the propagation speed needs to be modified with a logarithmic correction which, contrary to the Euclidean case, stays the same across all its invariances, and does not depend on the dimension. The proof adapts ideas developed by Roquejoffre, Rossi and Roussier-Michon in~\cite{RoquejofferNdimensional} for the Euclidean setting to take into account the difficulties introduced by the underlying geometry.

To explore this in greater detail, we first provide an in‑depth introduction to the well‑established theory of the KPP problem in the Euclidean space, followed by a presentation of the existing theory for the equation posed in the hyperbolic space. We conclude the section with our main results and a discussion of the heuristics that motivate them.

\medskip

\noindent{\bf The KPP problem in $\mathbb{R}^d$.}
We summarize here some landmark results, relevant for this work, from the vast literature on the KPP problem in the Euclidean setting. 

In 1937, Fisher, on the one side, and Kolmogorov, Petrovsky and Piskunov on the other, published simultaneously, but independently, two papers, \cite{Fisher,KPP-article} respectively, in which they describe the evolution of the spatial distribution of a dominant gene within a population. Fisher modeled it by the reaction-diffusion spatial unidimensional equation
\begin{equation*}
	u_{t}=u_{xx}+u(1-u)\quad\textrm{in }\mathbb{R}\times\mathbb{R}_{+},
\end{equation*}
where the term $u(1-u)$ models natural selection as a saturation process, while  the diffusion term $u_{xx}$ accounts for the dispersal of the descendants. On the other hand, Kolmogorov, Petrovsky, and Piskunov proposed a more general formulation:
\begin{equation}
\label{KPPEuclideanOneDimensional}	u_{t}=u_{xx}+f(u)\quad\text{in }\mathbb{R}\times\mathbb{R}_{+},
\end{equation}
where $f$ is a reaction term satisfying condition~\eqref{KPPfunction.1}--\eqref{KPPfunction.2*}. As a consequence, the Cauchy problem
\begin{equation}\label{KPPproblemEuclidean}
	u_{t}=\Delta u+f(u)\quad\text{in }\mathbb{R}^d\times \mathbb{R}_{+},\qquad u(\cdot,0)=u_{0}\quad\text{in }\mathbb{R}^d,
\end{equation}
where  $u_0:\mathbb{R}^d\to[0,1]$ is a given initial condition, is known as the Fisher-KPP problem. For short, we refer to it as the KPP problem in the Euclidean space. This is the Euclidean counterpart of~\eqref{KPPproblem}. 

Being $f$ a KPP function, there is a unique bounded solution $u$ to~\eqref{KPPproblemEuclidean} and $0\leq u\leq 1$ in $ \mathbb{R}^d\times \mathbb{R}_{+}$. Besides,  whenever $u_0$ is not the null function, propagation occurs: the state $u\equiv 1$ invades the state $u\equiv0$; see~\cite{AronsonWeinberger}. 

As proved in~\cite{KPP-article}  (see also~\cite{AronsonWeinberger, Uchiyama-1978}), if $f$ satisfies~\eqref{KPPfunction.1}--\eqref{KPPfunction.2}, equation~\eqref{KPPEuclideanOneDimensional} admits a family of traveling-wave solutions $\{u_c\}_{c\geq c_0}$, that is, solutions of the form $u_c(x,t):=U_{c}(x-ct)$ for all $(x,t)\in\mathbb{R}\times\mathbb{R}_+$, where the minimal speed $c_0$ is given by
\begin{equation}\label{lambda_0} 
    c_0:=2\lambda_{0},\quad\lambda_{0}:=\sqrt{f'(0)}>0.
\end{equation}
For each $c\ge c_0$,  the  traveling-wave profile  $U_c:\mathbb{R}\to\mathbb{R}$ is the unique solution (up to translations) to the ODE
\begin{equation}\label{ODE-tw}	
    U''+cU'+f(U)=0\quad\textrm{in }\mathbb{R}
\end{equation}
that satisfies
\begin{equation}\label{tw-initialdatum}	
     0\le U\le 1,\quad U(-\infty)=\lim\limits_{s\to-\infty}U(s)=1,\quad U(+\infty)=\lim\limits_{s\to+\infty}U(s)=0;
\end{equation}
The profiles $U_c$ are strictly monotone, indeed, $U_c'<0$. 

Among all the translates that solve~\eqref{ODE-tw}--\eqref{tw-initialdatum} with $c=c_0$, we define the \emph{normalized}  traveling-wave  profile of minimal speed $\Phi_{c_0}$  as the unique one that satisfies $\Phi_{c_0}(0)=1/2$. The behavior of this profile at $+\infty$ is given by
\begin{equation}\label{behaviorTw}
    \tw(s)=(s+\kappa)e^{-\lambda_0 s}+O\big(e^{-(\lambda_0+\delta)s}\big),\quad s\to\infty,
\end{equation}
with two constants $\delta>0$ and $\kappa\in\mathbb{R}$.

Kolmogorov, Petrovski and Piskunov proved in~\cite{KPP-article}, assuming~\eqref{KPPfunction.1}--\eqref{KPPfunction.2*}, that the solution~$u$ to ~\eqref{KPPEuclideanOneDimensional}  posed on $\mathbb R$ with initial datum a step function $u_0=\mathds{1}_{(-\infty,0]}$ converges to the traveling-wave profile $\Phi_{c_0}$ in the following sense: after showing that $u_x(\cdot, t)<0$ for all $t>0$ and defining the \emph{centering function} $m:\mathbb{R}_+\to\mathbb{R}$ by $u(m(t),t)=1/2$, they proved that
\begin{equation*}
	u(x+m(t),t)\to\Phi_{c_0}(x)\textrm{ uniformly in }x,\quad m'(t)\to c_0\quad\textrm{as }t\to\infty.
\end{equation*}
This gives a precise description of the propagation of level sets of  $u$ in terms of $m$. The constant $c_0$ can be regarded as the {\it spreading speed} of the species described by the model, since it yields the first term of the expansion of $m$, $m(t)=c_0t+o(t)$, and hence of all level sets. 

What about other initial data? Several years later, Aronson and Weinberger proved in~\cite{AronsonWeinberger-PrimerArticulo(Viejo)}, assuming only~\eqref{KPPfunction.1}--\eqref{KPPfunction.2}, that if~$u_0$ is nontrivial and  compactly supported from the right ($u_{0}(x)=0$ for all $x\geq A>0$), then
\begin{equation}\label{SpreadingSpeedEuclidean}	
    \lim\limits_{t\to\infty}\sup\limits_{\{x>ct\}}u(x,t)=0\quad\textrm{for all }c>c_0,\qquad
	\lim\limits_{t\to\infty}\inf\limits_{\{0<x<ct\}}u(x,t)=1\quad\textrm{for all }0<c<c_0.
\end{equation}
Hence, $c_0$ is also the spreading speed in this case. 

The quest for a more detailed description of the centering term $m$, or equivalently, for  the right moving frame in order to have convergence towards the profile $\Phi_{c_0}$, always under the stronger assumption~\eqref{KPPfunction.1}--\eqref{KPPfunction.2*}, continues with the work of Uchiyama, who proved in \cite{Uchiyama-1978}, under some restrictions on the initial data,   that $m(t)=c_0t- \frac3{c_0}\log  t+O(\log(\log  t))$. Bramson gave the definitive answer, first for  $u_0=\mathds{1}_{(-\infty,0]}$ in~\cite{Bramson-probability-1978}, and later for initial data compactly supported from the right in~\cite{Bramson-probability-1983},  using probabilistic methods: there is a constant $\beta$, depending on $u_0$, such that
\begin{equation}\label{BramsonResult}	
    u\left(x+c_0 t-\tfrac{3}{c_0}\log  t +\beta+o(1),t\right)\to\Phi_{c_0}(x)\quad\textrm{uniformly in }x>0\textrm{ as }t\to\infty.
\end{equation}
Later,  using purely PDE arguments, Hamel, Nolen, Roquejoffre, and Ryzhik observed in~\cite{NolenRoquejofre} this  logarithmic correction in the convergence whenever the  initial datum is compactly supported from the right. As a continuation of this work, the three last authors extended in ~\cite{NolenRoquejofre2} these results, by reproving Bramson's result~\eqref{BramsonResult}. They achieved this through a comparison principle and compactness arguments, for continuous initial data trapped between two step functions. Their approach can also be applied in the nonlocal setting; see~\cite{Roquejoffre-2024}.

Consider now   the multidimensional problem~\eqref{KPPproblemEuclidean} with $d>1$. Aronson and Weinberger proved in their famous paper~\cite{AronsonWeinberger} that  the solution $u$ has also spreading speed $c_0$ whenever the initial datum is compactly supported, that is,
\begin{equation*}\label{SpreadingSpeedEuclidean2}	
	\lim\limits_{t\to\infty}\sup\limits_{\{|x|>ct\}}u(x,t)=0\quad\textrm{for all }c>c_0,\qquad
	\lim\limits_{t\to\infty}\inf\limits_{\{|x|<ct\}}u(x,t)=1\quad\textrm{for all }0<c<c_0.
\end{equation*}
Moreover, as proved in \cite{Gartner-1982}, the moving frame in which  the solution of the multidimensional equation~\eqref{KPPproblemEuclidean} converges to the profile $\Phi_{c_0}$ also  needs a logarithmic correction.

Several years later, for more general initial data,  Roquejoffre, Roussier-Michon and Rossi showed that the choice of the moving coordinate system and the constant of the logarithmic correction depend strongly  on the geometry of the support of the initial datum (see Figure~\ref{FirgureSupports}). More precisely:

\medskip

\noindent (i) \underline{``Ball" case.}   In~\cite{RoquejofferNdimensional} the authors proved that if the initial datum $u_0$ satisfies
\begin{equation}\label{eq:condition.ball.case}
    \mathds{1}_{B_{R_1}(0)}\leq u_0\leq\mathds{1}_{B_{R_2}(0)}\quad\textrm{for some }0<R_1<R_2,
\end{equation}
then there is a  bounded Lipschitz function $\tilde\beta:\mathbb{S}^{d-1}\to\mathbb{R}$ such that
\begin{equation}\label{u-to-TW-ellipticEuclidean}
    \lim\limits_{t\to\infty}\sup\limits_{x\in\mathbb{R}^d\setminus\{0\}}\big|u(x,t)-\Phi_{c_0}\big(|x|-c_0 t+\tfrac{d+2}{c_0}\log t +\tilde\beta(x/|x|)\big)\big|=0.
\end{equation}

\noindent(ii) \underline{``Half-plane" case} (see~\cite{tranlacionesRn}.) If the initial datum $u_0$ is trapped between two translates (in the same direction) of the step function $\mathds{1}_{(-\infty,0)}$,  say,
\begin{equation*}\label{eq:conditions.transl.Roque}
    \mathds{1}_{(-\infty,a)}(x_d)\leq u_0(x)\leq \mathds{1}_{(-\infty,b)}(x_d)\quad\textrm{for }x=(x_1,...,x_d)\in\mathbb{R}^d,
\end{equation*}
with $a<b$, and
\begin{equation}\label{eq:condition.transverse}
    \lim\limits_{|(x_1,\dots,x_{d-1})|\to\infty}  u_0(x)=\hat u_0(x_d)\quad\textup{uniformly in }x_d\in\mathbb{R},
\end{equation}
for some one-dimensional function $\hat u_0$, then, there exists a constant $\beta$ such that   
\begin{equation}\label{u-to-TW-parabolicEuclidean}
    \lim\limits_{t\to\infty}\sup\limits_{x\in\mathbb{R}^d}\big|u(x,t)-\Phi_{c_0}(x_d-c_0 t+\tfrac{3}{c_0}\log  t +\beta)\big|=0.
\end{equation}
\vskip-.2cm
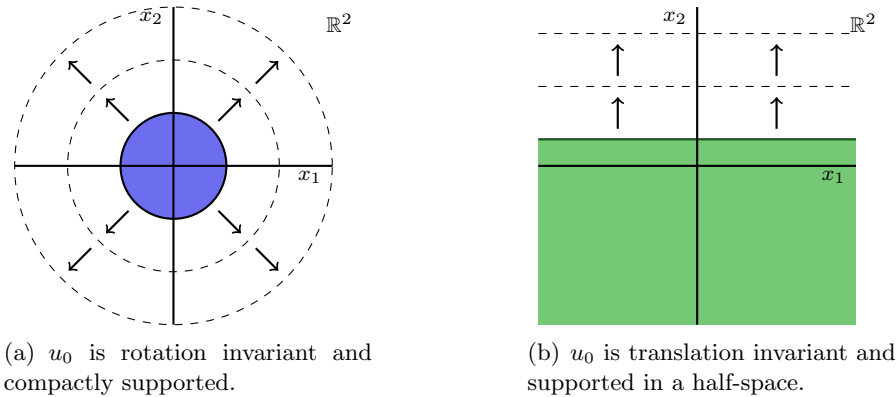
\begin{figure}[H]
\begin{center}
    \subfigure[$u_0$ is rotation invariant and compactly supported.]{
        \begin{tikzpicture}[scale=0.7]
        
            %esfera
            \filldraw[thick, fill=darkblue, fill opacity=0.6] (0,0) circle (1cm);
            \draw[dashed] (0,0) circle (2cm);
            \draw[dashed] (0,0) circle (3cm);
            %esfera pequeña
            %Flechas
            %% I cuadrante
            \draw[thick,->] (1.2*0.7071,1.2*0.7071) -- (1.8*0.7071,1.8*0.7071);
            \draw[thick,->] (2.2*0.7071,2.2*0.7071) -- (2.8*0.7071,2.8*0.7071);
            %% II  cuadrante
            \draw[thick,->] (-1.2*0.7071,1.2*0.7071) -- (-1.8*0.7071,1.8*0.7071);
            \draw[thick,->] (-2.2*0.7071,2.2*0.7071) -- (-2.8*0.7071,2.8*0.7071);
            %III cuadrante
            \draw[thick,->] (-1.2*0.7071,-1.2*0.7071) -- (-1.8*0.7071,-1.8*0.7071);
            \draw[thick,->] (-2.2*0.7071,-2.2*0.7071) -- (-2.8*0.7071,-2.8*0.7071);
            % IV cuadrante
            \draw[thick,->] (1.2*0.7071,-1.2*0.7071) -- (1.8*0.7071,-1.8*0.7071);
            \draw[thick,->] (2.2*0.7071,-2.2*0.7071) -- (2.8*0.7071,-2.8*0.7071);
            
            % Nombres
            \node[right] at (2.7,2.7) {{\scriptsize$\mathbb{R}^{2}$}};
                %Ejes
            \draw[thick] (-3,0) -- (3,0);
            \node[left] at (0, 2.8) {{\scriptsize$x_{2}$}};
            \draw[thick] (0,-3) -- (0,3);
            \node[left] at (3,-0.2) {{\scriptsize$x_{1}$}};
        \end{tikzpicture}
    }\;\;\;\;\;\;\;\;\;\;\;\;\;\;\;\;\;\;
    \subfigure[ $u_0$ is translation invariant and supported in a half-space.]{	\begin{tikzpicture}[scale=0.7]
        
            % Nombres
            \node[right] at (2.7,2.7) {{\scriptsize$\mathbb{R}^{2}$}};
            %Support
            \draw[thick] (-3,0.5) -- (3,0.5);
            \fill[darkgreen, fill opacity=0.6] (-3,0.5) -- (3,0.5) -- (3,-3) -- (-3,-3) -- cycle;
            %Moving of the support
            \draw[thick,->] (-1.5,0.7) -- (-1.5,1.3);
            \draw[thick,->] (1.5,0.7) -- (1.5,1.3);
            \draw[dashed] (-3,1.5) -- (3,1.5);
            \draw[thick,->] (-1.5,1.7) -- (-1.5,2.3);
            \draw[thick,->] (1.5,1.7) -- (1.5,2.3);
            \draw[dashed] (-3,2.5) -- (3,2.5);
                %Ejes
            \draw[thick] (-3,0) -- (3,0);
            \node[left] at (0, 2.8) {{\scriptsize$x_{2}$}};
            \draw[thick] (0,-3) -- (0,3);
            \node[left] at (3,-0.2) {{\scriptsize$x_{1}$}};
        \end{tikzpicture}
    }
    \caption{ \label{FirgureSupports} Propagation of the support of $u$.}
\end{center}
\end{figure}

\vskip-.5cm

\noindent\emph{Remarks.} (i) The condition from below in~\eqref{eq:condition.ball.case} can be relaxed to $u_0\ge 0$, $u_0\not\equiv0$, arguing as in the proof of our main theorem.

\noindent (ii) For the sake of clarity in the presentation, the results in~\cite{tranlacionesRn} for the \lq\lq half-plane'' case were only proved for $d=2$. However, as pointed out in that paper, higher dimensions can be treated in the same way.

\noindent (iii) Convergence in the \lq\lq half-plane'' case is not guaranteed without some condition on the behavior of~$u_0$ as $x\to\infty$ in the transverse direction to $x_d$, as in~\eqref{eq:condition.transverse}, since the behavior at infinity of the initial datum can generate a complex behavior for later times; see the counterexamples in~\cite{tranlacionesRn}.

\noindent (iv) There is a logarithmic correction term, $-\frac{3}{c_0}\log t$, in both cases. The extra correction $-\frac{d-1}{c_0}\log t$ in the \lq\lq ball'' case is a \emph{curvature term} coming from the fact that levels sets are balls in a first approximation.   

\noindent (v) The existence of logarithmic correction terms in reaction-diffusion problems in situations in which curvature does not play a role ---for example, for problems posed in $\mathbb{R}$--- is related to the fact that it is the reaction close to the 0 level that is ruling the propagation. Whether this is or not the case depends both on the diffusion operator and on the reaction nonlinearity. For instance, such a correction does not exist for problem~\eqref{KPPproblemEuclidean} for several nonlinearities that are not of KPP type~\cite{Fife-McLeod-1977,Kanel-1962,Uchiyama-1985,Stokes-1976}, nor for equations with a KPP-type $f$ but with a nonlinear diffusion operator~\cite{Biro-2002,Du-Garriz-Quiros-2025,Du-Quiros-Zhou-2020,Garriz-2020}. In the case of~\eqref{KPPproblem}, when there is propagation it has to be ruled by the reaction, and we will show that we have this ground logarithmic correction term unrelated to curvature.

\medskip

\noindent{\bf The KPP problem in $\Hd$.} The preceding discussion of the KPP problem in the Euclidean space naturally raises several questions about the corresponding problem in the hyperbolic space. Is propagation always guaranteed? When propagation occurs, what is the corresponding spreading speed? Is a logarithmic correction required to locate level sets accurately? Does the solution converge to a traveling‑wave profile in a suitable moving frame? We aim to address these questions, focusing on how the geometry of the hyperbolic space influences each of these phenomena.

The first difference with the Euclidean setting is encoded in $\lambda_{1}$,  the infimum of the $L^2$-spectrum of the operator $-\lh$, which  is given by
\begin{equation}\label{lambda1}
    \lambda_1:=\frac{(d-1)^2}{4}.
\end{equation}
It is well known (see, for instance, Section~\ref{section:analytic-preliminaries} and the paper ~\cite{Vazquez} for the necessary background) that this spectral gap due to the underlying geometry creates a stronger diffusion (sometimes known as \emph{ballistic} behavior). Non-linear variations can be found in \cite{MatteoHiperbolic,JLVPorusMediumHyperbolic}, for the porous medium equation.

This stronger diffusion affects the large-time behavior of solutions to the KPP problem in the hyperbolic space~\eqref{KPPproblem}, first studied in the paper~\cite{Matano} by Matano, Punzo and Tesei.  These authors proved that there is a competition between reaction and diffusion. If $f'(0)>\lambda_{1}$, propagation takes place, and if $f'(0)<\lambda_{1}$ and the initial datum has compact support, there is vanishing. Specifically, they established the following result.

\begin{theorem}[Theorem 3.2 in \cite{Matano}]\label{ThPropagation}
    Let $u$ be a solution to~\eqref{KPPproblem} with $0\le u_{0}\le 1$, $u_0\not \equiv 0$.
    \begin{itemize}
        \item[\textup{(i)}] If $f'(0)<\lambda_{1}$ and $u_0$ has compact support, then  $\displaystyle\lim_{t\to\infty}\sup_{x\in\Hd} u(x,t)=0$.
        \item[\textup{(ii)}] If $f'(0)>\lambda_{1}$, then, for any compact subset $K$ of $\Hd$, $\lim\limits_{t\to\infty}\inf\limits_{x\in K}u(x,t)=1$.
    \end{itemize}
\end{theorem}

This is the first significant difference with the KPP problem in the Euclidean space, where propagation always occurs. On the contrary, in the hyperbolic space, the extra diffusion due to the geometry may prevent propagation. 

When the extra diffusion is not able to prevent propagation, it will nevertheless slow the spreading speed from $c_0$ down to a new critical speed $c_*$ given by  
\begin{equation}\label{c0andc*}	
    c_*:=c_0-(d-1),
\end{equation} 
as proved also in~\cite{Matano}.  Notice that 
\begin{equation}\label{eq:equivalent.conditions}
    \operatorname{sign}(f'(0)-\lambda_1)=\operatorname{sign}(c_0-(d-1))
\end{equation}
so that $c_*>0$ if and only if $f'(0)>\lambda_1$.

\begin{theorem}[Theorem 3.6 in \cite{Matano}]\label{ThSpreadingSpeedElliptic}
    Let $f'(0)>\lambda_{1}$ and $u_0\not\equiv 0$ compactly supported. Then, the solution $u$ to~\eqref{KPPproblem} satisfies
    \begin{itemize}
        \item[\textup{(i)}] $\displaystyle\lim\limits_{t\to\infty}\sup_{\{x\in\Hd:\,\rho_\mathfrak{e}(x)>ct\}}u(x,t)=0$ for all $c>c_*$,
        \item[\textup{(ii)}] $\displaystyle\lim\limits_{t\to\infty}\inf_{\{x\in\Hd:\,0\leq\rho_\mathfrak{e}(x)<ct\}}u(x,t)=1$ for all $0<c<c_*$,
    \end{itemize}
    where $\rho_\mathfrak{e}: \Hd\to[0,\infty)$ is the geodesic distance to the origin (fixed at any point $p\in\mathbb{H}^d$).
\end{theorem}

The extra term $d-1$ in the expression of $c_*$, as compared to the Euclidean one, $c_0$, can be explained from the structure of the Laplacian in $\mathbb H^d$ (see, for instance, formula \eqref{Laplacian-introduction} in radial coordinates). The coefficient of the drift term in radial coordinates, differently from what happens in the Euclidean space, does not vanish at infinity, but is asymptotically constant. As a consequence, this drift from infinity must be overcome in order to have propagation.

\medskip

\noindent\textbf{Objectives. } Our first goal is to address the question of propagation \emph{versus} extinction, extending Theorem~\ref{ThPropagation} to include more general initial data and the critical case $f'(0)=\lambda_{1}$. This is done in Section~\ref{section:vs}. For further discussion on propagation versus extinction in the context of the KPP problem on other complete, non-compact Riemannian manifolds, we refer to~\cite{KPPinOtherManifold}.

Our second goal is to study, in the case of propagation, the convergence of solutions to~\eqref{KPPproblem} to a traveling wave in a moving coordinate system determined by the symmetries of the initial datum.  

As the Euclidean results~\eqref{u-to-TW-ellipticEuclidean} and~\eqref{u-to-TW-parabolicEuclidean} highlight, there is a relation between the isometry under which the initial datum is invariant and the coordinate system in which the solution converges to a traveling wave. Indeed, for~\eqref{u-to-TW-parabolicEuclidean}, in~\cite{tranlacionesRn} the authors require the initial datum to be trapped between two functions invariant under translations in one direction, that is, invariant under the same Euclidean parabolic isometry (translations). On the other hand,  to obtain~\eqref{u-to-TW-ellipticEuclidean}, one of the hypotheses in~\cite{RoquejofferNdimensional} is that the initial datum must be trapped between two radial functions, i.e., between two functions that are invariant under the same elliptic Euclidean isometry (rotations). The underlying idea is that  these are two settings where the problem can be essentially characterized in terms of solutions to a unidimensional PDE.

In the hyperbolic space there are also subgroups of  isometries that generate orbits of co-dimension one. These are the cohomogeneity one subgroups of isometries of $\mathbb H^d$, and can be characterized from the Iwasawa decomposition of the group of isometries of $\mathbb H^d$ (see \cite{Helgason2001} for a classical reference, and \cite[Section 2.2]{Gonzalez-Saez} for a friendly introduction).  Here we find an important  difference with the Euclidean space, since there are three types of such isometries, rather than two: elliptic ($\mathfrak e$), parabolic ($\mathfrak p$) and hyperbolic ($\mathfrak h$). In Section~\ref{section:geometry}, we provide some geometric background on $\Hd$  and its isometries in detail, for convenience of the non-expert reader. For a representation of the invariant sets under each isometry group in the Poincar\'e disk~$\mathbb{D}^2$, a model of~$\mathbb{H}^2$, see Figure~\ref{FirgureOrbits}.

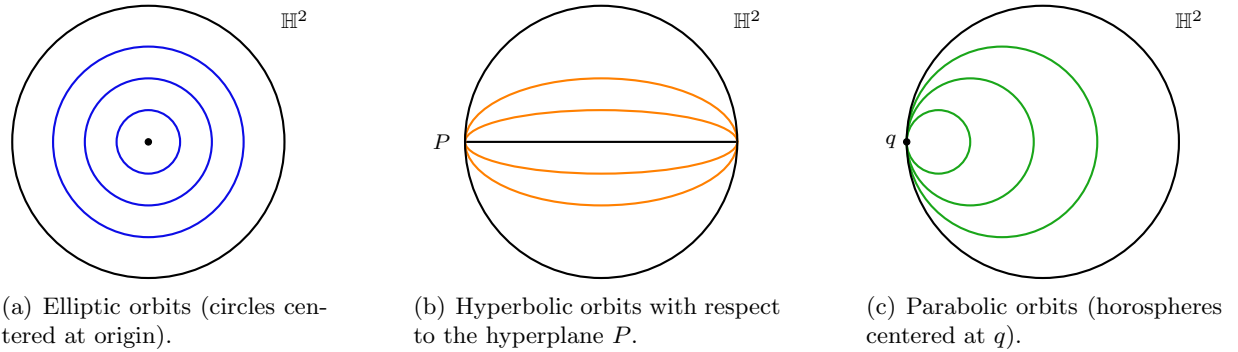
\begin{figure}[H]
	\begin{center}
		\subfigure[Elliptic orbits (circles centered at origin).]{
			\begin{tikzpicture}[scale=0.6]
				% Poincaré disc 2 dim
				\draw[thick] (0,0) circle (3cm);
				\node[right] at (2.7,2.7) {{\scriptsize$\mathbb{H}^{2}$}};
				\filldraw[black] (0,0) circle (2pt); %origen
				%Órbitas elípticas
				\draw[thick, darkblue] (0,0) circle (2.1cm);
				\draw[thick, darkblue] (0,0) circle (0.7cm);
				\draw[thick, darkblue] (0,0) circle (1.4cm);
			\end{tikzpicture}
		}
        \hskip1cm
		\subfigure[Hyperbolic orbits with respect to the hyperplane $P$.]{		\begin{tikzpicture}[scale=0.6]
				%Órbitas hiperbólicas
				\draw[thick](-3,0)--(3,0);
				\begin{scope}
					\clip (-3.1,0) rectangle (3.1,2.1); % Recorta la mitad inferior
					\draw[thick, orange] (0,0) ellipse (3cm and 0.7cm);
					\draw[thick, orange](0,0) ellipse (3cm and 1.4cm);
					%\draw[thick, orange](0,0) ellipse (3cm and 2.1cm);
				\end{scope}
				\begin{scope}
					\clip (-3.1,0) rectangle (3.1,-2.1); % Recorta la mitad superior
					\draw[thick, orange] (0,0) ellipse (3cm and 0.7cm);
					\draw[thick, orange](0,0) ellipse (3cm and 1.4cm);
					% \draw[thick, orange](0,0) ellipse (3cm and 2.1cm);
				\end{scope}
				% Poincaré disk 2 dim
				\draw[thick] (0,0) circle (3cm);
				\node[right] at (2.7,2.7) {{\scriptsize$\mathbb{H}^{2}$}};
				%Hyperplane P
				\draw[thick](-3,0)--(-3,0);
				\node[left] at (-3.1,0) {{\scriptsize$P$}};
			\end{tikzpicture}
		}
        \hskip1cm
		\subfigure[Parabolic orbits (horospheres centered at $q$).]{		\begin{tikzpicture}[scale=0.6]
				%Órbitas parabólicas
				\draw[thick, darkgreen] (-3+2.1,0) circle (2.1cm);
				\draw[thick, darkgreen] (-3+0.7,0) circle (0.7cm);
				\draw[thick, darkgreen] (-3+1.4,0) circle (1.4cm);
				% Poincaré disk 2 dim
				\draw[thick] (0,0) circle (3cm);
				\node[right] at (2.7,2.7) {{\scriptsize$\mathbb{H}^{2}$}};
				%q
				\filldraw[black] (-3,0) circle (2pt);
				\node[left] at (-3,0) {{\scriptsize$q$}};
			\end{tikzpicture}
		}
		\caption{\label{FirgureOrbits}Orbits in the Poincaré disk $\mathbb{D}^2$ for each isometry group.}
	\end{center}
\end{figure}

This rich structure has been explored in several PDEs, for instance \cite{BirindelliMazeoIsometriesExplicadas,BirindeMazzeoSaezCoordenadasHiperbolicas} for Allen-Cahn, \cite{Chossant-Faye-DibujitosIsometrias} for Swift-Hohenberg, and the aforementioned \cite{Matano} for (KPP).

Analytically, each of the three isometry subgroups generates a  coordinate system of the form
\begin{equation}\label{CoordenatesIntro}
    x=(\rho_i,\theta_i)\in\Hd,\quad i\in\{\mathfrak{e},\mathfrak{h},\mathfrak{p}\},
\end{equation}
where $\rho_i$ is unidimensional. This yields a splitting in the expression of the Laplace-Beltrami operator on $\mathbb H^d$; see \eqref{LaplaceBeltramiRhoTheta}. We will consider initial data that are symmetric, that is, that only depend on the variable $\rho_i$. 

An enlightening calculation is the  construction of (explicit) traveling waves in $\mathbb H^d$ in parabolic coordinates $x=(\rho_{\mathfrak p},\theta_{\mathfrak p})$; see~\cite{Matano}. These are known as \emph{horospheric traveling waves}. For this, note that equation \eqref{KPPproblem} reduces for solutions $u$ that only depend on the variable $\rho:=\rho_{\mathfrak p}$ to the unidimensional problem
\begin{equation}\label{problem-parabolic-intro}
     u_t= u_{\rho\rho}+(d-1) u_{\rho}+f(u),\quad \rho\in\mathbb R,\;\;t>0.
\end{equation}
If we look for solutions to~\eqref{problem-parabolic-intro} of the form $u(\rho,t)=U_c(\rho-ct)$, then the profile $U_c$ satisfies
\begin{equation*}
    U''+(c+d-1)U'+f(U)=0\quad\textrm{in }\mathbb{R}
\end{equation*}
which is precisely the ODE equation \eqref{ODE-tw} with the speed $c$ translated by $d-1$. Thus, the critical velocity $c_*$ in~ \eqref{c0andc*} is nothing but the minimal velocity for horospherical traveling waves.

This correction in the speed of traveling waves due to the underlying geometry is expected also for general solutions to~\eqref{problem-parabolic-intro}. With this in mind, we consider the translate $v(\rho,t):= u(\rho-(d-1)t,t)$ of~$u$. This function is a solution to the unidimensional Euclidean KPP equation
\begin{equation}\label{KPPEuclideanDiscussion}
    v_t=v_{\rho\rho}+f(v)\quad\text{in }\mathbb{R}\times\mathbb{R}_+.
\end{equation}
But we already know that solutions to this problem with compactly supported initial data spread for large times with speed $c_0$; see~\cite{AronsonWeinberger-PrimerArticulo(Viejo),AronsonWeinberger}. Therefore, the asymptotic spreading speed for $u$ should be $c_*$ if there is spreading.  In fact, we know more about the solutions to the unidimensional PDE~\eqref{KPPEuclideanDiscussion} when the initial data are compactly supported, namely
\begin{equation}\label{intro:log-behavior}    
    \lim\limits_{t\to\infty}\sup\limits_{\{\rho\geq 0\}}\big|v(\rho,t)-\Phi_{c_0}\big(\rho-(c_0t-\tfrac{3}{c_0}\log t+\beta)\big)\big|=0
\end{equation}
for some constant $\beta\in\mathbb{R}$ depending on the initial datum (recall formula ~\eqref{BramsonResult}  
and the references cited there). This convergence result translates into
\begin{equation}\label{eq:result.original.coordinates}
    \lim\limits_{t\to\infty}\sup\limits_{\{\rho\geq 0\}}\big|u(\rho,t)-\Phi_{c_0}\big(\rho-(c_*t-\tfrac{3}{c_0}\log t+\beta)\big)\big|=0
\end{equation}
in the original coordinates. Thus, after a logarithmic correction, we have convergence towards the horospherical traveling wave of minimal velocity.

In elliptic and hyperbolic coordinates one does not have traveling-wave solutions. However,~\eqref{problem-parabolic-intro} is still a good approximation to account for the traveling-wave behavior of solutions of~\eqref{KPPproblem} with symmetric initial data. 
Let us explain this idea for the case of elliptic isometries, $i=\mathfrak e$. Thus, suppose that $\rho:=\rho_{\mathfrak e}(x)$ is the geodesic distance between $x$ and a fixed point~$p\in\Hd$ (taken to be the origin), and take an initial datum  $0\leq u_0\leq 1$  compactly supported and radially symmetric with respect to this point. That is, there exists a unidimensional $\tilde{u}_0: [0,\infty)\to[0,1]$ such that  $u_0(x)=\tilde{u}_0(\rho(x))$ for all $x\in\Hd$. The Laplace-Beltrami operator in polar coordinates is given by
\begin{equation}\label{Laplacian-introduction}
    \Delta_{\Hd}u=u_{\rho\rho}+(d-1)\coth\rho\, u_{\rho}+\frac{1}{\sinh^2\rho}\Delta_{\mathbb{S}^{d-1}}u,\quad \rho> 0.
\end{equation}  
Since the  comparison principle holds (see Theorem~\ref{ComparsionPrincipleHd}) and the initial datum is invariant under rotations, the solution $u$ to~\eqref{KPPproblem} is also radial and the problem becomes unidimensional. That is, $u(x,t)=\tilde{u}(\rho(x),t)$ for all $(x,t)$ in $\Hd\times \mathbb{R}_{+}$, and  $\tilde{u}:[0,\infty)\times \mathbb{R}_+\to [0,1]$   satisfies
\begin{equation*}
    \begin{cases}
        \tilde{u}_t=\tilde{u}_{\rho\rho}+(d-1)\coth\rho\,\tilde{u}_\rho+f(\tilde{u}),&\quad\rho\geq0,\;\;t>0,\\
        \tilde{u}_\rho(0,t)=0,&\quad t>0,\\
        \tilde{u}(\rho,0)=\tilde{u}_0(\rho),&\quad\rho\geq 0.
    \end{cases}
\end{equation*}
Since $\coth{\rho}\to1$ as $\rho\to\infty$, we expect the behavior of $u$ to be governed by the unidimensional equation~\eqref{problem-parabolic-intro}. This suggests considering the translate $v(\rho,t):=\tilde u(\rho-(d-1)t,t)$, which satisfies
\begin{equation}\label{eq:translate.elliptic}
    v_t=v_{\rho\rho}+(d-1)h(\rho) v_\rho+f(v)\quad\text{in }\mathbb{R}\times\mathbb{R}_,
\end{equation}
with $h(\rho)=\coth\rho-1$. Notice that $h(\rho)\to0$ as $\rho\to\infty$. This is the same situation encountered when dealing with radial solutions $v$ to the KPP problem in the Euclidean space, that also satisfy~\eqref{eq:translate.elliptic}, but this time with $h(\rho)=1/\rho$. This suggests using the approach developed by Roquejoffre, Rossi and  Roussier-Michon in~\cite{RoquejofferNdimensional} for the latter problem to deal with ours. However, $h(\rho)=O(e^{-2\rho})$ in the \lq\lq  elliptic case'' of $\Hd$, a much faster decay to 0 than in the $d$-dimensional Euclidean case. Due to this important difference, the convection term will not give an extra contribution to the logarithmic correction, in contrast with the $d$-dimensional Euclidean situation, and we will get the behavior~\eqref{eq:result.original.coordinates} in the original coordinates, where now~$\rho:=\rho_{\mathfrak e}(x)$.

In the case of hyperbolic isometries, $i=\mathfrak h$, a similar reasoning yields again~\eqref{eq:translate.elliptic}, now with~$\rho:=\rho_{\mathfrak h}(x)$ and $h(\rho)=\tanh\rho$. The function $h$ has also an exponential decay to 0 in this case, and we get no extra logarithmic correction. 

The above information is the content of our main theorem.

\begin{theorem}\label{ThConvergenceTravellingWave}
 	Let $f'(0)>\lambda_{1}$.  Let $u$ be a solution to~\eqref{KPPproblem} with an initial datum $u_0:\Hd\to[0,1]$ given by a unidimensional $\rho_i$-function, $i\in\{\mathfrak{e},\mathfrak{h},\mathfrak{p}\}$,
    \begin{equation}\label{eq:symmetry.u0}
        u_{0}(x)=\widetilde{u}_0(\rho_{i}(x))\quad\textup{for all }x\in\Hd,
    \end{equation}
    with $\widetilde u_0$ nontrivial and compactly supported from the right,  that is
    \begin{equation}\label{eq:supported-from-right}
        \widetilde{u}_0(\rho_i)=0 \quad\text{for all }\rho_i\ge \rho_*\textup{ for some }\rho^*\in I_i.
    \end{equation}
    Then, there is a constant $\beta\in\mathbb{R}$ depending on the initial data such that 
    \begin{equation}\label{u-to-tw-equation-main-theorem}          
        \lim\limits_{t\to\infty}\sup\limits_{\{x\in\Hd:\,\rho_i(x)\geq 0\}}\big|u(x,t)-\Phi_{c_0}\big(\rho_i(x)-(c_*t-\tfrac{3}{c_0}\log t+\beta)\big)\big|=0.
 	\end{equation}
\end{theorem}

\begin{remark}
    We have an analogous result for $\rho_i(x)\leq 0$ in the hyperbolic and parabolic cases if $\widetilde{u}_0$ is compactly supported from the left.
\end{remark}

Notice that, in contrast with the Euclidean case, the linear term of the moving frame in $\Hd$ varies with the dimension, and the constant in the logarithmic correction is independent both of the chosen isometry and of the dimension. This fact, together with the scheme of the proof of the main theorem, will be explained in Section~\ref{section:strategy}.

\medskip

\noindent{\bf Future perspectives.} Our results for the case of initial data that are invariant under isometries  yield, by comparison, information about more general solutions if the initial data can be trapped between two functions that are invariant under the same isometry to which our results apply. 

\begin{corollary}\label{prop:final}
    Let $f'(0)>\lambda_{1}$. Consider a measurable function $u_0:\Hd\to[0,1]$ satisfying
    \begin{equation*}
        0\le\underline{u}_0(\rho_{i}(x))\leq u_0(x)\leq {\overline{u}}_0(\rho_{i}(x)),\quad x\in\Hd,
    \end{equation*}
    for some compactly supported unidimensional $\rho_i$-functions ${\underline{u}}_0,\, \overline{u}_0:I_i\to\mathbb{R}$, $i\in\{\mathfrak{e},\mathfrak{h},\mathfrak{p}\}$. Then, there exist $\beta_1\leq\beta_2$ such that, for all $\varepsilon>0$ there is a time $t_{\varepsilon}>0$ for which
    \begin{equation*}\label{ConvergenceTwNoInvariant}
        -\varepsilon+\Phi_{c_0}\big(\rho_i(x)-[c_*t-\tfrac{3}{c_0}\log t+\beta_1]\big)\leq u(x,t)\leq \Phi_{c_0}\big(\rho_i(x)-[c_*t-\tfrac{3}{c_0}\log t+\beta_2]\big)+\varepsilon
    \end{equation*}
    for all $x\in\Hd$ with $\rho_i(x)\geq 0$ and $t\geq t_{\varepsilon}$.
\end{corollary}

In the elliptic case we anticipate a sharper result: if the initial data are nontrivial and compactly supported, we conjecture the existence of a bounded Lipschitz function $\tilde{\beta}:\mathbb{S}^{d-1}\to\mathbb{R}$ such that
\begin{equation*}
	\lim\limits_{t\to\infty}\sup\limits_{\{x\in\Hd:\, \rho_{\mathfrak{e}}(x)\geq 0\}}\big|u(x,t)-\Phi_{c_0}\big(\rho_{\mathfrak{e}}(x)-[c_*t-\tfrac{3}{c_0}\log t+\tilde{\beta}(\theta_{\mathfrak{e}}(x))]\big)\big|=0.
\end{equation*}
In view of the counterexample constructed in~\cite{Rossi-2017} for the Euclidean setting, one should not expect $\tilde\beta$ to be constant.  

We predict a more intricate scenario if the initial datum \lq\lq detects'' infinity; see~\cite{tranlacionesRn} for the Euclidean case.

A second goal for the future would be to give the intermediate asymptotics for solutions when they vanish; that is, to determine their rate of decay to 0, and their possible convergence towards a profile once the solution is rescaled to annihilate the decay. If the initial data are compactly supported, one expects vanishing solutions to behave like a multiple of the heat kernel. One main difficulty would be to determine the right multiple, which should be given by the asymptotic mass of the solution, in case it is finite. We hope to return to this problem elsewhere.

\medskip

\noindent{\bf Organization of the paper.} In Section~\ref{section:geometry}, we present the necessary background on $\Hd$, its isometries and the three coordinate systems that we will chose in our analysis.  Section~\ref{section:analytic-preliminaries} gathers known preliminary results concerning the heat kernel and the existence and uniqueness of solutions to~\eqref{KPPproblem}. In Section~\ref{section:vs}, we extend the discussion of propagation versus extinction. We organize the proof of our main theorem, establishing the convergence of the solution in a specific moving coordinate system to the traveling wave of minimal speed, in two sections: in the first one, Section~\ref{section:strategy}, we describe the strategy of the proof, leaving the technical details for Section~\ref{section:details}. 

%%%%%%%%%%%%%%%%%%%%%%%%%%%%%%%%%%%%%%%%%%%%%%%%%%%%%%%%%%%%%%
%%%%%%%%%%%%%%%%%%%%%%%%%%%%%%%%%%%%%%%%%%%%%%%%%%%%%%%%%%%%%%
\section{Geometric background}\label{section:geometry}
\setcounter{equation}{0}

For the convenience of the reader, this section collects some background on the Laplace-Beltrami operator on the hyperbolic space  which we believe might be helpful for the PDE community. A classical reference is~\cite{Ratcliffe}, while a quick summary is found in \cite{BirindelliMazeoIsometriesExplicadas}.

\medskip

\noindent{\bf Laplace-Beltrami operator.}  Let $(\mathcal{M},g_{\mathcal{M}})$ be a $d$-dimensional, connected,  $C^{2}$-Riemannian manifold. The gradient $\nabla_{\mathcal{M}}$ of a $C^1$ real-valued function $f$ defined on the manifold is a vector field over $\mathcal{M}$ (with values in the tangent bundle) that evaluated in $x\in\mathcal{M}$ satisfies
\begin{equation*}
    \langle \nabla_{\mathcal{M}}f,\xi\rangle_{\mathcal{M}}=(f\circ w)'(0)
\end{equation*}
for all vectors $\xi$ in the tangent space of $\mathcal{M}$ at $x$, $T_x\mathcal{M}$, and for any path $w: \mathbb{R}\to \mathcal{M}$ such that $w(0)=x$ and $w'(0)=\xi$. Here, $\langle\cdot,\cdot\rangle_{\mathcal{M}}$ is the inner product in $T_x\mathcal{M}$ given by the metric $g_{\mathcal{M}}$. 

Let ${\rm d}\mu_{\mathcal{M}}$ be the volume element in $(\mathcal{M},g_{\mathcal{M}})$. The divergence operator of $\mathcal{M}$, $\operatorname{div}_{\mathcal{M}}$, is defined by duality as the unique linear map acting on vector fields of $\mathcal{M}$ such that 
\begin{equation*}
    \int_{\mathcal{M}}(\operatorname{div}_{\mathcal{M}} X)u\,{\rm d}\mu_{\mathcal{M}}= -\int_{\mathcal{M}}\langle X, \nabla_{\mathcal{M}}u\rangle_{\mathcal{M}}\,{\rm d}\mu_{\mathcal{M}}
\end{equation*}
for all vector fields  $X$ on $\mathcal{M}$ and $u\in C^{\infty}_0(\mathcal{M})$.

The Laplace-Beltrami operator in $\mathcal{M}$ is given by 
\begin{equation*}
    \Delta_{\mathcal{M}}u:=\operatorname{div}_{\mathcal{M}}\big(\nabla_{\mathcal{M}}u\big)\quad\textrm{for all }u:\mathcal{M}\to\mathbb{R}\textrm{ smooth}.
\end{equation*}
In terms of the metric $g_{\mathcal{M}}$, 
\begin{equation*}
	\Delta_\mathcal{M}=\frac{1}{\sqrt{\operatorname{det} g_{\mathcal{M}}}}\sum^{d}_{j=1}\sum^d_{l=1}\partial_j\big(\sqrt{\operatorname{det}g_{\mathcal{M}}}\left(g_{\mathcal{M}}\right)^{jl}\partial_l\big).
\end{equation*}
Here $(g_{\mathcal M})^{jl}$ is the inverse matrix of $(g_{\mathcal M})_{jl}.$

Assume now that $(\mathcal{M}, g_\mathcal{M})$ can be written as a \emph{warped product}. That is, let $(\mathcal{N},g_\mathcal{N})$ be  a $(d-1)$-submanifold embedded in $(\mathcal{M}, g_\mathcal{M})$ such that $\mathcal{M}=I\times\mathcal{N}$ and 
\[
    g_\mathcal{M}(\rho,\theta)={\rm d}\rho^{2}+\psi^{2}(\rho)g_{\mathcal{N}}(\theta)
\]
for all $\rho\in I$ and $\theta\in\mathcal{N}$, where $I\subset\mathbb{R}$ is an interval and $\psi:I\to[0,\infty)$. The Laplace-Beltrami operator in coordinates $\mathcal{M}\ni x=(\rho,\theta)\in I\times\mathcal{N}$  is given by
\begin{equation}\label{Laplace-WarpedProduct}
    \Delta_{\mathcal{M}}u=u_{\rho\rho}+(d-1)\frac{\psi'}{\psi}u_{\rho}+\frac{1}{\psi^2}\Delta_{\mathcal{N}}u.
\end{equation}
We remark that for each $\rho_0\in I$, the quotient $\psi'(\rho_0)/\psi(\rho_0)$ is the the mean curvature of $(\mathcal{N}(\rho_0), g_{\mathcal{N}})$ as an embedded submanifold of co-dimension one in $(\mathcal{M},g_{\mathcal{M}})$, where 
\[
    \mathcal{N}(\rho_0):=\{ (\rho,\theta)\in I\times \mathcal{N}:\,\rho=\rho_0\}.
\]

\medskip

\noindent{\bf The hyperbolic space $\big(\Hd,g\big)$} is the unique simply connected, complete $d$-dimensional Riemannian manifold with sectional curvature identically equal to $-1$.

Let $d_{\mathbb H^d}:\Hd\times\Hd\to[0,\infty)$ be the  hyperbolic  geodesic distance. For any two different points in $\Hd$ there is exactly one geodesic passing through them. Moreover, every geodesic is defined on the whole real line. A subset $A\subset\Hd$ is a totally geodesic hyperbolic subspace if, given any two different points in $A$, the entire geodesic passing through them is contained in $A$. A hyperplane $P$ is a totally geodesic subset of $\Hd$ of co-dimension one; in particular, $P$ is isomorphic to $\mathbb{H}^{d-1}$.  

Let us recall the definition of the reflection across a hyperplane $P\subset\Hd$. Consider $x\in\Hd\setminus P$, then there is a unique $y\in P$ such that $d_{\mathbb H^d}(x,y)=\min_{z\in P} d_{\mathbb H^d}(x,z)$. Let $\gamma: \mathbb{R}\to\Hd$ be the unique geodesic passing through $x$ and $y$  parametrized so that $\gamma(0)=y$ and $\gamma(1)=x$. Then, the reflection with respect to $P$ is the function $\mathcal{R}_P$ such that
\begin{equation*}
    \mathcal{R}_P:\Hd\to\Hd,\quad\mathcal{R}_P(x):=\gamma(-1).
\end{equation*}

The geodesic compactification $\overline{\Hd}=\Hd\cup\partial\Hd$ is obtained by fixing a point $x\in\Hd$ and adding a point at  both ends of each geodesic passing through $x$. The points of $\partial\Hd$ can be regarded as the points at infinity of $\Hd$. A horosphere centered at $q\in\partial\Hd$ is a closed hypersurface of $\Hd$ orthogonal to all the geodesics with endpoint $q$. Each horosphere is isomorphic to $\mathbb{R}^{d-1}$.

There are several isomorphic models describing $\Hd$. In this paper, we consider two of them: Poincaré's disk and Poincaré's half-space models.

\noindent{\it Poincaré's disk model} $\dhh$ is the unit ball $B_1(0):=\{x\in \mathbb{R}^d: | x|<1\}$, where $|\cdot|$ is the Euclidean metric, with the Riemannian metric $g$ given by
\begin{equation*}
	g_{jl}=\frac{4}{(1-|x|^2)^2}\delta_{jl}, \quad  x\in B_{1}(0),\;\;j,l\in\{1,...,d\},
\end{equation*}
where $\delta_{jl}=1$ if $j=l$ and $\delta_{jl}=0$ if $j\not =l$.
 
In this model:
\begin{itemize}
	\item geodesics are either semicircles orthogonal to $\partial B_1(0)$, or diameters of $B_1(0)$;
	\item a subset $P\subset\Hd$ is a hyperplane if and only if  $P$ is the intersection of $B_1(0)$ with a $(d-1)$-dimensional  Euclidean sphere orthogonal to $\partial B_1(0)$ or  it is  the intersection of $B_1(0)$ with a Euclidean hyperplane $\pi\subset\mathbb{R}^d$ such that $0\in\pi$;
	\item a horosphere with center $q\in\partial B_1(0)$ is a $(d-1)$-dimensional Euclidean sphere tangent to $\partial B_1(0)$ in $q$ and contained in $B_1(0)$;
	\item the Laplace-Beltrami operator is
	\begin{equation*}
		\Delta_{\Hd}u=\frac{1}{4}(1-|x|^2)^2\sum_{j=1}^{d}\frac{\partial^2 u}{\partial x_{j}^2}+\frac{d-2}{2}(1-|x|^2)\sum_{j=1}^{d}x_j\frac{\partial u}{\partial x_{j}}.
	\end{equation*}
\end{itemize}
For a visual representation of some geodesics, hyperplanes and horospheres in $\mathbb{D}^3$, see Figure~\ref{FigurePoincaredisc}.
\begin{figure}[H]
	\begin{center}
		\subfigure[Geodesics.]{
			\begin{tikzpicture}[scale=0.6]
			\clip (-3.05,-3.05) rectangle (3.05,3.05) ;% que rectángulo pintamos			
			% Draw the vertical diameter
			\draw[thick] (0,-3) -- (0,3);
			% Draw the horizontal dashed circle (equator)
			\draw[dashed] (0,0) ellipse (3cm and 0.6cm);
			% Draw a diagonal line
			\draw[thick] (1.8,2) -- (-1.8,-2);
			%\draw[thick] (-1.5,2.598) -- (-2.598,1.5);	
			% Add points
			\filldraw[black] (0,3) circle (2pt); % top point
			\filldraw[black] (0,-3) circle (2pt); % bottom point
			\filldraw[black] (0,0) circle (2pt); % center point
			\filldraw[black] (1.8,2) circle (2pt); % left point
			\filldraw[black] (-1.8,-2) circle (2pt); % right point
			
			%geodesica circular superior
			\draw[thick] (-2.049,2.049) circle (0.776cm);
			\fill[white] (-1.5-0.2,2.598+0.2) circle (0.25cm);%poniendoblanco
			\filldraw[black] (-1.5,2.598) circle (2pt);
			\fill[white] (-2.598-0.2,1.5+0.2) circle (0.25cm);%poniendoblanco
			\filldraw[black] (-2.598,1.5) circle (2pt);
			
			\fill[white ] (-6*0.7181,0) -- (0,6*.7181) -- (-6*0.7181-4,0) -- (0,6*.7181+4)  -- cycle;%poniendo blanco
			%geodesica circular inferior
			\draw[thick] (1.8365,-1.8365) circle (1.5cm);
			\fill[white] (0.776+0.2,-2.897-0.2) circle (0.2cm);%poniendoblanco
			\filldraw[black] (0.776,-2.897) circle (2pt);
			\fill[white] (2.897+0.2,-0.776-0.2)circle (0.20cm);%poniendoblanco
			\filldraw[black] (2.897,-0.776) circle (2pt);
			\fill[white] (3,-3)circle (0.5cm);%poniendoblanco
			%-------
			\node[right] at (3*0.7181,3*0.7181) {\tiny$ \mathbb{H}^{3}$};
			% Draw the outer circle (the sphere)
			\draw[thick] (0,0) circle (3cm);
		\end{tikzpicture}
		}
        \hskip1cm
		\subfigure[Hyperplanes.]{
				\begin{tikzpicture}[scale=0.6]
			\clip (-3.05,-3.05) rectangle (3.05,3.05) ;% que rectángulo pintamos
			% Draw the horizontal dashed circle (equator)
			\draw[thick, fill=darkred,  fill opacity=0.6] (0,0) ellipse (3cm and 0.6cm);
			%Elipse rotada
			\draw[thick, rotate=45,  fill=darkred,  fill opacity=0.6] (0,0) ellipse (3cm and 0.6cm);
			% Add points
			\filldraw[black] (0,0) circle (2pt); % center point
			% plano geodesico circular
			\draw[thick,  fill=darkred,  fill opacity=0.6] (-2.049,2.049) circle (0.776cm);
			%\filldraw[black] (-2.049,2.049) circle (2pt);
			\draw[thick, rotate=45] (0,2.95) ellipse (0.776cm and 0.3cm);
			%(aquí también se giran los ejes de referencia)
			%------
			\node[right] at (3*0.7181,3*0.7181) {\tiny$ \mathbb{H}^{3}$};
			
			%poniendo blanco
			\fill[white ] (-6*0.7181,-0.05) -- (0.1,6*.7181) -- (-6*0.7181-4,0) -- (0,6*.7181+4)  -- cycle;
			\fill[white] (-1.5-0.2,2.598+0.2) circle (0.27cm);%poniendoblanco
			\fill[white] (-1.5-0.3,2.598+0.17) circle (0.31cm);%poniendoblanco
			%\filldraw[black] (-1.5,2.598) circle (2pt);
			\fill[white] (-2.598-0.2,1.5+0.2) circle (0.27cm);%poniendoblanco
			\fill[white] (-2.598-0.15,1.5+0.35) circle (0.32cm);%poniendoblanco
			% Draw the outer circle (the sphere)
			\draw[thick] (0,0) circle (3cm);
		\end{tikzpicture}
		}
        \hskip1cm
			\subfigure[Horospheres centered at~$q$.]{
			\begin{tikzpicture}[scale=0.6]
			% Draw the outer circle (the sphere)
			\draw[thick] (0,0) circle (3cm);
			
			% Draw the horizontal dashed circle (equator)
			\draw[dashed] (0,0) ellipse (3cm and 0.6cm);
			%Horoesfera grande
			\draw[ fill=darkgreen,  fill opacity=0.7]  (-1.7,0) circle (1.3cm);
			%	\draw[thick] (-1.7,0) circle (1.3cm); %sin colorear
			\draw[dashed] (-1.7,0) ellipse (1.3cm and 0.26cm);
			
			%Horoesfera pequeña
			\draw[ fill=darkgreen,  fill opacity=0.7] (-2.1,0) circle (0.9cm);
			%\draw[thick] (-2.1,0) circle (0.9cm); %sin colorear
			\draw[dashed] (-2.1,0) ellipse (0.9cm and 0.18cm);
			% Add labels
			\node[left] at (-3.01,0) {$q$};
			\node[right] at (3*0.7181,3*0.7181) {\tiny$ \mathbb{H}^{3}$};
			
			% Add points
			\filldraw[black] (0,0) circle (2pt); % center point
			\filldraw[black] (-3,0) circle (2pt); % point q
		\end{tikzpicture}
		}
		\caption{\label{FigurePoincaredisc}Poincaré disk model $\mathbb{D}^3$.}
	\end{center}
\end{figure}
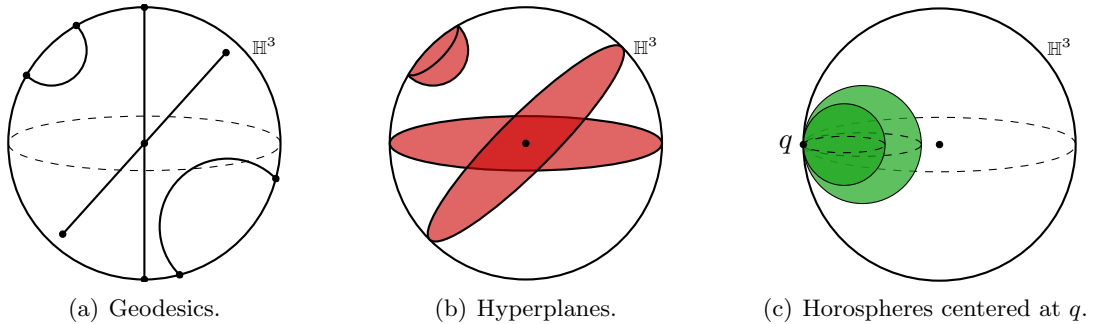

\noindent{\it The Poincaré half-space model} $\uh$ is the upper half-space $\{x\in\Rd:\,x_d>0\}$ endowed with the Riemannian metric
\begin{equation*}
	g_{jl}=\frac{1}{x_d^2}\delta_{jl}, \quad  x\in\Rd:\,x_d>0,\;\; j,l\in\{1,...,d\}.
\end{equation*}
Here:
\begin{itemize}
	\item geodesics are either upper half-lines or upper semicircles orthogonal to $\{x_{d}=0\}$;
	\item a subset $P\subset\Hd$ is a hyperplane if and only if  $P$ is the intersection of  $\{x\in\Rd:\,x_d>0\}$ with a Euclidean hyperplane  orthogonal to $\{x_d=0\}$ or it is the intersection of $\{x\in\Rd:\,x_d>0\}$ with a $d$-dimensional  Euclidean sphere with center in $\{x_d=0\}$;
	\item  for a given $q\in\partial\Hd$, we can identify $q$ with the point $-\infty$. Then, in $\uh$ any horosphere with center $q$ is given by $\{ x_d=\delta\}$ with $\delta>0$;
	\item the Laplace-Beltrami operator is
\begin{equation*}	\Delta_{\Hd}u=x_d^2\sum^{d}_{j=1}\frac{\partial^2u}{\partial x_j^2}+(2-d)x_d\frac{\partial u}{\partial x_d}.
\end{equation*}
\end{itemize}
For a visual representation of some geodesics, hyperplanes and horospheres in $\mathbb{U}^3$, see Figure~\ref{FigurePoincareHalf}.
\begin{figure}[H]
	\begin{center}
		\subfigure[Geodesics.]{
			\begin{tikzpicture}[scale=0.4]
			%labels
			\node[right] at (4,6)  {\tiny$\mathbb{H}^3$};
			
			\node[left] at (-0.1,5.7)  {\tiny$x_3$};
			\node[right] at (3.2,-1.3)  {\tiny$\{x_3=0\}$};

			%Geodesicas semicirculos
			\draw[thick] (2.2,0.6) circle (1.5cm);%GeodesicaCircular1
			\fill[white ] (2.2+1.5+0.1,0.6)  -- (2.2-1.5-0.1,0.6) -- (2.2-1.5,-1.8+0.6)  --	(2.2+1.5,-1.8+0.6)   --cycle; %pintar en blanco
			\filldraw[black] (2.2-1.5,0.6) circle (2pt); %end
			\filldraw[black] (2.2+1.5,0.6) circle (2pt);
			(1cm);
			%Geodesicas verticales
			\draw[thick] (-1.8,0.4) -- (-1.8,6+0.4);
			\filldraw[black] (-1.8,0.4) circle (2pt);
			%\draw[thick] (4.4,0.9) -- (4.4,6+0.9);
			%	\filldraw[black] (4.4,0.9) circle (2pt);
			\filldraw[black] (0,0) circle (2pt); % center point
			\draw[thick] (0,0) -- (0,6);
			%base
			\draw (-5,-1) -- (3,-1) -- (5,1) -- (-3,1) -- cycle; %base
		\end{tikzpicture}
		}
        \hskip0.6cm
		\subfigure[Hyperplanes.]{
			\begin{tikzpicture}[scale=0.4]
			\draw[thick] (0,0) -- (0,6);% eje
			\node[left] at (-0.2,5.8)  {\tiny$x_3$};
			\node[right] at (3.2,-1.3)  {\tiny$\{x_3=0\}$};
			\node[right] at (4,6)  {\tiny$\mathbb{H}^3$};
			
			%Planos geodésicos circulares
			\draw[ fill=darkred,  fill opacity=0.6]  (2.2,0.6) circle (1.5cm);
			% \draw[dashed, rotate=45] (1.4,-1.47) ellipse (1.4cm and 0.2cm);
			
			\fill[white ] (2.2+1.5+0.1,0.6)  -- (2.2-1.5-0.1,0.6) -- (2.2-1.5,-1.8+0.6)  --	(2.2+1.5,-1.8+0.6)   --cycle; %pintar en blanco
			%para recortar la elipse
			\begin{scope}
				\clip (0,0) rectangle (4.5,0.6); % Recorta la mitad superior
				\fill[darkred, opacity=0.6] (2.2,0.6) ellipse (1.5cm and 0.15cm); % Elipse rellena
			\end{scope}
			\draw[dashed] (2.2,0.6) ellipse (1.5cm and 0.15cm);
			
			% Planos geodesicos
			\draw (-1,-1) -- (1,1) -- (1,7) -- (-1,5) -- cycle;
			\fill[thick, fill=darkred,  fill opacity=0.6] (-1,-1) -- (1,1) -- (1,7) -- (-1,5) -- cycle;
			\draw (-2.5-1,-1) -- (-2.5+1,1) -- (-2.5+1,7) -- (-1-2.5,5) -- cycle;
			\fill[thick, fill=darkred,  fill opacity=0.6]
			(-2.5-1,-1) -- (-2.5+1,1) -- (-2.5+1,7) -- (-1-2.5,5) -- cycle;
			%Suelo
			\draw[thick](-5,-1) -- (3,-1);
			\draw[thick] (3,-1)--(5,1);
			\draw[thick] (3.65,1)--(5,1);
			\draw[dashed] (-1,1)--(3.65,1);
			\draw[thick] (-2.5+1,1)--(-1,1);
			\draw[dashed] (-3,1)--(-2.5+1,1);
			\draw[thick] (-5,-1)--(-2.5-1,0.5);
			\draw[dashed] (-3.5,0.5)--(-3,1);
		\end{tikzpicture}
		}
        \hskip0.6cm
		\subfigure[Horospheres centered at~$q$.]{
			\begin{tikzpicture}[scale=0.4]
					\node[right] at (4,6.7)  {\tiny$\mathbb{H}^3$};
			% Dibujar el romboide
			%\draw (-5,-1) -- (3,-1) -- (5,1) -- (-3,1) -- cycle; %base
			\draw[thick] (-5,-1)--(3,-1);%base
			\draw[thick] (3,-1)--(5,1);%base
			\draw[thick] (5,1)--(3.5,1);
			\draw[dashed](3.5,1)--(-3,1);%base
			\draw[dashed](-3,1)--(-3.5,0.5);%base
			\draw[thick](-3.5,0.5)--(-5,-1);%base
			\draw[thick](0,6)--(0,4);
			\draw[dashed] (0,4) -- (0,3);
			\draw[thick](0,3) -- (0,1.8);
			\draw[dashed] (0,4-1.8) -- (0,0.5);
			\draw[thick](0,0.5) -- (0,0);
			
			\node[left] at (-0.1,5.7)  {\tiny$x_3$};
			\node[centered] at (-0,6.7)  {\tiny$q$};
			\node[right] at (3.2,-1.3)  {\tiny$\{x_3=0\}$};
			% Horoesferas
			\draw (-5,-1+4) -- (3,-1+4) -- (5,1+4) -- (-3,1+4) -- cycle;
			\fill[thick, fill=darkgreen,  fill opacity=0.6] (-5,-1+4) -- (3,-1+4) -- (5,1+4) -- (-3,1+4) -- cycle;
			\draw (-5,-1+1.5) -- (3,-1+1.5) -- (5,1+1.5) -- (-3,1+1.5) -- cycle;
			\fill[thick, fill=darkgreen,  fill opacity=0.6] (-5,-1+1.5) -- (3,-1+1.5) -- (5,1+1.5) -- (-3,1+1.5) -- cycle;
		\end{tikzpicture}
		}
		\caption{\label{FigurePoincareHalf}Poincaré half-space model $\mathbb{U}^3$.}
	\end{center}
\end{figure}
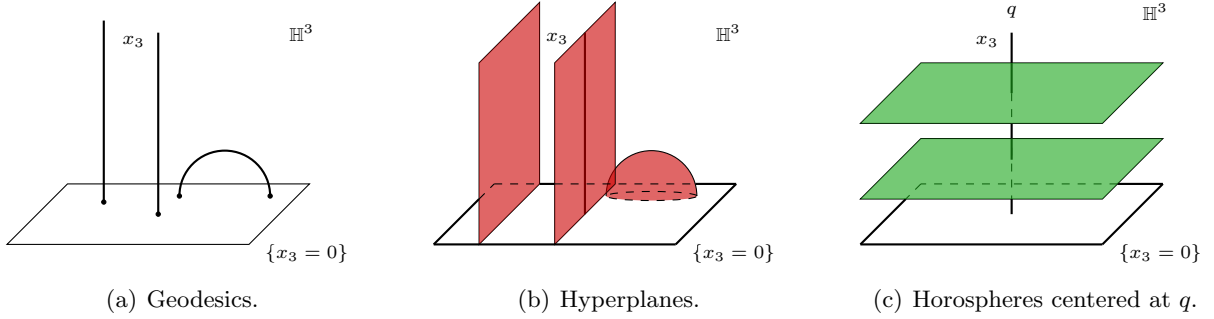

\noindent{\bf Isometries of the hyperbolic space.} There are three subgroups of $\operatorname{Isom}(\Hd)$ which have orbits in $\Hd$ of co-dimension one: $H_\mathfrak{e}$, $H_\mathfrak{h}$ and $H_\mathfrak{p}$. The subscripts \lq\lq$\mathfrak{e}$\rq\rq, \lq\lq$\mathfrak{h}$\rq\rq and \lq\lq$\mathfrak{p}$\rq\rq stand for elliptic, parabolic and hyperbolic, respectively.

\noindent\underline{\it $H_\mathfrak{e}$, elliptic isometries.} The action of $H_\mathfrak{e}$ is by rotations around a fixed point $p\in\Hd$, that we identify with the origin or by reflections across a totally geodesic hyperplane $P$ such that $p\in P$.  Then, the invariant sets are  hyperbolic spheres centered at the origin.

\noindent\underline{\it $H_\mathfrak{h}$, hyperbolic isometries.} To describe this subset of isometries we select  a totally geodesic hyperplane $P\subset\Hd$.  Since $P$ is isomorphic to $\mathbb{H}^{d-1}$, any isometry $h$ in $P$ can be uniquely extended to an isometry of the full space $\Hd$, up to a reflection across $P$. The subgroup $H_\mathfrak{h}$ is composed of all isometries constructed in this manner (without reflections).  If we identify $P$ with the horizontal plane passing through the origin in the Poincaré disk model, the \textit{half-rugby ball} sets (defined  precisely in~\eqref{Rugbyball}) remain invariant under the action of such isometries and are therefore isometric to $\mathbb{H}^{d-1}$.

\noindent\underline{\it $H_\mathfrak{p}$, parabolic isometries.}   The parabolic isometries are   Euclidean   motions on each horosphere centered at a point $q\in\partial\Hd$. Here, the invariant sets are the horospheres. Without loss of generality we can identify $q$ with $-\infty$.

We notice here that each particular isometry (elliptic, hyperbolic, parabolic) involves a choice (the origin point, the hyperplane $P$, the point $q$, respectively), which will be implicitly assumed in our discussion.

Associated with those subgroups of $\operatorname{Isom}(\Hd)$, we introduce three different coordinate systems in which $(\Hd, g)$ is a warped product, which yields a simplified formula for the Laplace-Beltrami operator~$\Delta_{\Hd}$ as given by~\eqref{Laplace-WarpedProduct}.

\noindent $\bullet$ \underline{$\mathfrak{e}$-coordinates (radial coordinates):}  \ Let  $ \dhh=\{x\in\Rd, |x|\leq 1\}$, with $|x|:=\big(\sum_{j=1}^dx_j^{2}\big)^{1/2}$. The $\mathfrak{e}$-coordinates of $x\in\dhh$ are given by
\begin{equation*}
    \rho_\mathfrak{e}(x):=\log\left(\frac{1+|x|}{1-|x|}\right),\qquad\theta_\mathfrak{e}(x):=\frac{x}{|x|},\quad\textrm{so that }(\rho_\mathfrak{e},
\theta_\mathfrak{e})\in\mathbb{R}_+\times\mathbb{S}^{d-1}.
\end{equation*}
Thus, the Poincaré disk $\dhh$ can be regarded as $\mathbb{R}_+\times\mathbb{S}^{d-1}$; see Figure~\ref{FigureElliptic}. Arrowed lines there represent the  geodesics parametrized by $\rho_{\mathfrak e}\geq0$. We color in blue some  hyperbolic spheres centered at the origin.
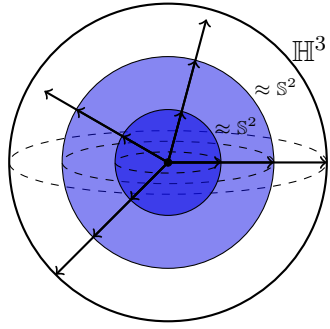
\begin{figure}[H]
    \begin{center}
        \begin{tikzpicture}[scale=0.7]
            \clip (-3.05,-3.05) rectangle (3.05,3.05) ;% que retángulo pintamos
            % Draw the outer circle (the sphere)
            \draw[thick] (0,0) circle (3cm); %sin colorear
            \draw[dashed] (0,0) ellipse (3cm and 0.6cm);
            %esfera grande
            \draw[ fill=darkblue,  fill opacity=0.5](0,0) circle (2cm);
            %	\draw[thick] (0,0) circle (2cm); %sin colorear
            \draw[dashed] (0,0) ellipse (2cm and 0.4cm);
            %esfera pequeña
            \draw[ fill=darkblue,  fill opacity=0.6] (0,0) circle (1cm);
            %\draw[thick] (0,0) circle (1cm); % sin colorear
            \draw[dashed] (0,0) ellipse (1cm and 0.2cm);
            % Add points
            \filldraw[black] (0,0) circle (2pt); % center point
            %Las geodesicas radiales% Dibuja el segmento con cabeza de flecha
            \draw[thick,->] (0,0) -- (3,0);
            \draw[thick,->] (0,0) -- (1,0);
            \draw[thick,->] (0,0) -- (2,0);
            %------
            \draw[thick,->] (0,0) -- (-2.338,1.35);
            \draw[thick,->] (0,0) -- (-2.598/3,1.5/3);
            \draw[thick,->] (0,0) -- (-2*2.598/3,2*1.5/3);
            %--------------
            \draw[thick,->] (0,0) -- (0.2588,0.9659);
            \draw[thick,->] (0,0) -- (2*0.2588,2*0.9659);
            \draw[thick,->] (0,0) -- (2.8*0.2588,2.8*0.9659);
            %-----------------
            \draw[thick,->] (0,0) -- (-0.7071,-0.7071);
            \draw[thick,->] (0,0) -- (-2*0.7071,-2*0.7071);
            \draw[thick,->] (0,0) -- (-3*0.7071,-3*0.7071);
            % Nombres
            \node[right] at (0.685,0.7) {{\tiny$\approx \mathbb{S}^{2}$}};
            \node[right] at (2*0.7181,2*0.7181) {{\tiny$\approx\mathbb{S}^{2}$}};
            \node[right] at (3*0.7181,3*0.7181) {$ \mathbb{H}^{3}$};
        \end{tikzpicture}
        \end{center}
        \caption{$\mathbb{D}^3$ as a warped product induced by elliptic coordinates.}
        \label{FigureElliptic}
\end{figure}

\noindent$\bullet$ \underline{$\mathfrak{h}$-coordinates:} \ Let $\dhh=\{x\in \Rd, |x|\leq 1\}$, where $|x|:=\big(\sum_{j=1}^dx_j^{2}\big)^{1/2}$. The $\mathfrak{h}$-coordinates  are given by
\begin{equation*}
    \rho_\mathfrak{h}(x)=\operatorname{arsinh}\!\left(\frac{2x_n}{1-|x|^2}\right),\qquad 
    \theta_{\mathfrak{h}}(x)=\frac{2x'}{1+|x|^2+\sqrt{(1-|x|^2)^2+4x_n^2}}.
\end{equation*} 
We define the \textit{half-rugby ball} set of level $s$ ---the name is inspired by its shape in the three dimensional Poincaré disk model~$\mathbb{D}^3$---, which is isometric to $\mathbb{H}^{d-1}$, as
\begin{equation}\label{Rugbyball}	
    \mathcal{R}_{\mathfrak{h}}(s):=\left\{x\in\Hd: \rho_\mathfrak{h}(x)=s\right\}\quad\textrm{with }s\in\mathbb{R}.
\end{equation}
Thus $ (\rho_\mathfrak{h},\theta_\mathfrak{h})\in\mathbb{R}\times\mathbb{H}^{d-1}$. This allows us to regard $\dhh$ as $\mathbb{R}_+\times\mathbb{H}^{d-1}$; see Figure~\ref{FigureRubgy}.
Arrowed curves there represent geodesics, parametrized by $\rho_{\mathfrak{h}}\in\mathbb{R}$. We color in orange some half-rugby balls, including the hyperplane $P$.

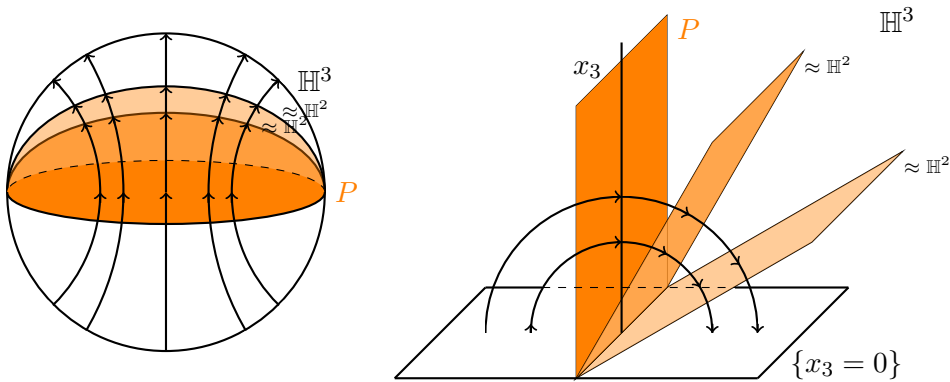
\begin{figure}[H]
    \begin{center}
        \subfigure
        %[Poincaré disk model $\mathbb{D}^3$.]
        {
            \begin{tikzpicture}[scale=0.7]
                %recortador
                \clip (-3.6,-3.6) rectangle (3.6,3.6);
                %Esfera de afuera $ \Hd$
                % Draw the outer circle (the sphere)
                \draw[thick] (0,0) circle (3cm);
                %\draw[thick, fill=orange,  fill opacity=0.9] (0,0) ellipse (3cm and 0.6cm);
                
                \filldraw[black] (0,0) circle (2pt); % center point
                
                %------------ Balones de rugby
                \begin{scope}
                    \clip (-3.1,0) rectangle (3.1,2.1); % Recorta la mitad inferior
                    \draw[thick, fill=orange,  fill opacity=0.6] (0,0) ellipse (3cm and 1.5cm);
                    \draw[thick, fill=orange,  fill opacity=0.4](0,0) ellipse (3cm and 2cm);
                \draw[dashed, fill=orange,  fill opacity=1] (0,0) ellipse (3cm and 0.6cm);
                \end{scope}
                    \begin{scope}
                    \clip (-3.1,0) rectangle (3.1,-2.1); % Recorta la mitad superior
                    \draw[thick, fill=orange,  fill opacity=1] (0,0) ellipse (3cm and 0.6cm);
                \end{scope}
                \node[right] at (2*0.7181+0.17,2*0.7-0.12) { {\tiny$\approx\mathbb{H}^{2}$}};
                \node[right] at (2.5*0.7181+0.2,2.5*0.7-0.17)  { {\tiny$\approx\mathbb{H}^{2}$}};    
                \node[right] at (3,0) { \textcolor{orange}{$P$}};             
                
                %---------Poner flechitas
                \draw[thick,->](-1.5,1.2)--(-1.55,1.3);
                \draw[thick,->](1.5,1.2)--(1.55,1.3);
                \draw[thick,->](0.97,1.3)--(1,1.43);
                \draw[thick,->](-0.97,1.3)--(-1,1.43);
                \draw[thick,->](1.09,1.7)--(1.15,1.85);
                \draw[thick,->](-1.09,1.7)--(-1.15,1.85);
                \draw[thick,->](1.67,1.55)--(1.74,1.65);
                \draw[thick,->](-1.67,1.55)--(-1.74,1.65);
                \draw[thick,->](2.02,2.02)--(2.15,2.12);
                \draw[thick,->](-2.02,2.02)--(-2.15,2.12);
                \draw[thick,->](1.45,2.52)--(1.5,2.62);
                \draw[thick,->](-1.45,2.52)--(-1.5,2.62);
                %Geodesicas rectas
                \draw[thick,->] (0,2) -- (0,3);
                \draw[thick,->] (0,1.5) -- (0,2);
                \draw[thick,->] (0,0) -- (0,1.5);
                \draw[thick,->] (0,-3) -- (0,0);
                % Geodesicas curvas
                \draw[thick] (-3*1.4142,0) circle (3 cm);
                \draw[thick,->](-3*1.4142+3,-0.2)--(-3*1.4142+3,0);
                \draw[thick] (3*1.4142,0) circle (3 cm);	   \draw[thick,->](3*1.4142-3,-0.2)--(3*1.4142-3,0);
                \draw[thick] (-6,0) circle (5.196 cm);
                \draw[thick,->](-6+5.196,-0.2)--(-6+5.196,0);
                \draw[thick] (6,0) circle (5.196 cm);
                \draw[thick,->](6-5.196,-0.2)--(6-5.196,0);
                %Pintar de blanco
                \fill[white] (-2.55,-2.44) circle (0.513cm);
                \fill[white] (2.55,-2.44) circle (0.513cm);
                \fill[white] (-2.55,2.44) circle (0.513cm);
                \fill[white] (2.55,2.44) circle (0.513cm);
                
                \fill[white] (-1.7,-3) circle (0.428cm);
                \fill[white] (1.7,-3) circle (0.428cm);
                \fill[white] (-1.7,3) circle (0.428cm);
                \fill[white] (1.7,3) circle (0.428cm);
                %Pintar de blanco
                \fill[white ] (0,5) -- (0,6) -- (11.4,6) --	(11.4,0)  -- (5,0) --cycle;
                \fill[white ] (0,-5) -- (0,-6) -- (11.4,-6) --	(11.4,0)  -- (5,0) --cycle; 	
                \fill[white ] (0,5) -- (0,6) -- (-11.4,6) --	(-11.4,0)  -- (-5,0) --cycle; 				
                \fill[white ] (0,-5) -- (0,-6) -- (-11.4,-6) --	(-11.4,0)  -- (-5,0) --cycle;
                %------ H3 simbolo
                \node[right] at (3.2*0.7181,3.2*0.7-0.1) {$ \mathbb{H}^{3}$};
            \end{tikzpicture}
        }
        \subfigure
        %[Poincaré half-space model $\mathbb{U}^3$.]
        {
            \begin{tikzpicture}[scale=0.6]
                \clip (-5.2,-1.1) rectangle (6*0.866+1.9+0.2,7+0.4) ;% que retángulo pintamos
                
                %	\filldraw[black] (0,0) circle (2pt); % center point
            
                \node[left] at (-0.2,5.8)  {$x_3$};
                \node[left] at (6*0.866+1.5,6.9)  {$\mathbb{H}^3$};
                \node[right] at (3.5,-0.7)  {$\{x_3=0\}$};
                %Planos-balones de rugby
                \draw (-1,-1) -- (1,1) -- (1,7) -- (-1,5) -- cycle;
                \fill[thick, fill=orange,  fill opacity=1] (-1,-1) -- (1,1) -- (1,7) -- (-1,5) -- cycle;
                \fill[thick, fill=white,  fill opacity=1] (-1,-1) -- (1,4*0.866-1)--(1,1)--cycle;%pintar esquina blanco
                \node[right] at (1,6.7) { \textcolor{orange}{$P$}};
                %-Plano Girado 30º
                \draw  (-1,-1) -- (6*0.866-1,6*0.5-1)--(6*0.866+1,6*0.5+1)--(1,1)--cycle;
                \fill[thick, fill=orange,  fill opacity=0.4] (-1,-1) -- (6*0.866-1,6*0.5-1)--(6*0.866+1,6*0.5+1)--(1,1)--cycle;
                \node[right] at (6*0.866+1-0.2,6*0.5+1-0.3) { {\tiny$\approx\mathbb{H}^{2}$}};
                %-Plano Girado 60º
                \draw  (-1,-1) -- (6*0.5-1,6*0.866-1)--(6*0.5+1,6*0.866+1)--(1,1)--cycle;
                \fill[thick, fill=orange,  fill opacity=0.7] (-1,-1) -- (6*0.5-1,6*0.866-1)--(6*0.5+1,6*0.866+1)--(1,1)--cycle;
                \node[right] at (6*0.5+1-0.2,6*0.866+1-0.3) { {\tiny$\approx\mathbb{H}^{2}$}};
                % Geodésicas
                \begin{scope}
                    \clip (-3,0) rectangle (3,3); % Recorta la mitad inferior
                    \draw[thick]  (0,0) circle (2cm);
                    \draw[thick]  (0,0) circle (3cm);
                \end{scope}
                %\draw[thick]  (0,0) circle (2cm);
                \draw[thick,->](-0.1,2)--(0,2);
                %\filldraw[black] (0,2) circle (1pt);
                \draw[thick,->](2*0.5-0.1,2*0.866+0.05) --(2*0.5,2*0.866) ;
                %\filldraw[black] (2*0.5,2*0.866) circle (1pt);
                \draw[thick,->](2*0.866-0.05,2*0.5+0.1) --(2*0.866,2*0.5) ;
                % \filldraw[black] (2*0.866,2*0.5) circle (1pt);
                \draw[thick,->](-2,0)--(-2,0.2);
                %\filldraw[black] (-2,0) circle (1pt);
                \draw[thick,->](2,0.2)--(2,0);
                % \filldraw[black] (2,0) circle (1pt);
                %	\draw[thick]  (0,0) circle (3cm);
                \draw[thick,->](-0.1,3)--(0,3);
                %\filldraw[black] (0,3) circle (1pt);
                \draw[thick,->](3*0.5-0.1,3*0.866+0.05)--(3*0.5,3*0.866);
                %\filldraw[black] (3*0.5,3*0.866) circle (1pt);
                \draw[thick,->](3*0.866-0.05,3*0.5+0.05) --(3*0.866,3*0.5) ;
                %\filldraw[black] (3*0.866,3*0.5) circle (1pt);
                \draw[thick,->](3,0.2)--(3,0);
                %\filldraw[black] (3,0) circle (1pt);
                
                % Dibujar el romboide (suelo)
                \draw[thick] (-5,-1) -- (3,-1); %abajo
                \draw[thick] (3,-1)-- (5,1);%izq
                \draw[thick] (-3,1)-- (-1.75,1);%arriba izq
                \draw[dashed] (-1.75,1)-- (2.5,1);%arriba centro
                \draw[thick] (5,1)-- (2.5,1);%arriba dcha
                \draw[thick] (-5,-1) -- (-3,1);%dcha
                
                %Eje ${x_3=0}$
                    \draw[thick] (0,0) -- (0,6.4);
            \end{tikzpicture}
        }
        \end{center}
        \caption{$\mathbb{D}^3$ and $\mathbb{U}^3$ as warped products induced by hyperbolic coordinates.}
        \label{FigureRubgy} 
\end{figure}

\noindent$\bullet$ \underline{$\mathfrak{p}$-coordinates:} \	Let $\uh=\{x\in\Rd:\,x_{d}>0\}$. The $\mathfrak{p}$-coordinates of  $x\in\uh$ are given by
\begin{equation*}
	\rho_\mathfrak{p}=-\log x_d,\qquad \theta_{\mathfrak{p}_{j}}=x_{j}\quad\textrm{for }j\in\{1,...,d-1\},\quad\textrm{so that }(\rho_\mathfrak{p},\theta_\mathfrak{p})\in \mathbb{R}\times \mathbb{R}^{d-1}.
\end{equation*}
Thus, the upper half-plane $\uh$ can be regarded as $\mathbb{R}_+\times\mathbb{R}^{d-1}$; see Figure~\ref{FigureParabolic}. Arrowed lines and curves represent  geodesics parametrized by $\rho_{\mathfrak h}\in\mathbb{R}$. We color in green some horospheres.

\begin{figure}[H]
    \begin{center}
        \subfigure
        %[Poincaré disk model.]ç
        {
            \begin{tikzpicture}[scale=0.7]
            % Que circulo pintamos
            \clip (-4.2,-3.6) rectangle (3.6,3.6);
            
            % Draw the horizontal dashed circle (equator)
            \draw[dashed] (0,0) ellipse (3cm and 0.6cm);
            %Horoesfera grande
            %	\draw[thick] (-1.7,0) circle (1.3cm); sin colorear
            \draw[ fill=darkgreen,  fill opacity=0.5] (-1.7,0) circle (1.3cm);
            \draw[dashed] (-1.7,0) ellipse (1.3cm and 0.26cm);
            
            %Horoesfera pequeña
            %	\draw[thick] (-2.1,0) circle (0.9cm); sin colorear
            \draw[ fill=darkgreen,  fill opacity=0.7] (-2.1,0) circle (0.9cm);
            \draw[dashed] (-2.1,0) ellipse (0.9cm and 0.18cm);
            % Add labels
            \node[right] at (-2.1,1) {{\tiny$\approx\mathbb{R}^{2}$}};
            \node[right] at (-1.7,1.5) {{\tiny$\approx\mathbb{R}^{2}$}};
            \node[right] at (3*0.7181,3*0.7181) {$ \mathbb{H}^{3}$};
            
            % Add points
            \filldraw[black] (0,0) circle (2pt); % center point
            \filldraw[black] (-3,0) circle (2pt); % point -$\infty$
            % Geodesica recta
            \draw[thick,->] (-3,0) -- (3,0);
            \draw[thick,->] (-3,0) -- (-0.4,0);
            \draw[thick,->] (-3,0) -- (-1.2,0);
            %-------
            
            %  Geodesicas curva grande
            \draw[thick] (-3, 5.2) circle (5.2cm);
            \draw[thick] (-3, -5.2) circle (5.2cm);
            %  Geodesicas curva pequeña
            \draw[thick] (-3, 2.5) circle (2.5cm);
            \draw[thick] (-3, -2.5) circle (2.5cm);
            %----
            % Poniendo blanco
            \fill[white ] (3,3.07) -- (3,6) -- (-6,6) -- (-6,3.07)  -- cycle;
            \fill[white ] (-3.07,5) -- (-6,5) -- (-6,-5) -- (-3.07,-5)  -- cycle;
            \fill[white ] (3,-3.07) -- (3,-6) -- (-6,-6) -- (-6,-3.07)  -- cycle;
            \fill[white] (1.7,3.1) circle (0.53cm);%poniendoblanco
            \fill[white] (1.7,-3.1) circle (0.53cm);%poniendoblanco
            \fill[white] (-0.5,3.4) circle (0.44cm);%poniendoblanco
            \fill[white] (-0.5,-3.4) circle (0.44cm);%poniendoblanco
            %Flechitas
            \draw[thick,->] (-1.6,0.44) -- (-1.4,0.58);
            \draw[thick,->] (-1.6,-0.44) -- (-1.4,-0.58);
            \draw[thick,->] (-1.4,0.25) -- (-1.25,0.3);
            \draw[thick,->] (-1.4,-0.25) -- (-1.25,-0.3);
            \draw[thick,->] (-0.6,0.58) -- (-0.55,0.6);
            \draw[thick,->] (-0.6,-0.58) -- (-0.55,-0.6);
            \draw[thick,->] (-1.1,0.9) -- (-0.95,1.1);
            \draw[thick,->] (-1.1,-0.9) -- (-0.95,-1.1);
            \draw[thick,->] (1.4,2.45) -- (1.5,2.58);
            \draw[thick,->] (1.4,-2.45) -- (1.5,-2.58);
            \draw[thick,->] (-0.525,2.8) -- (-0.53,2.95);
            \draw[thick,->] (-0.525,-2.8) -- (-0.53,-2.95);
            %label
            \node[left] at (-2.9,0.2) {$-\infty$};
            %esfera de afuera
            \draw[thick] (0,0) circle (3cm);
        \end{tikzpicture}
        }\;\;\;
        \subfigure
                %[Poincaré half-space model.]
                {
                \begin{tikzpicture}[scale=0.6]
                        \node[right] at (4,6.5)  {$\mathbb{H}^3$};
                % Dibujar el romboide coloreado de gris
                %\draw (-5,-1) -- (3,-1) -- (5,1) -- (-3,1) -- cycle; %base
                \draw[thick] (-5,-1)--(3,-1);%base
                \draw[thick] (3,-1)--(5,1);%base
                \draw[thick] (5,1)--(3.5,1);
                \draw[dashed](3.5,1)--(-3,1);%base
                \draw[dashed](-3,1)--(-3.5,0.5);%base
                \draw[thick](-3.5,0.5)--(-5,-1);%base
                %	\filldraw[black] (0,0) circle (2pt); % center point
                %\draw[thick] (0,0) -- (0,6);
                \node[left] at (-0.1,5.7)  {$x_3$};
                \node[centered] at (-0,6.7)  {$-\infty$};
                \node[right] at (3.2,-1.3)  {$\{x_3=0\}$};
            
                %\filldraw[black] (4.4,0.9) circle (2pt); % center point
                
                % Horoesferas
                \draw (-5,-1+4) -- (3,-1+4) -- (5,1+4) -- (-3,1+4) -- cycle;
                \fill[thick, fill=darkgreen,  fill opacity=0.7] (-5,-1+4) -- (3,-1+4) -- (5,1+4) -- (-3,1+4) -- cycle;
                \node[right] at (4,4)  { {\tiny$\approx\mathbb{R}^{2}$}};
                \draw (-5,-1+1.5) -- (3,-1+1.5) -- (5,1+1.5) -- (-3,1+1.5) -- cycle;
                \fill[thick, fill=darkgreen,  fill opacity=0.5] (-5,-1+1.5) -- (3,-1+1.5) -- (5,1+1.5) -- (-3,1+1.5) -- cycle;
                \node[right] at (4,1.5)  { {\tiny$\approx\mathbb{R}^{2}$}};
                    %Geodesicas verticales
                %	\draw[thick] (-2.3,-0.3) -- (-2.3,6-0.3);
                \draw[thick,->](-2.3,6-0.3)--(-2.3,4-0.3);
                \draw[dashed] (-2.3,4-0.3) -- (-2.3,3);
                \draw[thick,->](-2.3,3) -- (-2.3,1.5);
                \draw[dashed] (-2.3,4-1.5) -- (-2.3,0.5);
                \draw[thick,->](-2.3,0.5) -- (-2.3,-0.3);
                %\filldraw[black] (-2.5,0.4) circle (2pt);
                %---
                %\draw[thick] (0,0) -- (0,6);
                \draw[thick,->](0,6)--(0,4);
                \draw[dashed] (0,4) -- (0,3);
                \draw[thick,->](0,3) -- (0,1.8);
                \draw[dashed] (0,4-1.8) -- (0,0.5);
                \draw[thick,->](0,0.5) -- (0,0);
                %\draw[thick] (2.5,-0.7) -- (2.5,6-0.7);
                \draw[thick,->](2.5,6-0.7)--(2.5,4-0.7);
                \draw[dashed] (2.5,4-0.7) -- (2.5,3);
                \draw[thick,->](2.5,3) -- (2.5,1.8-0.7);
                \draw[dashed] (2.5,4-1.8-0.7) -- (2.5,0.8-0.5);
                \draw[thick,->](2.5,0.5) -- (2.5,-0.7);
            \end{tikzpicture}
        }
        \end{center}
        \caption{$\mathbb{D}^3$ and $\mathbb{U}^3$ as warped products induced by parabolic coordinates. }
    \label{FigureParabolic}
\end{figure}
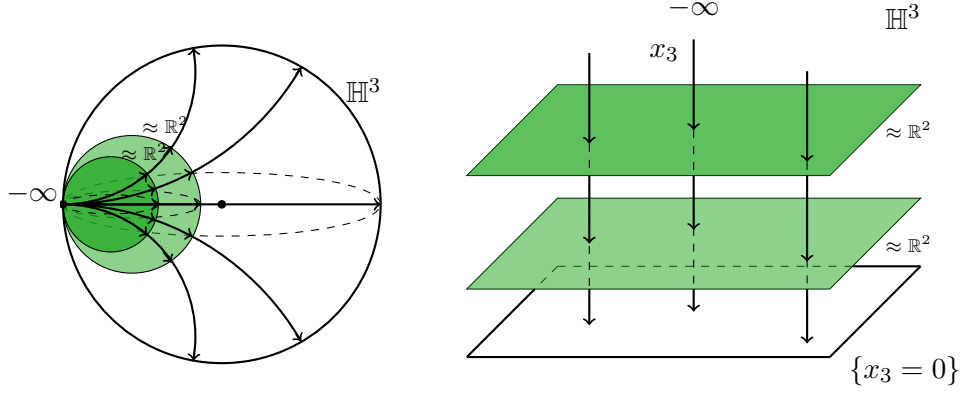

In such coordinates $(\rho,\theta)=(\rho_i,\theta_i)$,  with $i\in\{\mathfrak{e},\mathfrak{h},\mathfrak{p}\}$, $(\Hd, g)$ is  given by the warped product
\begin{equation*}
    \Hd=I_i\times \mathcal{N}_i,\quad g={\rm d}\rho^2+(\psi^i(\rho))^2g_{\mathcal{N}_i},
\end{equation*}
where $g_{\mathcal{N}_i}$ is the metric on the Riemannian manifold $\mathcal{N}_i$, which is independent of $\rho$. 
By \eqref{Laplace-WarpedProduct}, the operator $\lh$  splits as
\begin{equation}\label{LaplaceBeltramiRhoTheta}
    \lh u= u_{\rho\rho}+(d-1)h_{1}^i u_{\rho}+h^{i}_2\Delta_{\mathcal N_i} u,\quad(\rho,\theta)\in I_i\times\mathcal N_i,
\end{equation}
where 
\begin{center}
\begin{tabular}{|l|c|c|c|c|c|c|}
	\hline&&&&&&\\
	Type&$i$&$\psi^{i}(\rho)$&$h^{i}_{1}(\rho)$&$h^{i}_{2}(\rho)$&$I_{i}$&$\mathcal{N}_i$\\&&&&&&\\\hline\hline&&&&&&\\
	Elliptic&$\mathfrak{e}$&$\sinh(\rho)$&$\coth{\rho}$&$\sinh^{-2}\rho$&$\mathbb{R}_+$&$\mathbb{S}^{d-1}$\\ &&&&&&\\\hline	&&&&&&\\
	Hyperbolic&$\mathfrak{h}$&$\cosh(\rho)$&$\tanh\rho$&$\cosh^{-2}\rho$&$\mathbb{R}$&$\mathbb{H}^{d-1}$\\&&&&&&\\\hline	&&&&&&\\
	Parabolic&$\mathfrak{p}$&$e^{\rho}$&$1$&$e^{-2\rho}$&$\mathbb{R}$&$\mathbb{R}^{d-1}$\\	&&&&&&\\\hline
\end{tabular}
\end{center}
We remark that $h_1^i(\rho_0)$ is the mean curvature of the  hypersurface $\mathcal{N}_i(\rho_0):=\{(\rho,\theta)\in I_i\times\mathcal{N}_i:\rho=\rho_0\}$ (embedded in $(\Hd, g)$) for all $\rho_0\in I$ and for all $i\in\{\mathfrak{e},\mathfrak{h},\mathfrak{p}\}$. Moreover, observe that
\begin{equation}\label{limit_h}\quad
    \lim\limits_{\rho\to\infty}h^{i}_1(\rho)=\lim\limits_{\rho\to\infty}\frac{(\psi^{i})'}{\psi^{i}}=1 \quad\textrm{for all }i\in\{\mathfrak{e},\mathfrak{h},\mathfrak{p}\}.
\end{equation}
The above limit  introduces an important difference with respect to the Euclidean case: a drift of constant velocity $(d-1)$ that is slowing down the propagation. 

Observe that the drift from infinity phenomenon does not happen in the Euclidean space.  More precisely, consider  the two  warped product decompositions for   $\mathbb{R}^d$:
\begin{center}
    \begin{tabular}{lll}
        $(\rho_\mathfrak{e},\theta_\mathfrak{e})=(|x|,x/|x|)\in\mathbb{R}_+\times \mathbb{S}^{d-1}$,& $g_{\mathbb{R}^d}=\textup{d}\rho^{2}+(\psi_{\mathbb{R}^d}^{\mathfrak{e}})^{2}g_{\mathbb{S}^{d-1}}$& with  $\psi_{\mathbb{R}^d}^{\mathfrak{e}}:=\rho$,\\[8pt]
        $(\rho_\mathfrak{p},\theta_\mathfrak{p})=(x_d,(x_1,\dots,x_{d-1}))\in\mathbb{R}\times \mathbb{R}^{d-1}$,&$g_{\mathbb{R}^d}=\textup{d}\rho^{2}+(\psi_{\mathbb{R}^d}^{\mathfrak{p}})^{2}g_{\mathbb{R}^{d-1}}$& with $\psi_{\mathbb{R}^d}^{\mathfrak{p}}:=1$.
    \end{tabular}
\end{center}
Here, we use the notation $\mathfrak e$ (resp.~$\mathfrak p$), since we are considering isometries that preserve $\mathbb S^{d-1}$ (resp.~$\mathbb R^{d-1}$). Notice that, in contrast to \eqref{limit_h}, 
\begin{equation*}
    \lim\limits_{\rho\to\infty}\frac{(\psi^{i}_{\mathbb{R}^{d}})'}{\psi^{i}_{\mathbb{R}^{d}}}=0\quad\textrm{for }i=\mathfrak p,\mathfrak e.
\end{equation*}

%%%%%%%%%%%%%%%%%%%%%%%%%%%%%%%%%%%%%%%%%%%%%%%%%%%%%%%%%%%%%%%%%%%%%%%%%
\section{Analytic preliminaries.}\label{section:analytic-preliminaries}
\setcounter{equation}{0}

We gather here some preliminary results concerning the heat kernel and the existence and uniqueness of solutions to~\eqref{KPPproblem}.

\medskip

\noindent{\bf Heat kernel.} Let $v_{0}:\Hd\to\mathbb{R}$ be a bounded function. The unique bounded solution to
\begin{equation}\label{HeatEquationHd}	
    v_{t}=\lh v\quad\textrm{in }\Hd\times\mathbb{R}_+
\end{equation}
with initial datum $v_0$ is given by
\begin{equation}\label{solution-to-HE-hyperbolic}	
    v(x,t)=\int_{\Hd}G(x,y,t)v_0(y)\,{\rm d}\mu (y),
\end{equation}
where $\textup{d}\mu$ denotes the volume element of $\hs$ and $G:\Hd\times\Hd\times\mathbb{R}_+\to\mathbb{R}_+$ is a function known as the~\emph{heat kernel} of the hyperbolic heat equation. For each $y\in\mathbb{H}^d$, $G(\cdot,y,\cdot\cdot)$ satisfies~\eqref{HeatEquationHd}, and has~$\delta_y$ as initial trace (in the sense of Radon measures). Moreover, 	
\begin{equation}\label{GnormaL1}
    \|G(\cdot,y,t)\|_{L^{1}(\Hd)}=1\quad\textrm{for all }y\in\Hd,\ t>0.
\end{equation}

An explicit formula for the kernel is available,
\begin{equation}\label{P and G- Hkernel-hyperbolic}	
    G(x,y,t)=P_{t}(d_{\mathbb H^d}(x,y)),\quad x,y\in\Hd,\ t>0,
\end{equation}
with profiles $P_t$ given,  for all $r\ge0$ and $t>0$, by 
\begin{equation*}
        P_{t}(r)=
        \begin{cases}
    		\displaystyle\frac{(-1)^{m}}{2^n\pi^n (4\pi t)^{1/2}} \left(\frac{1}{\sinh r}\frac{\partial}{\partial r}\right)^n e^{-n^2 t-\frac{r^2}{4t}}&\textup{if }d=2n+1,\\
            \displaystyle\frac{(-1)^n}{2^{n+5/2}\pi^{n+3/2}t^{3/2}}e^{-\frac{(2n+1)^2t}{4}} \left(\frac{1}{\sinh r}\frac{\partial}{\partial r}\right)^n \int_{r}^{\infty}\frac{se^{-\frac{s^2}{4t}}}{\left(\cosh s-\cosh r\right)^{1/2}} \,{\rm d}s&\textup{if }d=2n,
    	\end{cases}
\end{equation*}     
with $n\in\mathbb{N}$; see~\cite{Grigoryan}. Each profile $P_t$, $t>0$, is a decreasing function and satisfies
\begin{equation}\label{HeatKernelControl}
    c_d^{-1}h_t(r)\leq P_t(r)\leq c_d h_{t}(r), \quad r\geq 0,
\end{equation}
for some constant $c_{d}>1$, where
\begin{equation}\label{HeatKernelApprox-h_t}
    h_t(r)=\frac{(1+r)}{(4\pi t)^{d/2}}(1+r+t)^{\frac{d-3}{2}}e^{-\lambda_{1}t-\frac{d-1}{2}r-\frac{r^{2}}{4t}}, \quad r\geq 0;
\end{equation}
see~\cite[Theorem 5.7.2]{Davies}. Using~\eqref{lambda1}, we can rewrite $h_{t}$ as
\begin{equation}\label{HeatKernelApprox} 
    h_{t}(r)=\frac{1}{(4\pi)^{d/2}}e^{\frac{-(r+(d-1)t)^{2}}{4t}}H_{t}(r),\quad\textrm{where }H_{t}(r)=\frac{(1+r+t)^{\frac{d-3}{2}}(1+r)}{t^{d/2}},\quad r\geq 0,\  t>0,
\end{equation}
an expression that will be useful later.

\medskip

\noindent{\bf Existence and uniqueness of solutions to~\eqref{KPPproblem}.} Uniqueness is a consequence of the following comparison result, which can be found within the proof of~\cite[Theorem 2.6]{UniquenessInHd}.

\begin{theorem}[Comparison principle in $\Hd$ for bounded functions]\label{ComparsionPrincipleHd}  
    Consider $T>0$ and $g:\mathbb{R}\to\mathbb{R}$ a locally Lipschitz function. Let $\underline{u}$ and $\overline{u}$ be two bounded $C^{2,1}_{x,t}(\Hd\times (0,T))$-functions such that
	\begin{equation*}
		\underline{u}_{t}-\Delta_{\Hd}\underline{u}- g(\underline{u})\leq 0,\qquad\overline{u}_{t}-\Delta_{\Hd}\overline{u} -g(\overline{u})\geq 0\quad\textrm{in }\Hd\times (0,T).
	\end{equation*}
	If $\underline{u}(x, 0)\leq \overline{u}(x, 0)$  for all $x\in\Hd$, then $\underline{u}\leq \overline{u}$ in $\Hd\times (0,T)$.
\end{theorem}

The existence of a local in time solution follows easily from a standard fixed point argument for the integral operator $\mathcal{T}_{u_0}$ defined by
\begin{equation*}\label{KPPIntegralEq}	
   \mathcal{T}_{u_0} u(x,t):=\int_{\Hd}G(x,y, t)u_0(y)\,{\rm d}\mu (y)+\int_{0}^{t}\int_{\Hd}G(x,y, t-s)f(u(y,s))\,{\rm d}\mu (y)\,{\rm d}s,
\end{equation*}
using the regularity of $G$ and $f$. The~\emph{a priori} estimates $0\le u\le 1$, that are an immediate consequence of the comparison principle, allow then to prove that the local in time solution is in fact global. See~\cite[Theorem 1]{KPP-article} for an analogous reasoning in the case of the Euclidean space. We thus have the following result.

\begin{proposition}
	Let $u_0:\Hd\to[0,1]$ measurable. There exists a unique classical bounded solution $u$ to~\eqref{KPPproblem} with initial datum $u_0$.  Moreover, $u\in C^\infty(\Hd\times\mathbb{R}_+)$ and $0\leq u\leq 1$ in $\Hd\times\mathbb{R}_+$.
\end{proposition}

Since the Laplace-Beltrami operator $\Delta_{\Hd}$ is invariant under the action of the isometries of the hyperbolic space, the comparison principle, Theorem~\ref{ComparsionPrincipleHd}, implies that if the initial datum has some symmetry, the same is true for the solution at any later time.

\begin{corollary}[Invariance by isometries]\label{CorInvariantIsometries}
	Let $u_0:\mathbb{H}^d\to[0,1]$ measurable and invariant under the action of a group of isometries of the hyperbolic space. Let $u$ be the solution to~\eqref{KPPproblem} with initial datum $u_0$. Then, for any $t>0$, $u(\cdot,t)$ is invariant under the action of the same group of isometries.
\end{corollary}

%%%%%%%%%%%%%%%%%%%%%%%%%%%%%%%%%%%%%%%%%%%%%%%%%%%%%%%%%%%%%%%%%%%%%%%%%%%%%%%%%%%%%
\section{Propagation versus extinction}\label{section:vs}
\setcounter{equation}{0}

In this section we obtain new propagation and extinction results, not contained in~\cite[Theorem~3.2]{Matano}. More precisely, we address these issues for more general data than the ones considered there, and we also study the critical case $f'(0)=\lambda_{1}$, with $\lambda_{1}$ given by~\eqref{lambda1}.

\noindent{\bf Extinction.} Let $u$ be a solution of~\eqref{KPPproblem} with a measurable initial datum $u_0:\Hd\to[0,1]$. Since~$f$ is a KPP function,  satisfying therefore~\eqref{KPPfunction.2},  by the comparison principle, Theorem~\ref{ComparsionPrincipleHd}, $u\leq \tilde{u}$ in $\Hd\times\mathbb{R}_+$, where $\tilde{u}$ is the unique bounded solution to the linear problem
\begin{equation*}
    \tilde{u}_{t}=\Delta_{\Hd}\tilde{u}+f'(0)\tilde{u}\quad\textrm{in }\Hd\times\mathbb{R}_{+},\qquad \tilde{u}(\cdot,0)=u_0\quad\textup{in }\Hd.
\end{equation*}
The function $v:=e^{-f'(0)t}\tilde{u}$ satisfies the heat equation~\eqref{HeatEquationHd}, and is therefore given by~\eqref{solution-to-HE-hyperbolic} with initial datum~$v_0=\tilde u_0=u_0$. Hence, 
\begin{equation*}
    \tilde{u}(x,t)=e^{f'(0)t}\int_{\Hd}G_{t}(x,y)u_{0}(y)\,{\rm d}\mu(y)\quad\textrm{for all }(x,t)\in\Hd\times\mathbb{R}_+.
\end{equation*}
Using also~\eqref{P and G- Hkernel-hyperbolic}, we finally arrive to the bound  
\begin{equation}\label{LineralizationEq} 
    u(x,t)\leq e^{f'(0)t}\int_{\Hd}P_{t}(d_{\mathbb H^d}(x,y))u_{0}(y)\,{\rm d}\mu(y)\quad\textrm{for all } (x,t)\in\Hd\times\mathbb{R}_+,
\end{equation}
crucial in the proof of next theorem, which has our first extinction result as a corollary.

\begin{theorem}\label{PropExtincionL1}	
    Let $u_{0}:\Hd\to[0,1]$ in $ L^{p}(\Hd)$, $p\in[1,\infty)$. Let $u$ be the solution to~\eqref{KPPproblem}. Then
	\begin{equation}\label{extinctionDecayForL1}
        \|u(\cdot,t)\|_{L^{\infty}(\Hd)}=O\big(t^{-\frac{3}{2p}}e^{(f'(0)-\frac{\lambda_1}{p})t}\big)\quad\textrm{as }t\to\infty.
	\end{equation}    
\end{theorem}

\begin{proof}
Applying first Hölder's inequality ($p'$ denotes the conjugate exponent of $p$), then interpolation, and finally that $\|P_t(d_{\mathbb H^d}(x,\cdot))\|_{L^{1}(\Hd)}=1$ for all $x\in\Hd$ and $t>0$,  we get
\begin{equation}\label{eq:proof.thm.extinction}
    \begin{aligned}
        \int_{\Hd}P_{t}(d_{\mathbb H^d}(x,y))u_{0}(y)\,{\rm d}\mu(y)&\le\|P_{t}(d_{\mathbb H^d}(x,\cdot))\|_{L^{p'}(\Hd)}\|u_0\|_{L^{p}(\Hd)}\\
        &\leq \|P_{t}(d_{\mathbb H^d}(x,\cdot))\|_{L^1(\Hd)}^{\frac{1}{p'}}\|P_{t}(d_{\mathbb H^d}(x,\cdot))\|_{L^{\infty}(\Hd)}^{\frac{1}{p}}\|u_0\|_{L^{p}(\Hd)},\\
        &\leq\|P_{t}(d_{\mathbb H^d}(x,\cdot))\|_{L^{\infty}(\Hd)}^{\frac{1}{p}}\|u_0\|_{L^{p}(\Hd)},\quad x\in\Hd,\ t>0.
    \end{aligned}
\end{equation}
On the other hand, since $P_t$ is, for all $t>0$, a decreasing function of the geodesic distance, we obtain, using also~\eqref{HeatKernelControl} and~\eqref{HeatKernelApprox-h_t}, that there is a constant $C>0$ such that  
\begin{equation}\label{P_t at the origen}	
    \|P_t(d_{\mathbb H^d}(x,\cdot))\|_{L^{\infty}(\Hd)}\leq P_{t}(0)\leq Ct^{-3/2}e^{-\lambda_{1} t},\quad x\in\Hd,\ t>0.
\end{equation}
The combination of~\eqref{LineralizationEq},~\eqref{eq:proof.thm.extinction} and~\eqref{P_t at the origen} yields the result.
\end{proof}

As a corollary we obtain the following extinction result that includes information for the critical case $f'(0)=\lambda_{1}$.

\begin{corollary}
    Let $f'(0)\le \lambda_1$ and $u_{0}:\Hd\to[0,1]$ in $ L^{p}(\Hd)$, $p\in[1,\lambda_1/f'(0)]$. Then the solution~$u$ to~\eqref{KPPproblem} vanishes,  that is, $\lim\limits_{t\to \infty}\|u(\cdot, t)\|_{L^\infty(\Hd)}=0$.
\end{corollary}

Our next result, based on a comparison argument with an explicit supersolution, shows that extinction may also occur for initial data that do not belong to any~$L^p$ space except $L^\infty$. 

\begin{proposition}\label{PropextinctionParabolic}
    Let  $f'(0)<\lambda_{1}$ and $u$ the solution to~\eqref{KPPproblem} with $0\le u_{0}(x)\leq \mathds{1}_{(-\infty,R_0)}(\rho_\mathfrak{p}(x))$ for some $R_0\geq 0$. Then, for all $\nu\in\mathbb{R}$,
    \begin{equation}\label{ExtinctionParabolic}	
        \lim\limits_{t\to\infty}\sup\limits_{\{x\in\Hd:\,\rho_\mathfrak{p}(x)\geq\nu\}}u(x,t)=0.
    \end{equation}
\end{proposition}

\begin{proof}
Since $f'(0)<\lambda_{1}$, then $d-1>c_0$; see~\eqref{eq:equivalent.conditions}. Hence, there exists a non-increasing function $q:[0,\infty)\to [0,1]$ such that
\begin{equation*}
	q''+(d-1) q'+f(q)=0\quad\textrm{in }\mathbb{R}_{+},\qquad q(0)=1,\quad	\lim\limits_{s\to\infty}q(s)=0;
\end{equation*}
see~\cite[Proposition~4.1]{AronsonWeinberger}. Let $v$ be the solution to
\begin{equation*}
	v_{t}=v_{\rho\rho}+(d-1)v_{\rho}+f(v)\quad\textrm{in }\mathbb{R}\times\mathbb{R}_{+},\qquad v(\rho,0)=
    \begin{cases}
        1,&\rho\leq R_0,\\
        q(\rho-R_0),&\rho\geq R_0.
    \end{cases}
\end{equation*}
Since $u$ satisfies the same equation as $v$, see~\eqref{LaplaceBeltramiRhoTheta}, and $0\le u_0\le v(\cdot,0)$, the comparison principle, Theorem~\ref{ComparsionPrincipleHd},  implies that $0\leq u(x,t)\leq v(\rho_{\mathfrak{p}}(x),t)$ for all $(x,t)\in\Hd\times\mathbb{R}_{+}$.
But  
\[
    \lim\limits_{t\to\infty}\sup\limits_{\{\rho\geq\nu \}}v(\rho,t)=0\quad \textup{for all } \nu\in\mathbb{R},
\]
see the proof of \cite[Theorem 5.1]{AronsonWeinberger}, hence the result.
\end{proof}

\noindent\emph{Remark. } An analogous result in hyperbolic coordinates cannot be proved with the current method due to the behavior of the associated ODE.

\medskip

\noindent{\bf Propagation.}  As a consequence of Theorem~\ref{ThPropagation}, the comparison principle and the invariance of Laplace-Beltrami's operator $\Delta_{\Hd}$ under $i$-isometries, $i\in\{\mathfrak{e},\mathfrak{h},\mathfrak{p}\}$, we get a propagation result that in the case of initial data with hyperbolic or parabolic symmetry is stronger than \cite[Theorem 3.6]{Matano}, since it shows uniform convergence to 1 in sets that for such data are not compact.

\begin{proposition}\label{prop:convergence.to.1.compact}
    Let $f'(0)>\lambda_{1}$ and  $u_0: \Hd\to[0,1]$ measurable such that 
    \begin{equation*}
    	u_0(x)\geq v_0(\rho_i(x))\quad\textrm{for all }x\in\Hd
    \end{equation*}
    for some $i\in\{\mathfrak{e},\mathfrak{h},\mathfrak{p}\}$ and some $v_0:I_i\to[0,1]$, $v_0\not \equiv 0$. Let $u$ be the solution to~\eqref{KPPproblem}. Then, 
    \begin{equation}\label{propagation1}	
        \lim\limits_{t\to\infty}\inf\limits_{\{x\in\Hd:\, |\rho_i(x)|\leq \delta\}}u(x,t)=1\quad\textup{for all }\delta>0.
    \end{equation}
\end{proposition}	

\begin{proof}
Let $\tilde{u}$ be the solution to~\eqref{KPPproblem} with initial datum $\tilde{u}(x,0)=v_0(\rho_i(x))$ for all $x\in\mathbb{H}^d$. It has the form $\tilde{u}(x,t)= v(\rho_i(x),t)$; see Corollary~\ref{CorInvariantIsometries}. Let 
\[
    \mu(s)=\min_{\{x\in\Hd:\, \rho_i(x)=s\}} d_{\mathbb H^d}(x,0),\quad s\in I_i,\qquad \kappa=\max_{\{s\in I_i:\, |s|\le \delta\}}\mu(s).
\]
Then, using the comparison principle and the symmetry of $\tilde u$, that implies that it is constant on each level set $\rho_i(x)=s$,
\begin{equation}\label{eq:propation.non.compact}
    \inf_{\{x\in\Hd:\, |\rho_i(x)|\leq \delta\}}u(x,t)\ge\inf_{\{x\in\Hd:\, |\rho_i(x)|\leq \delta\}}\tilde u(x,t)=\inf_{\{d_{\mathbb H^d}(x,0)\le \kappa\}}\tilde u(x,t).
\end{equation}
Since the set $\{d_{\mathbb H^d}(x,0)\le \kappa\}$ is compact, Theorem~\ref{ThPropagation} implies that the right-hand side of~\eqref{eq:propation.non.compact} goes to 1 as $t\to\infty$, and the result follows. 
\end{proof}

In view of Theorem \ref{ThSpreadingSpeedElliptic}, the spreading speed of solutions $u$ with a nontrivial, compactly supported initial datum  is $c_*$. However, the following result shows that to locate level sets with more precision a logarithmic correction is needed. 
\begin{proposition}\label{NecesityOfLogathElliptic}
    Let $f'(0)>\lambda_1$ and let $u$ be a solution  to~\eqref{KPPproblem} with a compactly supported initial datum $u_0:\Hd\to[0,1]$. Then, 
    \begin{equation}\label{NecesityLogElliptic}	
        \lim\limits_{t\to\infty}\sup\limits_{\{x\in\Hd:\;\rho_\mathfrak{e}(x)>c_*t-C\log t\}}u(x,t)=0\quad\textrm{ for all }0\leq C<1/c_0.
    \end{equation}
\end{proposition}

\begin{proof} 
Since $u_0$ is compactly supported there exists $R_0>0$ such that $u_0(x)\leq\mathds{1}_{[0,R_0]}(\rho_\mathfrak{e}(x))$ for all $x\in\Hd$. Then, using inequality~\eqref{LineralizationEq} and passing to polar coordinates, we have that 
\begin{align*}
    u(x,t)&\leq e^{f'(0)t}\int_{\Hd}G_{t}(x,y)\mathds{1}_{[0,R_0]}(\rho_\mathfrak{e}(y))\,{\rm d}\mu(y)\\
    &=\omega_{d-1}e^{f'(0)t}\int_{0}^{R_0}P_{t}\big(d_{\mathbb H^d}((\rho(x),\theta(x)),(r,\xi)\big)\sinh^{d-1}r\,\textup{d}r
\end{align*}
for all $(x,t)\in \Hd\times\mathbb{R}_{+}$, where  $\omega_{d-1}$ is the surface measure of the unit sphere $\mathbb{S}^{d-1}$, due to the fact that $G_{t}(x,y)=P_{t}(d_{\mathbb H^d}(x,y))$ for all $x,y\in\Hd$.

On the other hand, let $g:\mathbb{R}_{+}\to\mathbb{R}_+$, $g=o(t^{1/2})$, to be determined at the end of the proof. If $r\geq c_*t-g(t)$, then $d_{\mathbb H^d}((\rho,\theta),(r,\xi))\geq c_*t-g(t)-R_0>0$ for all $r\in[0,R_0]$ and $\theta,\xi\in\mathbb{S}^{d-1}$. Therefore, since the heat kernel profile $P_t$ is decreasing in the geodesic distance variable, we have that
\begin{align*}
    \sup\limits_{\{x\in\Hd:\;\rho_\mathfrak{e}(x)>c_*t-C\log t\}}u(x,t)&\leq \omega_{d-1}e^{f'(0)t}P_{t}(c_*t-g(t)-R_0)\int_{0}^{R_0}\sinh^{d-1}r\,\textup{d}r\\
	&\leq C_{d,R_0}t^{-1/2}\exp\big(\lambda_{0}g(t)\big)\quad\textrm{for }t>0\textrm{ large enough},
\end{align*}
due to~\eqref{HeatKernelControl}, \eqref{HeatKernelApprox}, \eqref{c0andc*} and~\eqref{lambda_0}. Taking $g(t):=C\log t$ with $0\leq C<1/c_0$, we obtain~\eqref{NecesityLogElliptic}.
\end{proof}

\noindent\emph{Remark. } Our main result will show that Proposition~\ref{NecesityOfLogathElliptic} is still valid if $0\le C<3/c_0$.

%%%%%%%%%%%%%%%%%%%%%%%%%%%%%%%%%%%%%%%%%%%%%%%%%%%%%%%%%%%%%%%%%%%%%%%%%%%%%
%%%%%%%%%%%%%%%%%%%%%%%%%%%%%%%%%%%%%%%%%%%%%%%%%%%%%%%%%%%%%%%%%%%%%%%%%%%%%
\section{Traveling-wave behavior: Strategy of the proof of Theorem~\ref{ThConvergenceTravellingWave}}\label{section:strategy}
\setcounter{equation}{0}

This section is devoted to explaining the strategy of the proof of our main result, Theorem~\ref{ThConvergenceTravellingWave}. Throughout it we always assume, without further mention, that $f'(0)>\lambda_{1}$, so that there is propagation, and that the initial datum satisfies~\eqref{eq:symmetry.u0}--\eqref{eq:supported-from-right} for some $i\in\{\mathfrak{e},\mathfrak{h},\mathfrak{p}\}$.

\noindent \underline{1. $i$-invariance.} Since $u_{0}$ has the symmetry~\eqref{eq:symmetry.u0}, Corollary~\ref{CorInvariantIsometries} implies that $u(x,t)=\widetilde{u}(\rho_i(x),t)$ for all $(x,t)$ in $\Hd\times \mathbb{R}_{+}$, where the function $\widetilde{u}:I_{i}\times \mathbb{R}_+\to [0,1]$  satisfies
\begin{equation*}
	\begin{cases}
		\widetilde{u}_{t}(\rho,t)=\widetilde{u}_{\rr\rr}(\rr,t)+(d-1)h^{i}_{1}(\rr)\widetilde{u}_{\rr}(\rr,t)+ f(\widetilde{u}(\rr,t)),	&\rho\in I_i ,\ t>0,\\
		\widetilde{u}(\rr,0)=\widetilde{u}_{0}(\rr),&\rho\in I_i.
	\end{cases}
\end{equation*}
If $i=\mathfrak{e}$, we should add an extra condition at the origin in order to have a unique solution, for instance $\widetilde{u}_\rho(0,t)=0$ for all $t>0$. However, as we will see, this extra condition will play no role.

From now on, with some abuse of notation, we will drop the tildes in the notation for $i$-symmetric functions, that is, we simply write $u(\rho,t)$.

\noindent \underline{2. Moving coordinates.}  Looking at Proposition~\ref{NecesityOfLogathElliptic}, we expect  the transition zone where $u$ is not  $0$ nor~$1$  to be located, for $t$ big enough, around
\begin{equation}\label{R,c0,c*,lambda0}	
    R(t):=c_* t-k\log t+\rho_{0},
\end{equation}
where $c_*$ is given by~\eqref{c0andc*}, and $k$ and  $\rho_0$  are positive constants that we will choose later. Thus we study ${u}$ in the moving coordinates $\varrho=\rho-R(t)$, that is,
\begin{equation}\label{utransladadaR(t)}
    \hat{u}(\varrho,t):={u}(\varrho+R(t),t-1),
\end{equation}
which is a solution to
\begin{equation*}
	\begin{cases}
    	\hat{u}_{t}=\hat{u}_{\varrho\varrho}+\big(\mathcal H^{i}(\varrho,t)+c_0-\frac{k}{t}\big)\hat{u}_{\varrho}+f(\hat{u}),&\varrho+R(t)\in I_i,\ t>1,\\
		\hat{u}(\varrho,1)={u}_{0}(\varrho+R(1)),&\varrho+R(1)\in I_i,
	\end{cases}
\end{equation*}
where we have defined
\begin{equation}\label{h}
    \mathcal H^{i}(\varrho,t):=(d-1)\left(h^{i}_{1}(\varrho+R(t))-1\right), \quad i\in\{\mathfrak{e},\mathfrak{h},\mathfrak{p}\}.
\end{equation}
The translation in time is just to avoid difficulties with the logarithm at the initial time. 

Observe  that taking $\rho_0>0$ large enough, we can assume $R(t)\geq \tfrac{c_*}2t$ for all $t\geq 1$. Therefore, $\mathcal H^{i}(\varrho,t)$ is bounded for all $\rho\in I_i$, $\varrho=o(t)$ as $t\to\infty$, $t\geq 1$ and $i\in\{\mathfrak{e},\mathfrak{h},\mathfrak{p}\}$, thus overcoming the problem introduced by the asymptote of $h^{\mathfrak{e}}_1$. In particular, if $\varrho=o(t)$, then $\varrho+R(t)\to\infty$ as $t\to\infty$, and the \lq\lq curvature coefficient'' $\mathcal H^i$ satisfies
\begin{equation}\label{Hi-to-0}
    \mathcal H^i(\varrho,t)\to0 \quad\text{as }t\to\infty.
\end{equation}
This estimate will be made precise in Lemma \ref{lemma:bound-H}. 

Hence, we expect $\hat{u}$ to solve approximately, for $\varrho=o(t)$ and $t$ large, 
\begin{equation}\label{eq:check-u}    
    \check{u}_t=\check{u}_{\varrho\varrho}+\big(c_0-\tfrac{k}{t}\big)\check{u}_{\varrho}+ f(\check{u}).
\end{equation}
In the following, we will use the ``check'' notation to denote an approximate solution that serves as a model.

\begin{remark}
    In the parabolic case $\mathcal{H}^{\mathfrak{e}}(\varrho,t)\equiv 0$ (recall the equation for $u$~\eqref{problem-parabolic-intro} in the introduction), and there is no error in the approximation $\check{u}=\hat u$.
\end{remark}

Note that \eqref{eq:check-u} is the equation studied by Roquejoffre, Rossi and Roussier-Michon in~\cite[Section~2]{RoquejofferNdimensional} to address the $d$-dimensional Euclidean KPP problem~\eqref{KPPproblemEuclidean} in the radial-like setting. Nevertheless, the neglect of the error term here is one of the most delicate points of the paper, since it is not enough to have the decay given in \eqref{Hi-to-0}. We need at least $\mathcal H^i(\varrho,t)=o(1/t)$. Otherwise, the error term may produce an additional contribution to the logarithmic correction; see, for instance, the radial Euclidean case in high dimensions, where the error is $O(1/t)$ and $k=(d+2)/c_0$.  

In our case, we will take
\begin{equation}\label{k}
    \lambda_0k=\frac{3}{2},
\end{equation}
a choice that will become clear after our discussion below.

\noindent\underline{3. Adaptation of the Euclidean proof.} 
Going back to \eqref{eq:check-u}, we observe that this is nothing but the Euclidean unidimensional KPP equation in moving coordinates. In view of \eqref{BramsonResult},  $\check{u}$ is known to approach the traveling wave $\Phi_{c_0}$, choosing $k=3/c_0$. Hence we also expect a traveling-wave behavior for $\hat u$. 

The proof of this fact will be carried out separately in each of the following three regions: 
\begin{itemize}
    \item Diffusive: $|\rho+\rho_0-R(t)|\le o(\sqrt t)$. 
    \item Subdiffusive: $\rho+\rho_0\le R(t)-o(\sqrt t)$.
    \item Superdiffusive: $\rho+\rho_0\ge R(t)+o(\sqrt t)$.
\end{itemize}
The need to introduce $\rho_0$ is technical, its only effect being to translate the regions under consideration by a constant quantity. This difficulty does not appear in the Euclidean case, as it comes from the curvature terms, as we have explained in Step 2.  In any case, observe that
\begin{equation*}
 R(t)-\rho_0=c_* t-\tfrac{3}{c_0}\log t,
\end{equation*}
so that the diffusive region reduces to $\big|\rho-(c_* t-\tfrac{3}{c_0}\log t)\big|\le o(\sqrt t)$. The precise $o(\sqrt t)$  will be of order~$t^\gamma$ for some $\gamma\in(0,1/2)$, as we will explain in Section~\ref{section:details}. Thus, for $t$ large, $\rho_0\ll t^\gamma$, and its effect can be neglected.  

Let us explain now what happens in each of the three different scenarios.

In the subdiffusive region solutions converge to 1. This is shown by comparison with a cleverly chosen subsolution, once one knows the behavior at the left end of the diffusive region and  knowing also that solutions tend to 1 in compact sets.

In the superdiffusive region solutions converge to 0. This is done by comparison, once one knows the behavior at the right end of the diffusive region.

In the diffusive region we will show convergence to a traveling wave. The behavior in this region will be determined by the behavior at its right edge. This will be made more precise in Lemma~\ref{LemmaS}, where we study the difference  between the solution $u$ and a suitable traveling wave. In particular, we will show that if at the right edge both $u$ and the traveling wave coincide exactly, then we have convergence in the whole diffusive interval. Note that in our proof,  we only need a good estimate for  the boundary condition at the left edge, but not a sharp behavior.

In summary, we expect KPP reactions to be in the so-called \emph{pulled} case, using the terminology of the reaction-diffusion community, so that the large-time behavior of the solution is determined by its behavior at the leading edge of the front, $\rho=c_* t-\tfrac{3}{c_0}\log t+o(\sqrt t)$ \cite{Stokes-1976}.

\noindent\underline{4. The behavior at the front.}
Thus, our goal is to determine the dynamics of $u$ at the leading edge of the front, where we need a very precise matching of our solution to a traveling wave. Recall that we have set $\varrho=\rho-R(t)$. 
At this position the traveling wave has an exponential decay, see~\eqref{behaviorTw}. To compensate for this decay, we make a zoom of $\hat{u}$. For this, we consider the new variable
\begin{equation}\label{vdefinicion}	
	v(\varrho, t):=e^{\lambda_0\rho} \hat u(\varrho,t),
\end{equation}
which satisfies
\begin{equation}\label{eq:v-step3}
	\begin{cases}	
        v_{t}=v_{\varrho\varrho}+\big(\mathcal H^{i}-\frac{k}{t}\big)(v_{\varrho}-\lambda_0 v)+ e^{\lambda_0\varrho} f(e^{-\lambda_0\varrho}v)-f'(0)v,	&\varrho+R(t)\in I_i,\ t>1,\\
		v(\varrho,1)=e^{\lambda_0\varrho}{u}_{0}(\varrho+R(1)),&\varrho+R(1)\in I_i.
	\end{cases}
\end{equation}
 
To understand the behavior at $\varrho=o(\sqrt t)$ in the moving frame, it is natural to use self-similar coordinates
\begin{equation}\label{hatw-transformacion4b}	
    \eta:=\frac{\varrho}{\sqrt{t}},\qquad \tau:=\log t.
\end{equation}
Nevertheless, we have to take some care here. As we have mentioned above, we expect the solution $\hat u$ in moving coordinates to be close to a traveling wave $\Phi_{c_0}(\varrho)$ if $\varrho$ is large. But if $\varrho$ is large, then $\Phi_{c_0}(\varrho)\asymp \varrho e^{-\lambda_0\varrho}$ (recall~\eqref{behaviorTw}). This means that
\[
    v(\eta t^{1/2},t)= e^{\lambda_0\eta t^{1/2}}\hat u(\eta t^{1/2},t)\asymp e^{\lambda_0\eta t^{1/2}} \Phi_{c_0}(\eta t^{1/2})\asymp \eta t^{1/2}.
\]
Therefore, to have a function of order 1 as $t\to\infty$ for $\eta>0$ we need to zoom down $v$ in this scale by a factor $t^{-1/2}$.

Motivated by this discussion, we define
\begin{equation}\label{hatw-transformacion4a}
    {w}(\eta,\tau):=e^{-\tau/2}v(\varrho,t), 
\end{equation}
in self-similar coordinates \eqref{hatw-transformacion4b}. Then $w$ satisfies
\begin{equation}\label{hatw with L}
	\begin{cases}
        {w}_{\tau}+\mathcal{L}{w}=e^{-\tau/2}H^{i}{w}_{\eta}-(\lambda_0 H^{i}+\frac{3}{2}){w}+ {F}(\eta,\tau,{w}),&\eta e^{\tau/2}\geq -R(e^\tau),\ \tau>0,\\
		{w}(\eta,0)=e^{\lambda_0\eta}{u}_{0}(\eta+R(1)),&\eta\geq-R(1),
	\end{cases}
\end{equation}
where 
\begin{align}
    \mathcal{L} w&:=- w_{\eta\eta}-\tfrac{\eta}{2} w_{\eta}-(\lambda_0k-\tfrac12) w,\\
    \label{hat{h}^i} H^{i}(\eta,\tau)&:=e^{\tau}\mathcal H^{i}(\eta e^{\tau/2},e^{\tau})-k,\quad i\in\{\mathfrak{e},\mathfrak{h},\mathfrak{p}\},\\
    \label{hat{F}} {F}(\eta,\tau,w)&:=e^{\tau/2+\lambda_0\eta e^{\tau/2}}f\big(e^{\tau/2-\lambda_0\eta e^{\tau/2}} w\big)-e^{\tau}f'(0) w.
\end{align}

Finally, we will need  some additional information about $w$ far from the front. Notice that
\begin{equation}\label{Dirichlet1}
    w(0,\tau)=t^{-1/2}v(0,t)=t^{-1/2}\hat u(0,t).
\end{equation}
Since $0\le u(0,t)\le 1$, this yields  
\begin{equation}\label{Dirichlet-w-motivate}
    w(0,\tau)\approx0\quad\textup{as }\tau\to\infty.
\end{equation}

\noindent\underline{5. Convergence to the dipole.}
To illustrate the general case, let us explain this scheme for the approximate equation \eqref{eq:check-u}. Motivated by the above discussion,  we set
\begin{equation}	
	\check v(\varrho, t):=e^{\lambda_0\varrho} \check u(\varrho,t),
\end{equation}
which is a solution to 
\begin{equation}\label{eq:checkv-step3}
    \check v_{t}=\check v_{\varrho\varrho}-\tfrac{k}{t}(\check v_{\varrho}-\lambda_0 \check v)+ e^{\lambda_0\varrho} f(e^{-\lambda_0\varrho}v)-f'(0)v,
\end{equation}
and  
\begin{equation}\label{equation-check-w}
    \check{w}(\eta,\tau):= t^{-1/2}\check v(\eta t^{1/2},t).
\end{equation}
Then, the approximation $\check w$ satisfies 
\begin{equation}\label{checkw with L}
    \check w_{\tau}+\mathcal{L}\check w=\check F, 	
\end{equation}
for some right-hand side $\check F$ that is expected to be small. If we take $k$ satisfying \eqref{k}, then $\mathcal{L}$ is the Fokker-Planck operator 
\[
    \mathcal{L}\check w=-\check w_{\eta\eta}-\tfrac{\eta}{2}\check w_{\eta}-\check{w},
\]
which can be made self-adjoint through the standard change of variables
\begin{equation}\label{w-to-omega}	
    \check \omega:=e^{-\frac{\eta^2}{8}}{\check w}.
\end{equation}
Indeed, a straightforward computation shows that, up to a small right-hand side,
\begin{equation}
    \check \omega_{\tau}+\mathcal{M}\check \omega=0,\qquad\mathcal{M}\omega:=-\omega_{\eta\eta}+\big(\tfrac{\eta^2}{16}-\tfrac{3}{4}\big)\omega.
\end{equation}
Notice that $\mathcal{M}$ is self-adjoint, with an associated quadratic form in $H^{1}_0(\mathbb{R}_+, \eta\,{\rm d}\eta)$. This yields the general structure of solutions of \eqref{checkw with L} in terms of eigenfunction expansions.

Next, motivated by \eqref{Dirichlet-w-motivate}, we complement equation \eqref{checkw with L} with an homogeneous Dirichlet condition at $\tau=0$. Therefore, we expect $w$ to be close for large times to some solution $\check w$ to the problem in a half-line
\begin{equation}\label{arriving-to-dipole}
    \begin{cases}	
        \check w_{\tau}+\mathcal{L}\check w=0,&\eta> 0, \ \tau >0,\\
        \check w(0,\tau)=0,&\tau >0.
    \end{cases}
\end{equation}
It turns out that the large-time behavior for this latter problem is well known: 
\begin{equation}\label{checkw-to-dipole}	
    \check{w}(\eta,\tau)\to\alpha_\infty\phi_0(\eta)\quad\textrm{as }\tau\to\infty\textrm{ in the region }\{\eta\geq 0\},
\end{equation}
where $\phi_0$ is the so-called dipole profile,  given by
\begin{equation}\label{phi_0-dipole}
    \phi_0(\eta):=\eta e^{-\frac{\eta^2}{4}},
\end{equation}
and $\alpha_\infty$ is a constant depending on the initial datum of $\check{w}$ (we refer to~\cite{JLVDipolo} for the full description of solutions for the Fokker-Planck equation).

\noindent\underline{6. Matching the dipole with the tail of the traveling wave.} Step 5 suggests that $\check w(\eta,\tau)\sim \alpha_\infty\eta$ for $\eta$ small, whence $\check v(\varrho,t)\sim \alpha_\infty \varrho$, so that, if all our approximations are true,
\begin{equation}\label{behavior_u_coordenasmoviles}
u(\varrho+R(t), t)\sim  \alpha_\infty\varrho e^{-\lambda_{0}\varrho}
\end{equation}
in the diffusive region for $t$ large enough. We would like to find the exact translation $\xi$ of the traveling wave that matches the behavior~\eqref{behavior_u_coordenasmoviles} at $\varrho+\rho_0=t^\gamma$. 
By~\eqref{behaviorTw}, if $\varrho\to\infty$,  the traveling wave satisfies
\begin{equation}\label{behaviorTWcoordenasmoviles}	
    \Phi_{c_0}(\varrho+\xi)\sim (\varrho+\xi+\kappa) e^{-\lambda_{0}(\varrho+\xi)}.
\end{equation}
Comparing~\eqref{behavior_u_coordenasmoviles} and~\eqref{behaviorTWcoordenasmoviles}, we see that in order to perform the matching we need, up to lower order terms,
\begin{equation*}
    \xi=\rho_0-\beta,\quad\text{with }\beta=\tfrac{1}{\lambda_0}\log\alpha_\infty.
\end{equation*}
As mentioned in Step 3, it turns out that the behavior at the right  endpoint  $\varrho=\rho_0+t^\gamma$ controls the whole diffusive region. Thus we will be able to conclude that
\begin{equation*}
    u(\rho,t)\sim\Phi_{c_0}\big(\rho-[c_*t-\tfrac{3}{c_0}\log t+\beta]\big),
\end{equation*}
as desired.

\noindent\underline{7. Control of the error.} For this argument to be rigorous, we need to control the error terms that appear when we approximate $u$ by $\check u$. In particular, we need to estimate the right-hand side that appears in equation \eqref{arriving-to-dipole}. To this aim,  we set
\begin{equation*}
    \omega:=e^{\frac{\eta^2}{8}} w,
\end{equation*}
as in the approximated \eqref{w-to-omega}. Then $\omega$ is a solution to
\begin{equation*}
    \begin{cases}
        \omega_{\tau}+\mathcal{M}\omega=\ell^{i}_1(\eta,\tau)\omega_{\eta}+\ell^{i}_2(\eta,\tau)\omega+ \mathcal F(\eta,\tau,\omega),&\eta  e^{\tau/2}\geq -R(e^\tau),\ \tau>0,\\
        \omega(\eta,0)=e^{\lambda_0(\eta)+\frac{\eta^2}{8}}{u}_{0}(\eta+R(1)),&\eta\geq-R(1),
	\end{cases}
\end{equation*}
with Dirichlet condition $\omega(0,\tau)=w(0,\tau)$, where
\begin{equation}\label{defi-coefficients}
    \begin{split}
        \mathcal F(\eta,\tau,\omega)&:=e^{\tau/2+\frac{\eta^2}{8}+\lambda_0(\eta) e^{\tau/2}}f\big(e^{\tau/2-\frac{\eta^2}{8}-\lambda_0(\eta) e^{\tau/2}}\omega\big)-e^{\tau}f'(0)\omega,\\
    	\ell_{1}^{i}(\eta,\tau)&:=e^{\tau/2}\mathcal H^{i}(\eta e^{\tau/2},e^{\tau})-ke^{-\tau/2},\\
    	\ell_{2}^{i}(\eta,\tau)&:=-\lambda_0\mathcal H^{i}(\eta e^{\tau/2},e^{\tau})-\tfrac{\eta}{4}l^{i}_1(\eta,\tau)+\lambda_0 k-\tfrac{3}{2}.
    \end{split}
\end{equation}

If $\mathcal F$, $\ell_{1}^{i}$ and $\ell_{2}^{i}$ go to 0 fast enough as $\tau\to\infty$, we expect $\omega$ to behave like the approximate~$\check\omega$ for large times. For this, we need to choose $k$ satisfying \eqref{k}, since otherwise $\ell_{2}^{i}$ will not even vanish asymptotically.

There were some previous results in the literature for equations of the form
\begin{equation}\label{eq6}
	\begin{cases}
        \omega_{t}+\mathcal{M}\omega=-\varepsilon e^{-\frac{\tau}{2}}(\omega_{\eta}-\frac{\eta}{4}\omega),\quad\eta\geq0,\ \tau>0,\\
        \omega(0,\tau)=0,\quad\tau>0.
	\end{cases}
\end{equation}  
with $\varepsilon>0$ small enough and a fast decaying initial datum; see the proof of \cite[Lemma 2.2]{NolenRoquejofre}. Nevertheless, in~\cite{RoquejofferNdimensional}, the authors develop an extensive theory of convergence for solutions of some version of~\eqref{eq6} with more general error terms in the right-hand side, that we will be able to adapt for our problem in the hyperbolic space. \\

%%%%%%%%%%%%%%%%%%%%%%%%%%%%%%%%%%%%%%%%%%%%%%%%%%%%%%%%%%%%%%%%%%%%%%%%%%%%%
%%%%%%%%%%%%%%%%%%%%%%%%%%%%%%%%%%%%%%%%%%%%%%%%%%%%%%%%%%%%%%%%%%%%%%%%%%%%%
\section{Traveling-wave behavior: Details of the proof of Theorem~\ref{ThConvergenceTravellingWave}}\label{section:details}
\setcounter{equation}{0}

Once we have understood the main challenges,  we are able to complete the proof of  Theorem~\ref{ThConvergenceTravellingWave} and show convergence to a traveling wave. We will employ the  techniques developed in~\cite{NolenRoquejofre, NolenRoquejofre2, RoquejofferNdimensional}, taking into account the modifications introduced by the geometry.

As mentioned above, to capture the dynamics of $u$, we will consider three different regions, as we described in Step 3 from Section~\ref{section:strategy}: the diffusive region, the superdiffusive and the subdiffusive one, which will be considered in Subsections \ref{subsection:diffusive}, \ref{subsection:superdiffusive} and \ref{subsection:subdiffusive} below, respectively.

More precisely, we will be able to follow the same strategy as in~\cite{RoquejofferNdimensional}, constructing, in Subsection~\ref{subsect:sub.super}, sub and supersolutions in a translated variable. The reason behind is that the equation that $w$ satisfies is very close to the one given by~\cite[equation (12)]{RoquejofferNdimensional}, if the error terms are small. We devote Subsection~\ref{subsect:control.errors} to control the errors. 

%%%%%%%%%%%%%%%%%%%%%%%%%%%%
\subsection{Control of the errors}
\label{subsect:control.errors}

To write precisely the region $ \pm o(\sqrt t)$ introduced  above,  we set
\begin{equation}\label{eta_delta}	
    \delta\in (0,\tfrac{1}{4}),\quad\eta_{\delta}\in\{\eta_{\delta}^{-},0,\eta_{\delta}^{+}\}\quad\textrm{with } \eta_{\delta}^{\pm}=\eta_{\delta}^{\pm}(\tau):=\pm e^{-(\frac{1}{2}-\delta)\tau}.
\end{equation}
Observe that both $\eta^{\pm}_{\delta}(\tau)$ vanish exponentially fast as $\tau\to\infty$. 

Consider the translations 
\begin{equation*}
    \tilde{w}(\eta,\tau):={w}(\eta+\eta_{\delta},\tau).
\end{equation*}
Now $\tilde w$ solves
\begin{equation}
	\begin{cases}
		\label{tilde{w}- equation}	\tilde{w}_{\tau}+\mathcal{L}\tilde{w}=\big((\delta-\frac{1}{2})\eta_{\delta}+e^{-\tau/2}H^{i}(\eta+\eta_{\delta},\tau)\big)\tilde{w}_{\eta} &\\
		\hspace{20mm}-(\lambda_0H^{i}(\eta+\eta_{\delta},\tau)+\frac{3}{2})\tilde{w}+{F}(\eta+\eta_{\delta},\tau,\tilde{w}),&\eta  e^{\tau/2}+\eta_{\delta}\geq -R(e^\tau),\;\;\tau>0,\\
		\tilde{w}(\eta,0)=\tilde w_0(\eta),&\eta+\eta_\delta\geq-R(1),
	\end{cases}
\end{equation}
where we have defined
\begin{equation}\label{tilde-w0}
    \tilde w_0(\eta):=e^{\lambda_0(\eta+\eta_{\delta})}{u}_{0}(\eta+\eta_{\delta}(0)+R(1)).
\end{equation}

Motivated by the change \eqref{w-to-omega} and the discussion above, we set
\begin{equation}\label{w-definition}	
    \tilde\omega:=e^{\frac{\eta^2}{8}}\tilde{w},
\end{equation}
 which is a solution of
\begin{equation}\label{equation-omega-translated}
    \begin{cases}
        \tilde\omega_{t}+\mathcal{M}\tilde\omega=\ell^{i}_1(\eta,\tau)\tilde\omega_{\eta}+\ell^{i}_2(\eta,\tau)\tilde\omega+ \mathcal F(\eta,\eta_{\delta},\tau,\tilde\omega),&\eta  e^{\tau/2}+\eta_{\delta}\geq -R(e^\tau),\ \tau>0,\\
        \tilde\omega(\eta,0)=e^{\frac{\eta^2}{8}}\tilde w_0(\eta),
	\end{cases}
\end{equation}
where
\begin{equation}\label{defi-coefficients-tilde}
    \begin{split}
        \mathcal F(\eta,\eta_{\delta},\tau,\tilde\omega)&:=e^{\tau/2+\frac{\eta^2}{8}+\lambda_0(\eta+\eta_{\delta}) e^{\tau/2}}f\big(e^{\tau/2-\frac{\eta^2}{8}-\lambda_0(\eta+\eta_{\delta}) e^{\tau/2}}\tilde\omega\big)-e^{\tau}f'(0)\tilde\omega,\\
        \ell_{1}^{i}(\eta,\eta_{\delta},\tau)&:=(\delta-\tfrac{1}{2})\eta_{\delta}(\tau)+e^{\tau/2}\mathcal H^{i}((\eta+\eta_\delta) e^{\tau/2},e^{\tau})-ke^{-\tau/2},\\
        \ell_{2}^{i}(\eta,\eta_{\delta},\tau)&:=-\lambda_0e^{\tau}\mathcal H^{i}((\eta +\eta_\delta) e^{\tau/2},e^{\tau})-\tfrac{\eta}{4}\ell^{i}_1(\eta,\tau)+\lambda_0 k-\tfrac{3}{2}.
    \end{split}
\end{equation}
With the choice of $k$ as in \eqref{k} we will prove that the error terms $\ell^i_j$, $j=1,2$, vanish as $\tau$ goes to infinity. To this aim we have to control the  curvature factor $\mathcal H^i$, given in \eqref{h}, showing that it goes to zero fast enough to compensate the exponential in front of it.

\begin{lemma}\label{lemma:bound-H}
    For $\rho_0$ large enough, we have the estimate
    \begin{equation*}
        e^{3\tau/2}\|\mathcal H^{i}(\eta e^{\tau/2},e^\tau)\|_{L^\infty(0,\infty)} \to 0\quad \text{as }\tau\to\infty.
    \end{equation*}
    for all $i\in \{\mathfrak{e},\mathfrak{h},\mathfrak{p}\}$.
\end{lemma}

\begin{proof}
Recall the definition \eqref{h}, set 
\begin{equation*}    
    \sigma(\eta,\tau):=\eta e^{\tau/2}+R(e^\tau),
\end{equation*}
and take $\rho_0$ large enough such that $\sigma(\eta,\tau)\geq 1$ for all $\eta\geq 0$ and $\tau>0$. Since
\begin{equation}\label{g-decay}	
    |h^{i}_{1}(\varrho)-1|=
	\begin{cases}
		\frac{2}{e^{2\varrho}-1},&i=\mathfrak{e},\\
		\frac{2}{e^{2\varrho}+1},&i=\mathfrak{h},\\
		0,&i=\mathfrak{p},
	\end{cases}
\end{equation}
for $\varrho>0$, then we  have $|h^{i}_{1}(\sigma(\eta,\tau))-1|\leq \frac{2}{e^{2}-1}$ and
\begin{equation}\label{behavior g1}
    e^{3\tau/2}	\|h^{i}_{1}(\sigma(\cdot,\tau))-1\|_{L^{\infty}(0,\infty)}\to 0\quad\textrm{as }\tau\to \infty\quad\text{for all }i\in \{\mathfrak{e},\mathfrak{h},\mathfrak{p}\},
\end{equation}
which immediately yields the desired exponential decay for $\mathcal H^i$.
\end{proof}

Additionally, we need to modify slightly this lemma by taking the translated 
\begin{equation*}    
    \sigma(\eta,\eta_{\delta},\tau):=(\eta+\eta_{\delta})e^{\tau/2}+R(e^\tau)\geq-e^{\tau/4}+ c_*e^{\tau}-k\tau+\rho_0\geq1
\end{equation*}
for all $\eta\geq 0$ and $\tau>0$. Then we have the same conclusion
\begin{equation}\label{behavior g}
    e^{3\tau/2}	\sup\limits_{\delta\in (0,\frac{1}{4})}\|h^{i}_{1}(\sigma(\cdot,\eta_{\delta}(\tau),\tau))-1\|_{L^{\infty}(0,\infty)}\to 0\quad\textrm{as }\tau\to \infty\quad\text{for all }i\in \{\mathfrak{e},\mathfrak{h},\mathfrak{p}\}.
\end{equation}
As a consequence, there exists $C>0$ depending only on  $d$, $\delta$ and $f'(0)$ such that
\begin{equation}\label{CotaErrores-li}
	|\ell^{i}_{1}(\eta,\eta_{\delta},\tau)|\leq Ce^{-(\frac{1}{2}-\delta)\tau},\quad
	|\ell^{i}_{2}(\eta,\eta_{\delta},\tau)|\leq C(1+\eta)e^{-(\frac{1}{2}-\delta)\tau}
\end{equation}
for all $\eta\geq 0$, $\tau>0$, $\eta_{\delta}\in\{\eta_{\delta}^{-},0,\eta_{\delta}^{+}\}$ with $\delta\in (0,\frac{1}{4})$ and  $i\in\{\mathfrak{e},\mathfrak{h},\mathfrak{p}\}$. 

\begin{remark}
    Observe that the bounds obtained in~\eqref{CotaErrores-li} for $\ell^{i}_{1}$ and $\ell^{i}_{2}$, for all $i\in\{\mathfrak{e},\mathfrak{h},\mathfrak{p}\}$ improve  the ones in~\cite[eq. (19)]{RoquejofferNdimensional} for their counterparts in the Euclidean case, defined in  \cite[transformation~\textsc{7}, section 2]{RoquejofferNdimensional}. 
\end{remark}

\begin{remark}\label{rem:always.same.correction}
    The parameter $k$, given in ~\eqref{k}, does not depend on the dimension $d$, neither on the type of isometry $i\in\{\mathfrak{e},
    \mathfrak{h},\mathfrak{p}\}$. This is an important difference with the KPP problem in the  Euclidean space, where the parameter $k$ depends on the isometry. Specifically, $k=\frac{d+2}{c_0}$ in the  elliptic case and $k=3/c_0$ in the parabolic one. 
\end{remark}

%%%%%%%%%%%%%%%%%%%%%%%%%%%%
\subsection{Sub and supersolutions} 
\label{subsect:sub.super}

Now we are able to give sharp bounds for $w$. We will construct a  subsolution in the region $\{\eta\geq\eta^{+}_\delta(\tau)\}$, i.e., a lower barrier where $\rho$ is  far ahead of the front of the traveling wave; and a supersolution in the region $\{\eta\geq\eta^{-}_\delta(\tau)\}$, that is, where $\rho$ is far behind the rear of the front of $\Phi_{c_0}$. 

For this, consider the translated $\tilde w$ in the notation from \eqref{eta_delta}. In the translated variables, we need to study the region $\{\eta>0\}$.

It is also convenient to switch to the notation $\tilde \omega$ from \eqref{w-definition}	and study the equation
\begin{equation}\label{EquationM}
    \tilde\omega_{\tau}+\mathcal{M}\tilde\omega=\ell^{i}_1(\eta,\tau)\tilde\omega_{\eta}+\ell^{i}_2(\eta,\tau)\tilde\omega+ \mathcal F(\eta,\eta_{\delta},\tau,\tilde\omega), \quad\eta>0.
\end{equation}
Specifically, we will obtain  two functions $\overline{\omega}$ and $\underline{\omega}$ such that
\begin{align*}
	&\overline{\omega}_{\tau}+\mathcal{M}\overline{\omega}\geq \ell^{i}_1(\eta,\tau)\overline{\omega}_{\eta}+\ell^{i}_2(\eta,\tau)\overline{\omega}+ \mathcal F(\eta,\eta_{\delta},\tau,\overline{\omega}),\\
	&\underline{\omega}_{\tau}+\mathcal{M}\underline{\omega}\leq \ell^{i}_1(\eta,\tau)\underline{\omega}_{\eta}+\ell^{i}_2(\eta,\tau)\underline{\omega}+ \mathcal F(\eta,\eta_{\delta},\tau,\underline{\omega}),
\end{align*}
in $\{\eta\geq 0\}$  such that $\overline{\omega}(0,\tau)$ and $\underline{\omega}(0,\tau)$ vanish as $\tau\to\infty$. Recall that we are choosing $\eta_\delta=\eta_\delta^{-}$ for $\overline{\omega}$ and $\eta_\delta=\eta_\delta^{+} $ for $\underline{\omega}$, in the notation from~\eqref{eta_delta}.

The following proposition, that we borrow from~\cite{RoquejofferNdimensional}, yields a supersolution $\overline{\omega}$ and a subsolution~$\underline{\omega}$ that are comparable to $e^{\eta^2/8}\phi_0$ as $\tau\to\infty$. This proposition also holds in our case because $\ell^{i}_{1},  \ell^{i}_2$ satisfy~\eqref{CotaErrores-li} for $i\in\{\mathfrak{e},\mathfrak{h},\mathfrak{p}\}$.

\begin{proposition}\label{PropostionSubAndSuperFor M} \textup{(\cite[Proposition 1]{RoquejofferNdimensional})}
    Let $a\in (0,1)$, $\gamma_{1}(\eta)$ a nonnegative smooth function, identically  equal to $1$ if $\eta\leq a/2$ and zero if $\eta\geq 2a/3$, and $\gamma_2(\eta)$ be a nonnegative function, identically zero if $\eta\leq 1$ and equal to $1$ if $\eta\geq 2$. Then,
    \begin{equation*}\label{SupersolutionM}
        \overline{\omega}(\eta,\tau):=\xi^{+}(\tau)e^{\eta^2/8}\phi_0(\eta)+q^{+}_1(\tau)\gamma_{1}(\eta)\cos(\tfrac{\pi}{2a}\eta)
        +q^{+}_2(\tau)\gamma_{2}(\eta)e^{-\frac{\eta^2}{16}}
    \end{equation*}
    is a supersolution to~\eqref{EquationM} with $\eta_{\delta}=\eta_{\delta}^{-}$ in the range $\tau>0$, $\eta>0$, whereas the function
    \begin{equation*}\label{SubsolutionM}
        \underline{\omega}(\eta,\tau):=\xi^{-}(\tau)e^{\eta^2/8}\phi_0(\eta)-q^{-}_1(\tau)\gamma_{1}(\eta)\cos(\tfrac{\pi}{2a}\eta)
        -q^{-}_2(\tau)\gamma_{2}(\eta)e^{-\frac{\eta^2}{16}}
    \end{equation*}
    is a subsolution to~\eqref{EquationM} with $\eta_{\delta}=\eta_{\delta}^{+}$ in the range $\tau>\tau_1$, $\eta>0$ for some $\tau_1>0$.
    Here  we can choose, for any given constants  $C_0^-,C_0^+>0$, functions $q^{\pm}_{1}(\tau)$, $q^{\pm}_{2}(\tau)$ and $\xi^{\pm}(\tau)$  that satisfy,   
    \begin{equation}\label{estimate:q}
        \begin{split}
            &C_0^+e^{-\tau}\leq q_{j}^{+}(\tau)\leq Ce^{-(\frac{1}{2}-\delta)\tau},\quad 0\leq q^{-}_{j}(\tau)\leq C_0^-e^{-(\frac{1}{2}-\delta)\tau}\quad\textrm{for all }\tau>0\textrm{ and }j\in\{1,2\},\\
            &\xi^{\pm}(\tau)\in[\underline{\xi}_0,\overline{\xi}_0],\quad \partial_{\tau}\xi^{+}(\tau)>0,\quad \partial_{\tau}\xi^{-}(\tau)<0\textrm{ for all }\tau>0,
        \end{split}
    \end{equation}
    for some $C>1$ and $\overline{\xi}_0\geq\underline{\xi}_0>0$.     
\end{proposition}

Before we continue, let us describe the approximate Dirichlet condition at $\eta_\delta^\pm$. Since
\[
    {w}(\eta,\tau):=e^{-\tau/2}e^{\lambda_0\eta e^{\tau/2}} \hat u(\varrho,t),
\]
and $u$ is bounded, then
\begin{equation}\label{Dirichlet2}
    \begin{aligned}
        w(\eta_{\delta}^\pm(\tau),\tau)&:=e^{-\tau/2}e^{\pm\lambda_0e^{\delta\tau} }\hat u(e^{\tau/2}\eta_\delta^{\pm}(\tau),e^{\tau})\\
        &=e^{-\tau/2}e^{\pm\lambda_0e^{\delta\tau} } u(\pm e^{\delta\tau}+R(e^\tau),e^\tau)=O(e^{-\tau/2}e^{\pm\lambda_0e^{\delta\tau} })\quad\text{as }\tau\to\infty.
    \end{aligned}
\end{equation}

Now we use the comparison principle with the sub/supersolutions we have just found:

\begin{proposition}\label{Sub/Super for hat{w}}\textup{(Adaptation of~\cite[Proposition 5 and Corollary 1]{RoquejofferNdimensional})}	 
    Let $f'(0)>\lambda_{1}$ and $u_0$ satisfying~\eqref{eq:symmetry.u0}--\eqref{eq:supported-from-right} for some~$i\in\{\mathfrak{e},\mathfrak{h},\mathfrak{p}\}$. Let  ${w}$ be the solution from~\eqref{hatw-transformacion4a}. Let us use the same notations as in  Proposition~\ref{PropostionSubAndSuperFor M}.
    
    \noindent\textup{(i)} (Control of ${w}$ from above and from the rear of the front) We have that
    \begin{equation}\label{super-hat{w}}	
        {w}(\eta+\eta_{\delta}^{-},\tau)\leq\xi_{+}(\tau)\phi_0(\eta)+q^{+}(\tau)\big(\mathds{1}_{[0,a_0/2]}(\eta)\cos(\tfrac{\pi}{2a_0}\eta)+e^{-\frac{\eta^2}{16}}\big)e^{-\frac{\eta^2}{8}},\quad \eta\geq 0,\ \tau>0.
    \end{equation}
        
    \noindent\textup{(ii)} (Control of ${w}$ from below and from the head of the front) In addition,
    \begin{equation}\label{sub-hat{w}} 
        {w}(\eta+\eta_{\delta}^{+},\tau)\geq \xi_{-}(\tau)\phi_0(\eta)-q^{-}(\tau)\big(\mathds{1}_{[0,a_0/2]}(\eta)\cos(\tfrac{\pi}{2a_0}\eta)+e^{-\frac{\eta^2}{16}}\big)e^{-\frac{\eta^2}{8}},\quad \eta\geq 0,\ \tau>\tau_1.
    \end{equation}
    Here the functions $q^{\pm}$ satisfy the same estimates as $q_i^\pm$ from \eqref{estimate:q}.
\end{proposition}

\begin{proof}
This is essentially \cite[Proposition 5]{RoquejofferNdimensional}, so here we just explain the necessary modifications.

We first observe that, since  $h^{i}_{1}(\cdot+\rho_0)$ is in $L^\infty([-1,\infty))$ for $\rho_0\geq 2$, for all $i\in\{\mathfrak{e},\mathfrak{h},\mathfrak{p}\}$ (here we can see the use of the constant $\rho_{0}$ in order to obtain that $\coth\big((\eta+\eta_{\delta})e^{-\tau/2}+c_*e^{\tau}-k\tau+\rho_0\big)  $ is uniformly bounded  for all $\eta\geq 0$ and $\tau>0$), then the coefficients are smooth and we may use the standard comparison principle for~\eqref{EquationM}.

Now, to obtain \eqref{super-hat{w}}, we compare the supersolution $\overline{\omega}(\eta)$ given by Proposition~\ref{PropostionSubAndSuperFor M} with $\omega(\eta+\eta_\delta^-)$ in the region  $\{\eta\geq 0\}$. For this, we need to match the boundary value at $\eta=0$. But this follows by comparing with the  Dirichlet data for $\omega$, which results from the behavior of $w(\eta_{\delta}^-(\tau),\tau)$  from~\eqref{Dirichlet2}, taking  into account that $\overline{\omega}(0,\tau)=q^{+}_1(\tau)$ and the lower bound for $q^+_1$ from Proposition~\ref{PropostionSubAndSuperFor M}. As for the initial conditions, just note that 
\[
    {w}(\eta+\eta_\delta^+(0),0)=e^{\lambda_0(\eta+1)} \hat u(\eta+1,1),
\]
and that $u_0$ is compactly supported from the right (recall relation \eqref{utransladadaR(t)}). Hence, it is enough to take $C_0$ large enough in Proposition~\ref{PropostionSubAndSuperFor M} in the construction of $q_1^+$.

To obtain \eqref{sub-hat{w}}, we compare the subsolution $\underline \omega(\eta)$ from Proposition~\ref{PropostionSubAndSuperFor M} with $\omega(\eta+ \eta_\delta^+)$ in the region $\{\eta\geq 0\}$. On the one hand, for the Dirichlet conditions at $\eta=0$ we know that $\underline \omega\leq 0$ by construction, while $w\geq 0$. As for the initial data at $\tau=\tau_1$, $w(\eta+\eta_\delta^+,\tau_1)$ is trivially positive. On the other hand, from the two terms in $\underline\omega(\eta,\tau_1)$, the first one is positive and the other two negative. But taking $C_0^-$ small enough the first term wins the game, so that $\underline\omega(\eta,\tau_1)\leq 0$ for all $\eta\geq 0$.
\end{proof}

\begin{corollary}
In the hypothesis of the previous proposition, there exists $C>0$ such that   \begin{equation}\label{CotasDerivadas-hat{w}}	
        |{\tilde w}_{\tau}(\eta,\tau)|+|{\tilde w}_{\eta}(\eta,\tau)|\leq Ce^{-\frac{3\eta^2}{16}}\quad\textrm{for all }\eta\geq 0,\ \tau>\tau_1.
    \end{equation}
\end{corollary}

\begin{proof}
This is analogous to \cite[Corollary 1]{RoquejofferNdimensional}, using parabolic regularity. 
\end{proof}

As a consequence of Proposition~\ref{Sub/Super for hat{w}},  we know that there is a constant $C>1$ such that
\begin{equation}\label{estimate-hat-w}
	C^{-1}\phi_0(\eta-\eta^+_\delta)\leq {w}(\eta,\tau)\leq C\phi_0(\eta-\eta^-_\delta)\quad\text{ for } \tau\textrm{ large}, 
\end{equation}
for $\eta\ge \eta_\delta^-$ for the upper bound and $\eta\ge \eta_\delta^+$ for the lower one. This is far from a convergence result, since these estimates do not even determine the right multiple of $\phi_0$ that one should obtain in the limit. However, Proposition~\ref{Sub/Super for hat{w}} is a first step towards the desired result, since it will give the compactness of the trajectories $\{{w}(\cdot,\cdot+T)\}_{T>0}$ in  a spatial  weighted $L^{\infty}$-norm. In addition, it also yields the asymptotic Dirichlet condition $w(0,\tau)\to 0$ as $\tau\to\infty$.

\begin{proposition}(Adaptation of \cite[Theorem 2.1]{RoquejofferNdimensional} and \cite[Lemma 5.1]{NolenRoquejofre2})\label{Prop-alpha-infty}
    Consider the same assumptions of Proposition~\ref{Sub/Super for hat{w}}. Then,  there exists a positive constant $\alpha_\infty$ such that
    \begin{equation}\label{hat w convergence to dipole-limit}
        \lim_{\tau\to\infty}\sup_{\eta>0}e^{\frac{\eta^2}{16}}\big|{w}(\eta,\tau)-\alpha_\infty\phi_0(\eta)\big|=0.
    \end{equation}
\end{proposition}

\begin{proof} 
We will follow the ideas of the proof of \cite[Lemma 5.1]{NolenRoquejofre2}. Consider  a sequence of times  $\{\tau_n\}_{n\in\mathbb{N}}$ such that $\tau_n\to\infty$ as $n\to\infty$. By the barriers~\eqref{super-hat{w}} and~\eqref{sub-hat{w}}, the bounds for the derivatives~\eqref{CotasDerivadas-hat{w}} and using a standard compactness argument, there exist a subsequence $\{\tau_{n'}\}_{n'\in\mathbb{N}}\subset\{\tau_n\}_{n\in\mathbb{N}}$ and two functions ${w}_\infty: [0,\infty)\times\mathbb{R}_+\to\mathbb{R}$ and ${W}_\infty: [0,\infty)\to\mathbb{R}$  such that
\begin{equation}\label{eq:limits.subsequences}	
    \lim\limits_{n'\to\infty}\sup\limits_{\eta\geq0,\,\tau\in[0,T]}\,e^{\frac{\eta^2}{16}}\big|{w}(\eta,\tau+\tau_{n'})-{w}_{\infty}(\eta,\tau)\big|=0,\quad\lim\limits_{n'\to\infty}\sup\limits_{\eta\geq0}\;e^{\frac{\eta^2}{16}}\big|{w}(\eta,\tau_{n'})-{W}_{\infty}(\eta)\big|=0
\end{equation}
for all $T>0$. Moreover, ${w}_\infty$ solves
\begin{equation*}
	\begin{cases}
		({w}_\infty)_{\tau}+\mathcal{L}{w}_\infty=0,&\eta\geq0,\ \tau>0,\\
        {w}_\infty(0,\tau)=0,&\tau>0,\\
        {w}_\infty(\eta,0)={W}_\infty(\eta),&\eta\geq0.
	\end{cases}
\end{equation*}
Then,
\begin{equation}\label{eq:limit.dipole}	
    \lim\limits_{\tau\to\infty}\sup\limits_{\eta\geq 0}\,e^{\frac{\eta^2}{16}}\big|{w}_{\infty}(\eta,\tau)-\alpha_\infty\phi_0(\eta)\big|=0,
\end{equation}
where $\alpha_\infty:=\int_0^\infty \eta{W}_{\infty}(\eta)\,{\rm d}\eta$ (this is a classical result for the Fokker-Planck equation, see, for instance, ~\cite{JLVDipolo}). In addition, by~\eqref{estimate-hat-w}, we have that $\alpha_\infty\in(0,\infty)$.
	
To prove~\eqref{hat w convergence to dipole-limit} we need to show that the convergence does not depend on the choice of  subsequences.
By~\eqref{eq:limit.dipole}, for a given $\varepsilon>0$, there is $T_{\varepsilon}$ big enough  such that
\begin{equation*}
    \big|{w}_{\infty}(\eta,T_{\varepsilon})-\alpha_\infty\phi_0(\eta)\big|\leq \varepsilon e^{-\frac{\eta^2}{16}}\quad\textrm{for all }\eta\geq 0,
\end{equation*}
and, by  the first limit in~\eqref{eq:limits.subsequences}, there is $n'_{\varepsilon}\in\mathbb{N}$ (that depends also in $T_\varepsilon$) such that	
\begin{equation*}
    \big|{w}(\eta,T_{\varepsilon}+\tau_{n'_{\varepsilon}})-{w}_{\infty}(\eta,T_\varepsilon)\big|\leq \varepsilon e^{-\frac{\eta^2}{16}}\quad\textrm{for all }\eta\geq 0.
\end{equation*}
Therefore, putting all together,
\begin{equation}\label{eq5}	
    \big|{w}(\eta,T_{\varepsilon}+\tau_{n'_{\varepsilon}})-\alpha_\infty\phi_0(\eta)\big|\leq 2\varepsilon e^{-\frac{\eta^2}{16}}\quad\textrm{for all }\eta\geq 0.
\end{equation}

Now, we define 
\begin{equation}\label{relation-z-w}
    z(\eta,\tau):={w}(\eta+\eta_{\delta}^{-},\tau+T_{\varepsilon}+\tau_{n'_{\varepsilon}} )-{w}(\eta_{\delta}^{-},\tau+T_{\varepsilon}+\tau_{n'_{\varepsilon}})\varphi(\eta),
\end{equation}
where $\varphi$ is a smooth function such that $\mathds{1}_{[0,1]}\leq \varphi\leq \mathds{1}_{[0,2]}$.
We will show that $z$ satisfies an equation of the form\begin{equation*}
	\begin{cases}		
        \big|z_{\tau}+\mathcal{L}z\big|\leq\tilde\varepsilon e^{-\sigma\tau}\big(|z_{\eta}|+|z|+G(\eta,\tau)\big),&\eta>0,\, \tau>0,\\
		z(0,\tau)=0,&\tau>0,
	\end{cases}
\end{equation*}
for some $G(\eta,\tau):[0,\infty)\times\mathbb{R}_+\to\mathbb{R}$ smooth, bounded and compactly supported in $\eta$, and some $\sigma>0$ to be chosen, and where  $\tilde\varepsilon=o(1)$ as $T_\varepsilon\to\infty$. 
    
We start with the terms coming from the first addend in~\eqref{relation-z-w}. Observe that the translated  $\tilde{w}(\eta,\tau)={w}(\eta+\eta_{\delta},\tau)$ for $\eta_{\delta}=\eta_{\delta}^-$ satisfies \eqref{tilde{w}- equation} with  the nonlinear term ${F}$ given by~\eqref{hat{F}}. 

In order to control the coefficient of the drift term in the right-hand side of~\eqref{tilde{w}- equation}, we 
observe first that $|\eta_\delta^-(\tau)|\le \tilde\varepsilon e^{-\sigma \tau}$ for $\tau$ large if $\sigma\in(0,\tfrac12-\delta)$. On the other hand, by our choice of $k$ from~\eqref{k}, using~\eqref{hat{h}^i},~\eqref{h} and the estimates in  Lemma~\ref{lemma:bound-H},  we get
\begin{equation}\label{eq4}	
    e^{-\tau/2}|H^i(\eta,\tau)|+|\lambda_0H^i(\eta,\tau)+\tfrac{3}{2}|\leq \tilde\varepsilon e^{-\sigma\tau}\quad\textrm{for all }\eta\geq0\text{ if }\tau \text{ is large}.
\end{equation}
This also controls the coefficient of the zero-order term in the right-hand side of~\eqref{tilde{w}- equation}.

As for the term involving ${F}$, noting that $f(0)=0$ ($f$ is a KPP function), by Taylor expansion, there exists a constant $C>0$ such that
\begin{equation}\label{eq3}	
    |{F}(\eta+\eta_{\delta}^{-},\tau,w(\eta+\eta_{\delta}^{-},\tau))|\leq C\exp\big(\tfrac{3}{2}\tau-\lambda_{0}(\eta +\eta_{\delta}^{-})e^{\frac{\tau}{2}}\big){w}^{2}(\eta+\eta_{\delta}^{-},\tau).%\leq  C\exp\big(\tfrac{3}{2}\tau-\lambda_{0}e^{\delta\tau}\big){w}^2.
\end{equation}
By Proposition~\ref{Sub/Super for hat{w}}, $w$ is bounded. Thus we immediately have
\begin{equation*} 
|{F}(\eta+\eta_{\delta}^{-},\tau,w(\eta+\eta_{\delta}^{-},\tau))|\le C\exp\big(\tfrac{3}{2}\tau-\lambda_{0}e^{\delta\tau}\big) w(\eta+\eta_{\delta}^{-},\tau)
\leq \tilde\varepsilon e^{-\sigma\tau}{w(\eta+\eta_{\delta}^{-},\tau)}
\end{equation*}
for all $\eta\geq0$ and $\tau$ large and $\sigma>0$ to be chosen later. 

Next, for the second addend in the definition of $z$ in~\eqref{relation-z-w}, we just need to take into account the fast decay of $w$ at $\eta_\delta^-$ given in \eqref{Dirichlet2} and Corollary~\ref{CotasDerivadas-hat{w}} regarding derivative bounds.

Notice that we will evaluate all expressions at times $\tau+T_{\varepsilon}+\tau_{n'_{\varepsilon}}$, that are indeed large if $\tau>0$. This completes the derivation of \eqref{relation-z-w}.

As for the initial condition, by~\eqref{eq5}, $z(\tau,0)\in L^{2}(\mathbb{R}_+, e^{\frac{\eta^2}{6}}{\rm d}\eta)$. We can hence apply  \cite[Lemma~5.2]{NolenRoquejofre2} to obtain that there exists $Z_\infty:[0,\infty)\to\mathbb{R}$, with 
$(\alpha_{\infty}-C\tilde{\varepsilon})\phi_0\leq Z_\infty\leq (\alpha_{\infty}+C\tilde{\varepsilon})\phi_0$,
such that 
\begin{equation*}
    \lim\limits_{\tau\to\infty}\sup\limits_{\eta\geq 0}\,e^{\frac{\eta^2}{16}}|z(\eta,\tau)-Z_\infty(\eta)|=0.
\end{equation*}
As a consequence, noting that $\tilde{\varepsilon}$ is arbitrary small, $\eta_\delta^+(\tau)\to0$, and, by~\eqref{super-hat{w}}, $w(\eta_{\delta}^{-},\tau+T_{\varepsilon}+\tau_{n^{'}_{\varepsilon}})\to 0$ as $\tau\to\infty$, since $q^+(\tau)=o(1)$, using the definition of $z$ from \eqref{relation-z-w}, the above implies~\eqref{hat w convergence to dipole-limit}, as desired.
\end{proof}

Proposition~\ref{Prop-alpha-infty} gives us the right constant $\alpha_\infty$ in front of the dipole function $\phi_0$. In the following paragraphs we will improve the estimate~\eqref{estimate-hat-w} taking this into account, giving  precise bounds for $w$ near $\eta_\delta^-$.   The proof follows verbatim  the ideas of~\cite[Proposition~7]{RoquejofferNdimensional} once we estimate the curvature terms. Therefore we do not reproduce it here. Let us just mention that the  arguments  are very similar in spirit to those of Proposition \ref{Prop-alpha-infty}.

\begin{proposition}\cite[Proposition 7]{RoquejofferNdimensional}\label{Prop7-RadialEuclideo}
    Consider the assumptions of Proposition~\ref{Sub/Super for hat{w}}. Let $\alpha_\infty>0$  be the constant given by Proposition~\ref{Prop-alpha-infty}. Then, for every $\varepsilon>0$ there exists $\tau_{\varepsilon}>0$ (possibly depending also on $\delta$) and $\tilde{\eta}_{\varepsilon}>1$ such that
    \begin{equation*}
        (\alpha_{\infty}-C\varepsilon)(\eta-Ce^{-(\frac{1}{2}-\delta)(\tau-\tau_{\varepsilon)}})\leq{w}(\eta,\tau)\leq (\alpha_{\infty}+C\varepsilon)(\eta+Ce^{-(\frac{1}{2}-\delta)(\tau-\tau_{\varepsilon)}}),\quad \tau\geq\tau_{\varepsilon},\  \eta\in[\eta_{\delta}^{-},\tilde{\eta}_\varepsilon].
    \end{equation*}
\end{proposition}

%%%%%%%%%%%%%%%%%%%%%%%%%%%%
\subsection{Diffusive region}\label{subsection:diffusive}
%Note that this region is where $|\rho|\leq o(\sqrt{t})$ in moving coordinates. 

Let us first give a technical, but crucial lemma to show that  the difference $s$ between  a shift of the traveling wave of minimal speed and the solution vanishes as $t$ goes to infinity in the diffusive region if we have exact matching at the right end  and a good estimate at the left edge.

\begin{lemma}\label{LemmaS}
	Let $\gamma\in (0,1/2)$, $\mathcal H^i$ given by~\eqref{h} and $s:\mathbb{R}\times\mathbb{R}_{+}\to\mathbb{R}$ be a locally bounded solution to
	\begin{equation}\label{equationProp4}
		\begin{cases}
			s_{t}-s_{\varrho\varrho}-\big(\mathcal H^{i}(\varrho,t)-\frac{k}{t}\big)(s_{\varrho}-\lambda_0 s)=O(t^{-(1-2\gamma)}),\quad&|\varrho+\rho_0|\leq t^{\gamma},\ t>t_0,\\
			s(-t^{\gamma}-\rho_0,t)=O(e^{-\lambda_0t^{\gamma}}-\rho_0),\quad s(t^{\gamma},t)=0,&t>t_0,
		\end{cases}
	\end{equation}
    for some $t_0$ large enough.
	Then, 
    $$
        \displaystyle\lim\limits_{t\to\infty}\sup\limits_{ |\varrho+\rho_0|\leq t^{\gamma}}|s(\varrho,t)|=0.
    $$
\end{lemma}

\begin{proof}
Set $\bar\varrho=\varrho+\rho_0$, and rewrite the problem \eqref{equationProp4} in the new variable $\bar\varrho$, with some abuse of notation.

Following ideas from the proof of~\cite[Proposition 4.1]{NolenRoquejofre2}, we consider a supersolution to~\eqref{equationProp4} of the form
\begin{equation}\label{sdef}	
    \bar{s}(\bar\varrho,t)=t^{-\lambda}\cos\big(\frac{\bar\varrho}{t^{\gamma}}\big)>0,\quad |\bar\varrho|\leq t^{\gamma},\ t>t_0,
\end{equation}
where $\lambda>0$.  It holds that
\begin{equation*}
	-\bar{s}_{\bar\varrho\bar\varrho}(\bar\varrho,t)=t^{-2\gamma}\bar{s}(\bar\varrho,t),\quad
	\bar{s}_{t}(\bar\varrho,t)\geq -(\lambda+1) t^{-1}\bar{s}(\bar\varrho,t)\quad\textrm{for all }|\bar\varrho|\leq t^{\gamma},\;\; t\geq 1.
\end{equation*}
Observe that, $|\bar{s}_{\bar\varrho}|\le C\bar{s}$ for $|\bar\varrho|\leq t^{\gamma}$. Arguing similarly as in Lemma~\ref{lemma:bound-H}, we can bound $\mathcal H^i(\bar\varrho-\rho_0,t)$ in the region $|\bar\rho|\leq t^{\gamma}$,  $t>t_0$ if we take $t_0$ large enough. From here we obtain
\begin{equation*}
    \left|(\mathcal H^i(\bar\varrho-\rho_0,t)-\tfrac{k}{t})(\bar{s}_{\bar\varrho}-\lambda_0\bar{s})\right|\leq Ct^{-1}\bar{s}\quad\textrm{for all }|\bar\varrho|\leq t^{\gamma},\ t\geq t_0.
\end{equation*}
Hence, if $\gamma\in (0,1/2)$, there exists $t_0>1$ such that
\begin{equation*}
	\bar{s}_{t}-\bar{s}_{\bar\varrho\bar\varrho}-\big(\mathcal H^{i}(\bar\varrho-\rho_0,t)-\tfrac{k}{t}\big)(\bar{s}_{\bar\varrho}-\lambda_0 \bar{s})
	\geq\big(t^{-2\gamma}-(\lambda+1+C)t^{-1}\big) \bar{s}\geq Ct^{-1}\bar{s}\quad\textrm{for all }|\bar\varrho|\leq t^{\gamma},\ t\geq t_0.
\end{equation*}
That is,
\begin{equation*}
	\begin{cases}
        \bar{s}_{t}-\bar{s}_{\bar\varrho\bar\varrho}-\big(\mathcal H^{i}(\bar\varrho-\rho_0,t)-\frac{k}{t}\big)(\bar{s}_{\bar\varrho}-\lambda_0 \bar{s})\geq Ct^{-(1+\lambda)},\quad&|\bar\varrho|\leq t^{\gamma},\ t>t_0,\\
		\bar{s}(-t^{\gamma},t)=\bar{s}(t^{\gamma},t)=t^{-\lambda}\cos(1),&t>t_0,\\
		\bar{s}(\bar\varrho,t_0)\geq t_0^{-\lambda}\cos(1),&|\bar\varrho|\leq t_0^{\gamma}.
	\end{cases}
\end{equation*}
	
Since $s$ is a locally bounded  function, there exists $C>0$ such that $|s(\bar\varrho,t_0)|\leq C\bar{s}(\bar\varrho,t_0)$ for all $|\bar\varrho|\leq t_0^{\gamma}$. Then, by the linearity of the equation, the comparison principle yields
\begin{equation*}
	-C\bar{s}(\bar\varrho, t)\leq	s(\bar\varrho,t)\leq C\bar{s}(\bar\varrho, t)\quad\textrm{for all }|\bar\varrho|\leq t^{\gamma},\ t>t_{0}.
\end{equation*}
Therefore, 
$$\displaystyle	\lim\limits_{t\to\infty}\sup_{|\bar\varrho|\leq t^{\gamma}}|s(\bar\varrho,t)|\leq\lim_{t\to\infty}\sup_{|\bar\varrho|\leq t^{\gamma}}C\bar{s}(\bar\varrho,t)=0,$$ as desired.
\end{proof}

At this point, we have all the ingredients to prove  Theorem \ref{ThConvergenceTravellingWave} in the diffusive region.  For this, we need to estimate the difference between $u$ and an appropriate shift to $\Phi_{c_0}$ in moving coordinates. The key idea here is to prepare $u$ and $\Phi_{c_0}$ in order to use Lemma~\ref{LemmaS}.

First we rewrite the estimate in Proposition \ref{Prop7-RadialEuclideo} back to the notation $v$ from by~\eqref{vdefinicion}, obtaining
\begin{equation}\label{estimate-v-above-below}
        (\alpha_{\infty}-C\varepsilon)(\varrho-Ct^{\delta}t_\varepsilon^{\frac{1}{2}-\delta})\leq{v}(\varrho,t)\leq (\alpha_{\infty}+C\varepsilon)(\varrho-Ct^{\delta}t_\varepsilon^{\frac{1}{2}-\delta}),\quad \tau\geq\tau_{\varepsilon},\  \eta\in[\eta_{\delta}^{-},\tilde{\eta}_\varepsilon].
    \end{equation}
Consider the  family of functions
\begin{equation*}
    \psi^{\pm}_{\varepsilon,T,\delta}(\varrho)=(\alpha_{\infty}\pm C\varepsilon)(\varrho \pm T^{\frac{1-\delta}{2}}\sqrt{|\varrho|}),\quad\varrho\in\mathbb{R}.
\end{equation*}
Fix $\delta_{0}=2\delta>0$. Then, by \eqref{estimate-v-above-below}, there exists $\varepsilon_0>0$ such that for each $\varepsilon\in (0,\varepsilon_0)$ there is $t_{\varepsilon}>0$ such that the function $v$ given by~\eqref{vdefinicion} satisfies
\begin{equation*}
    0\leq \psi^{-}_{\varepsilon,t_{\varepsilon},\delta_0}(t^{\delta_0})\leq v(t^{\delta_0},t)\leq\psi^{+}_{\varepsilon, t_{\varepsilon},\delta_0}(t^{\delta_0})\quad\textrm{for all }t\geq t_{\varepsilon}.
\end{equation*}

Now we go back to the original notation \eqref{vdefinicion}.
Let $\xi_{\varepsilon,\delta_0}^{\pm}$ be two functions such that
\begin{equation*}
\Phi_{c_0}\big(t^{\delta_0}+\xi_{\varepsilon,\delta_0}^{\pm}(t)\big)=e^{-\lambda_0 (t^{\delta_{0}}-\rho_0)}\psi^{\pm}_{\varepsilon, t_{\varepsilon},\delta_0}(t^{\delta_0}-\rho_0)\quad\textrm{for all }t\geq t_{\varepsilon}.
\end{equation*}
Since the implicit function theorem guarantees that $\xi_{\varepsilon,\delta_{0}}^{\pm}\in C^1$, using the decay estimate~\eqref{behaviorTw} for the profile $\Phi_{c_0}$  we have 
\begin{equation}\label{xi}	
    \xi_{\varepsilon,\delta_{0}}^{\pm}(t)=\rho_0-\tfrac1{\lambda_0}\log(\alpha_{\infty}\pm C\varepsilon)+O(t^{-\delta_0}), \quad (\xi^\pm_{\varepsilon,\delta_0})'(t)=O(t^{-(1+\delta_0)}).
\end{equation}
We now construct comparison functions $v_{\varepsilon,\delta_0}^{\pm}$ matching exactly $v$ at $\varrho+\rho_0=t^{\delta_0}$. They have to be solutions to
\begin{equation*}
	\begin{cases}
        z_{t}=z_{\varrho\varrho}+\big(\mathcal H^{i}-\frac{k}{t}\big)(z_{\varrho}-\lambda_0 z)+ e^{\lambda_0\varrho} f(e^{-\lambda_0\varrho}z)-f'(0)z,&\varrho+\rho_0\in (-t^{\delta_{0}},t^{\delta_{0}}),\;\;t>t_\varepsilon,\\
		z(\varrho,t_{\varepsilon})=v(\varrho,t_{\varepsilon}),&\varrho+\rho_0\in(-t_\varepsilon^{\delta_{0}},t_\varepsilon^{\delta_{0}}),
	\end{cases}
\end{equation*} 
with Dirichlet conditions
\[
    \begin{array}{l}
	   v_{\varepsilon,\delta_0}^\pm(t^{\delta_0}-\rho_0,t)=\psi^{\pm}_{\varepsilon, t_\varepsilon,\delta_0}(t^{\delta_{0}}-\rho_0),\\
	   v_{\varepsilon,\delta_0}^{+}(-t^{\delta_0}-\rho_0,t)=e^{-\lambda_0t^{\delta_{0}}},\\
	   v_{\varepsilon,\delta_0}^{-}(-t^{\delta_0}-\rho_0,t)=0,
    \end{array}
    \qquad\qquad\textrm{for all }t\geq t_{\varepsilon}.
\]
Since $0\leq u\leq 1$ and $v(\varrho,t)=e^{\lambda_0\varrho}u(\varrho+R(t),t)$, we have, by the comparison principle,
\begin{equation}\label{vComparison}	
    v^{-}_{\varepsilon,\delta_0}(\varrho,t)\leq v(\varrho, t )\leq v^{+}_{\varepsilon,\delta_0}(\varrho,t)\quad\textrm{for all } t\geq t_{\varepsilon}\textrm{ and } \varrho+\rho_0\in (-t^{\delta},t^{\delta}).
\end{equation}
We claim that, to be proved later, that
\begin{equation}\label{Vto0}
    \lim\limits_{t\to\infty}\sup\limits_{|\varrho+\rho_0|\leq t^{\delta_0}}\big|e^{-\lambda_0\varrho}v^{\pm}_{\varepsilon,\delta_0}(\varrho,t)
    -\Phi_{c_0}\big(\varrho+\xi_{\varepsilon,\delta_0}^{\pm}(t)\big)\big|=0.
\end{equation}
Therefore, as a consequence of the above, since $\Phi_{c_0}$ is continuous and $\varepsilon$ is arbitrary small, using~\eqref{xi} and~\eqref{vComparison}, we obtain that
\begin{equation}\label{vToTWintdelta}	
    \lim\limits_{t\to\infty}\sup\limits_{|\varrho+\rho_0|\leq t^{\delta_0}}\big|e^{-\lambda_0\rho}v(\varrho,t)-\Phi_{c_0}\big(\varrho+\rho_0-\tfrac1{\lambda_0}\log\alpha_{\infty}\big)\big|=0.
\end{equation}
Thus, by the definition to $v$ ---see~\eqref{utransladadaR(t)} and~\eqref{vdefinicion}---  we have
\begin{equation}\label{u-to-tw-diffusive-zone}	
    \lim\limits_{t\to\infty}\sup\limits_{ \{x\in\Hd: \,c_*t-\frac{3}{c_0}t-t^{\delta_{0}}\leq \rho_i(x)\leq c_*t-\frac{3}{c_0}\log t+t^{\delta_0}\}}\big|u\big(\rho_i(x),t)
	-\Phi_{c_0}(\rho_i(x)-[c_*t-\tfrac{3}{c_0}\log t-\tfrac{1}{\lambda_0}\log \alpha_0]\big)\big|=0.
\end{equation}
This would complete the proof of formula~\eqref{u-to-tw-equation-main-theorem} in Theorem~\ref{ThConvergenceTravellingWave} if we restrict the supremum to the diffusive region.

We now show claim~\eqref{Vto0}. As in the proof of Lemma \ref{LemmaS}, it is more convenient to translate by $\rho_0$ by setting $\bar\varrho=\varrho+\rho_0$. Let
\begin{equation*}
    V_{\varepsilon,\delta_{0}}^{\pm}(\bar\varrho,t)=v_{\varepsilon,\delta}^{\pm}(\bar\varrho,t)-e^{\lambda_0\bar\varrho}\Phi(\bar\varrho+\xi^{\pm}_{\varepsilon,\delta_{0}}(t)) 
\end{equation*}
and \begin{equation*}
    \tilde{F}(\sigma)=f(\sigma)-f'(0)\sigma.
\end{equation*}
Then $V_{\varepsilon,\delta_{0}}^{\pm}(\bar\varrho,t)$ are  solutions to
\begin{align*}
    z_{t}-z_{\bar\varrho\bar\varrho}-\big(\mathcal H^{i}(\bar\varrho-\rho_0,t)-\tfrac{k}{t}\big)(z_{\bar\varrho}-\lambda_0 z)&=-\big((\xi^{\pm}_{\varepsilon,\delta_{0}})^{'}-\mathcal H^{i}(\bar\varrho-\rho_0,t)+\tfrac{k}{t}\big)e^{\lambda_0\bar\rho}\Phi_{c_0}^{'}(\bar\varrho+\xi^{\pm}_{\varepsilon,\delta_{0}}(t))\\
    &\quad +e^{\lambda_0\bar\varrho}\big(\tilde{F}(e^{-\lambda_0\bar\varrho}v^{\pm}_{\varepsilon,\delta_0})-\tilde{F}(\Phi_{c_0}(\bar\varrho+\xi^{\pm}_{\varepsilon,\delta_{0}}(t))\big)
\end{align*}
for all $\bar\varrho\in (-t^{\delta_{0}},t^{\delta_{0}})$ and $t>t_\varepsilon$, with Dirichlet conditions $V_{\varepsilon,\delta_{0}}^{\pm}(t^{\delta_0},t)=0$,
\begin{equation*}
    \begin{aligned}
        &\begin{aligned}
            &V_{\varepsilon,\delta_{0}}^{+}(-t^{\delta_0},t)=e^{-\lambda_0t^{\delta_0}}(1-\Phi_{c_0}(-t^{\delta_0}
            +\xi_{\varepsilon,\delta_0}^{+}(-t^{\delta_{0}})))\\ &V_{\varepsilon,\delta_{0}}^{-}(-t^{\delta_0},t)=e^{-\lambda_0t^{\delta_0}}\Phi_{c_0}(-t^{\delta_0}
            +\xi_{\varepsilon,\delta_0}^{-}(-t^{\delta_{0}}))
        \end{aligned}
        &&\text{for all }t\geq t_{\varepsilon},\\
        &V_{\varepsilon,\delta_{0}}^{\pm}(\bar\varrho,t_\varepsilon)=v_{\varepsilon,\delta}^{\pm}(\bar\varrho,t_\varepsilon)-e^{\lambda_0\bar\varrho}
        \Phi_{c_0}(\bar\varrho+\xi^{\pm}_{\varepsilon,\delta_{0}}(t_\varepsilon))&&\textrm{for all }|\bar\varrho|\leq t^{\delta_{0}}.
    \end{aligned}
\end{equation*}
Observe that  if $|\bar\varrho|\leq t^{\delta_0}$, since $R(t)\asymp c_*t$, then $\mathcal{H}^i(\bar\varrho-\rho_0,t)=O(1/t)$ for large time; see the proof of Lemma \ref{lemma:bound-H}. Therefore, using also~\eqref{xi}, we have that 
\begin{equation}\label{error2}
   ( \xi^{\pm}_{\varepsilon,\delta_{0}})^{'}-\mathcal H^{i}(\bar\varrho-\rho_0,t)+\tfrac{k}{t}=O\big(\tfrac{1}{t}\big)\quad\textrm{for all }|\bar\varrho|\leq t^{\delta_0}.
\end{equation}
Moreover, we claim that there exists $C>0$ with
\begin{equation}\label{ClaimDerivatiestraveling wave}
    |e^{\lambda_0 \bar\varrho }\Phi_{c_0}^{'}(\bar\varrho)|\leq C(1+\bar\varrho^{2}\mathds{1}_{\{\bar\varrho>0\}}(\bar\varrho))\quad\textrm{for all }\bar\varrho\in\mathbb{R}.
\end{equation}
To see this, let $\omega(\bar\varrho)=e^{\lambda_0\bar\varrho}\Phi'_{c_0}(\bar\varrho)$, $\bar\varrho\in\mathbb{R}$. Then, 
\begin{equation*}
    \omega_{\bar\varrho}+\lambda_0\omega=-e^{\lambda_0\bar\varrho}f(\Phi_{c_0}(\bar\varrho)),\quad\bar\varrho\in\mathbb{R},
\end{equation*}
whence
\begin{equation*}
	\omega(\bar\varrho)=-|\Phi'_{c_0}(0)|e^{-\lambda_0\bar\varrho}-\int_{0}^{\bar\varrho}e^{\lambda_0(2s-\bar\varrho)}f(\Phi_{c_0}(\sigma))\,{\rm d}\sigma.
\end{equation*}
Using again that $f$ is a KPP function and  the fact that $0\leq \Phi_{c_0}\leq 1$  has the decay~\eqref{behaviorTw}, we get~\eqref{ClaimDerivatiestraveling wave}. 

The bounds~\eqref{error2} and~\eqref{ClaimDerivatiestraveling wave} imply that, if $|\bar\varrho|\le t^{\delta_0}$, then
\begin{equation}\label{cotaf2}	
    \sup\limits_{|\bar\varrho|\leq t^{\delta_0}}\Big|\big((\xi^{\pm}_{\varepsilon,\delta_{0}})^{'}-\mathcal H^{i}(\bar\varrho-\rho_0,t)+\tfrac{k}{t}\big) e^{\lambda_0\bar\varrho}\Phi_{c_0}^{'}(\bar\varrho+\xi^{\pm}_{\varepsilon,\delta_{0}}(t)\big)\Big|=O(t^{-(1-2\delta_0)}).
\end{equation}
On the other hand, thanks to~\eqref{KPPfunction.2*}, $\tilde{F}'\leq 0$, which implies that
\begin{equation}\label{signof1}
    \operatorname{sign}(V^{\pm}_{\varepsilon,\delta_0})e^{\lambda_0\bar\varrho}\big(\tilde{F}(e^{-\lambda_0\bar\varrho}v^{\pm}_{\varepsilon,\delta_0})-\tilde{F}(\Phi_{c_0}(\bar\varrho+\xi^{\pm}_{\varepsilon,\delta_{0}}(t))\big)\leq 0.
\end{equation}
Notice that hypothesis~\eqref{KPPfunction.2} would not be enough here.
Then, combining \eqref{signof1} and~\eqref{cotaf2} and using Kato's inequality $|V|_{\bar\varrho\bar\varrho}\geq\operatorname{sign}(V)V_{\bar\varrho\bar\varrho}$, we have that  there is $C>0$ such that $|V_{\varepsilon,\delta_{0}}^{\pm}|$ is a subsolution of
\begin{equation*}
	\begin{cases}
		z_{t}-z_{\bar\varrho\bar\varrho}-\big(\mathcal H^{i}(\bar\varrho-\rho_0,t)-\frac{k}{t}\big)(z_{\bar\varrho}-\lambda_0 z)=C t^{-(1-2\delta_0)},\quad &|\bar\varrho|\leq t^{\delta_0},\ t>t_0,\\
		z(-t^{\delta_0},t)=O(e^{-\lambda_0t^{\delta_0}}),\quad z(t^{\delta_0},t)=0,&t>t_0.
	\end{cases}
\end{equation*}
Thus, taking $\delta_0\in (0,1/2)$,  since $V_{\varepsilon,\delta_{0}}^{\pm}$ is locally bounded, by Lemma~\ref{LemmaS} and the comparison principle we obtain~\eqref{Vto0}.

%%%%%%%%%%%%%%%%%%%%%%%%%%%%
\subsection{Superdiffusive region} \label{subsection:superdiffusive}

The following result captures the behavior of $u$ in the superdiffusive region.

\begin{proposition}\label{ThmNecesityLogarithmicCorrection} 
    Consider $f'(0)>\lambda_{1}$ and $u_0$ satisfying~\eqref{eq:symmetry.u0}--\eqref{eq:supported-from-right} for some~$i\in\{\mathfrak{e},\mathfrak{h},\mathfrak{p}\}$. Then, there is $C>0$ for which the solution $u$ to~\eqref{KPPproblem} with initial datum $u_0$ satisfies, for any $\varepsilon>0$,
    \begin{equation}\label{NecesityOfLogarithmicCorrection}	
        \lim\limits_{t\to\infty}\sup\limits_{\{x\in\Hd:\;\rho_i(x)>c_*t-\frac{3}{\co}\log t+(\frac{1}{c_0}+\varepsilon)\log t\}}u(x,t)=0.
    \end{equation}
\end{proposition}

\begin{proof}
Recall that we are rewriting the  solution $u$ to~\eqref{KPPproblem}   as \begin{equation}\label{change:u-w}u(\varrho+R(t),t)=e^{-\lambda_0\varrho}t^{1/2}{w}(\varrho t^{-1/2},\log t).
\end{equation}
By  point (i) in Proposition~\ref{Sub/Super for hat{w}} we have that, for each $\delta\in (0,1/4)$ there is $C>0$ such that
\begin{equation*}
    u(\varrho+R(t),t)\leq Ce^{-\lambda_0\varrho}\big((\varrho+t^{\delta})e^{-\frac{(\varrho+t^\delta)^2}{4t}} +t^\delta\big(1+e^{-\frac{\varrho^2}{16t}}\big)e^{-\frac{\varrho^2}{8t}}\big)\quad\textrm{for all }\varrho\geq 0\;\;\textrm{and }t>0.
\end{equation*}
Then, since $\delta\in (0,1/4)$ and $0\leq se^{-s^2}\leq 1$ for all $s\geq 0$, we have that for all $t_1>0$ there is $C>0$ such that
\begin{equation*}
	u(\rho+R(t),t)\leq Ct^{1/2}e^{-\lambda_0\rho}\quad\textrm{for all }\rho\geq 0\textrm{ and }t\geq t_1.
\end{equation*}
Therefore, since $c_0=2\lambda_0$, for all $\varepsilon>0$,
\begin{equation*}
	\lim\limits_{t\to\infty}\sup\limits_{\{\varrho\geq (\frac{1}{c_0}+\varepsilon)\log t\}}u(\varrho+R(t),t)=0.
\end{equation*}
Recalling the expression of $R(t)$ from \eqref{R,c0,c*,lambda0} with $k=3/\co$ (see~\eqref{k}), we conclude~\eqref{NecesityOfLogarithmicCorrection}.
\end{proof}

\begin{corollary}
    Under the hypotheses of Theorem~\ref{ThConvergenceTravellingWave}, formula~\eqref{u-to-tw-equation-main-theorem} holds true if we restrict the supremum to the superdiffusive region.
\end{corollary}

\begin{proof}
Since  $\Phi_{c_0}$ is a decreasing continuous function and $\displaystyle\lim_{s\to\infty}\Phi_{c_0}(s)=0$, we get
\begin{equation*}
		\lim\limits_{t\to\infty}	\big\| \Phi_{c_0}\big(\cdot-[c_*t-\tfrac{3}{c_0}\log t+\beta]\big)\big\|_{L^{\infty} (R(t)+t^{\delta_0},\infty)}=\lim\limits_{t\to\infty}\Phi_{c_0}\big(t^{\delta_{0}}-\tfrac{1}{\lambda_0}\log\alpha_{\infty}\big)=0.
\end{equation*}
On the other hand, by Proposition~\ref{ThmNecesityLogarithmicCorrection}, we have that
\begin{equation*}
    \lim\limits_{t\to\infty}\sup\limits_{\{x\in\Hd:\, \rho_i(x)\geq R(t)+t^{\delta_0}\}} u(x,t)=0.
\end{equation*}
As a consequence,
\begin{equation*}
    \lim\limits_{t\to\infty}\sup\limits_{ \{x\in\Hd: \, \rho_i(x)\geq R(t)+t^{\delta_0}\}}\big|u(\rho_i(x),t)
    -\Phi_{c_0}\big(\rho_i(x)-R(t)-\tfrac1{\lambda_0}\log \alpha_{\infty}\big)\big|=0.\qedhere
\end{equation*}
\end{proof}

%%%%%%%%%%%%%%%%%%%%%%%%%%%%
\subsection{Subdiffusive region}
\label{subsection:subdiffusive}
Now, we will prove that $u$ converges to $1$ in the subdiffusive region.

\begin{proposition}\label{ThmLowerSpredingSpeedBestia}
    Consider $f'(0)>\lambda_{1}$ and $u_0$ satisfying~\eqref{eq:symmetry.u0}--\eqref{eq:supported-from-right} for some~$i\in\{\mathfrak{e},\mathfrak{h},\mathfrak{p}\}$. Then, the solution $u$ to the~\eqref{KPPproblem} problem with initial datum $u_{0}$ satisfies that for all $g:[0,\infty)\to\mathbb{R}$  such that $\displaystyle\lim_{t\to\infty}g(t)=\infty$, we have
    \begin{equation}\label{LowerSpreedingSpeed}	
        \lim\limits_{t\to\infty}\inf_{\{x\in\Hd,\ 0\leq\rho_{i}(x)\leq c_{*}t-\frac{3}{c_{0}}\log  t-g(t)\}}u(x,t)=1.
    \end{equation}
\end{proposition}

\begin{proof}
\textsc {Step 1.} Fix $m\in (0,1)$. Let $f_1:[0,1]\to\mathbb{R}$  be a  $C^1$ function such that
\begin{gather}\label{f_1}
	f_1(0)=f_1(m)=0,\quad f_1'(0)=f'(0),\qquad f_1\leq f\quad\textrm{in }[0,m],\quad f_1>0\quad\textrm{in }(0,m).
\end{gather}
Then, following arguments of~\cite[Theorem 4.1]{AronsonWeinberger}, there exists $ \tilde{\Phi}_{c_0}:\mathbb{R}\to (0,m)$ such that
\begin{gather}
    \label{EqTildeTW} \tilde{\Phi}_{c_0}''+\co\tilde{\Phi}_{c_0}'+f_1(\tilde{\Phi}_{c_0})=0\quad\textrm{in }\mathbb{R},\\
	\notag\lim\limits_{s\to-\infty}\tilde{\Phi}_{c_0}(s)=m,\quad\tilde{\Phi}_{c_0}'<0,\quad\lim\limits_{s\to+\infty}\tilde{\Phi}_{c_0}(s)=0,
\end{gather}
and for some $B>1$ and $\rho_1>0$
\begin{equation}\label{behaviorFalseTw}
    B^{-1}s e^{-\lambda_0 s}\leq \tilde{\Phi}_{c_0}(s)\leq Bs e^{-\lambda_0 s}\quad\textrm{for all }s\geq \rho_{1}.
\end{equation}

Our aim is to prove that
\begin{equation*}
	\underline{u}(\rho,t)=\tilde{\Phi}_{c_0}(\rho-R(t)+\rho_3)
\end{equation*}
is a subsolution for $u$ in $\{(\rho,t):\;\rho\in A_{i},\;t\geq T\}$, where $\rho_3, T>0$ will be chosen later and
\begin{equation*}
	A_{i}(t)=
    \begin{cases}
		[0,c_*t+t^{1/2}]&\textrm{if }i\in\{\mathfrak{h},\mathfrak{p}\},\\
		[\frac{c*}{2}t,c_*t+t^{1/2}]&\textrm{if }i=\mathfrak{e}.
	\end{cases}
\end{equation*}

On the one hand, since $\tilde{\Phi}_{c_0}$ satisfies~\eqref{EqTildeTW}, then, using~\eqref{f_1} and the fact that $0\leq \underline{u}\leq m$,
\begin{align*}
    \underline{u}_{t}-\underline{u}_{\rho\rho}-(d-1)h^{i}_{1}\underline{u}_{\rho}-f(\underline{u})&=f_{1}(\underline{u})-f(\underline{u})+\Big(\tfrac{k}{t}+(d-1)(1-h^{i}_{1}(\rho))\Big)\underline{u}_{\rho}\\
    &\leq \Big(\tfrac{k}{t}+(d-1)(1-h^{i}_{1})\Big)\underline{u}_{\rho}.
\end{align*}
Moreover, $1-h^{i}_{1}\geq 0$ for $i\in\{\mathfrak{h},\mathfrak{p}\}$ and since $\tilde{\Phi}_{c_0}'<0$ implies $\underline{u}_{\rho}<0$, we have
\begin{equation}\label{sub u for h,p}	
    \underline{u}_t-\underline{u}_{\rho\rho}-(d-1)h^{i}_{1}(\rho)\underline{u}_{\rho}-f(\underline{u})\leq 0\quad\textrm{in }\{(\rho,t):\rho\in A_{i},\ t>0\}
\end{equation}
for $i\in\{\mathfrak{h},\mathfrak{p}\}$. For $i=\mathfrak{e}$  it holds $h^{i}_{\mathfrak{e}}(\rho)=\coth\rho$. Then,
\begin{equation*}
	1-\coth\rho=-\frac{2}{e^{2\rho}-1},\quad\rho>0,
\end{equation*}
which implies that there exists $t_{1}>0$ such that
\begin{equation*}
    \inf\limits_{\rho\in A_{\mathfrak{e}}(t)}\left(\frac{k}{t}+(d-1)(1-h^{i}_{1})\right)=\frac{k}{t}-\frac{2(d-1)}{e^{c_*t}-1}>0\quad\textrm{for all }t\geq t_1.
\end{equation*}
Thus,
\begin{equation}\label{sub u for e}	
    \underline{u}_{t}-\underline{u}_{\rho\rho}-(d-1)\coth\rho\,	\underline{u}_{\rho}-f(	\underline{u})\leq 0\quad\textrm{in } \{(\rho,t):\rho\in A_{\mathfrak{e}},\;t\geq t_{1}\}.
\end{equation}

On the other hand, by \cite[Theorem 3.2 ii)]{Matano} (see also Proposition~\ref{prop:convergence.to.1.compact}) we know that there exists $t_{2}\geq t_1$ such that
\begin{equation*}
	u(0,t)\geq m\quad\textrm{for all }t\geq t_2\;\;\textrm{for }i\in\{\mathfrak{h},\mathfrak{p}\},
\end{equation*}
and by \cite[Theorem 3.6, ii)]{Matano} it holds that 
\begin{equation*}
	\inf\limits_{\{0\leq\rho\leq\frac{c_*}{2}t\}}u(\rho,t)\geq m\quad\textrm{for all }t\geq t_2\textrm{ for }i=\mathfrak{e}.
\end{equation*}
Then, since $\tilde{\Phi}_{c_0}\leq m$,  for all $t\geq t_2$, $u(0,t)\geq \underline{u}(0,t)$ if $i\in\{\mathfrak{h},\mathfrak{p}\}$ and $u(\frac{c_*}{2}t,t)\geq \underline{u}(\frac{c_*}{2}t,t)$ if $i=\mathfrak{e}$.

In addition, recalling \eqref{change:u-w}, by Proposition~\ref{Prop7-RadialEuclideo}, we have that, for $\delta\in (0,1/4)$, there exists $t_3\geq t_2$ such that
\begin{equation}\label{Cota1LowerSpreadingSpeed}
    u(\varrho+R(t),t)\geq\frac{\alpha_\infty}{2}(\varrho-Ct^\delta)e^{-\lambda_{0}\varrho}\quad\textrm{for all }\varrho\in[-t^{\delta},at^{1/2}],\ t\geq t_3,
\end{equation}
for some $a>1$ and  where $\alpha_\infty$ and $C$ denotes the same positive constants as in Proposition~\ref{Prop7-RadialEuclideo}. Then, since there is $T\geq t_3$ such that  $\varrho=t^{1/2}+\frac{3}{c_0}\log  t -\rho_0\in [-t^{\delta}, a t^{1/2}]$ for all $t\geq T$, using~\eqref{Cota1LowerSpreadingSpeed}, we conclude that
\begin{equation*}
	u(c_*t+t^{1/2},t)\geq\frac{\alpha_\infty e^{\lambda_0\rho_0} e^{-\lambda_{0}t^{1/2}}}{2t^{3/2}}\big(t^{1/2}+\tfrac{3}{c_0}\log t -\rho_0-Ct^{\delta}\big)\geq \frac{\alpha_\infty e^{\lambda_0\rho_0} e^{-\lambda_{0}t^{1/2}}}{4t}
\end{equation*}
for all $t\geq T$. This implies, together with~\eqref{behaviorFalseTw} and the choice~\eqref{k}, that for $\rho_3>0$ big enough,
\begin{equation*}
	u(c_* t+t^{1/2},t)\geq \tilde{\Phi}_{c_0}\left(t^{1/2}+\tfrac{3}{c_0}\log t+\rho_3\right)=\underline{u}(c_* t+t^{1/2},t)\quad\textrm{for all }t\geq  T.
\end{equation*}

Finally, since the solution of the heat equation~\eqref{HeatEquationHd} with initial datum $u_0$ is strictly positive and $f\geq 0$, by the comparison principle, Theorem~\ref{ComparsionPrincipleHd}, we have that
\begin{equation*}
	\min\limits_{\{0\leq \rho\leq c_{*}T+T^{1/2}\}}u(\rho,T)>0.
\end{equation*}
Moreover, since $\displaystyle\lim_{s\to\infty}\tilde{\Phi}_{c_0}(s)=0$, it holds that, for $\rho_3>0$ big enough,
\begin{equation*}
	u(\rho,T)\geq \min\limits_{\{0\leq \rho\leq c_{*}T+T^{1/2}\}}u(\rho,T)\geq \tilde{\Phi}_{c_0}(-R(T)+\rho_3)\geq \tilde{\Phi}_{c_0}(\rho-R(T)+\rho_3)=\underline{u}(\rho,T)
\end{equation*}
for all $\rho\in [0,c_{*}T+T^{1/2}]$.
	
Summarizing, $\underline{u}\leq u$ in the boundary $\partial\{(\rho,t):\,\rho\in A_{i}, t\geq T\}$. Therefore, since $\underline u$ is a subsolution of the equation, see ~\eqref{sub u for h,p} and~\eqref{sub u for e}, by the comparison principle, $\underline{u}\leq u$ in the full region $\{(\rho,t):\,\rho\in A_{i}, t\geq T\}$.
	
\noindent\textsc {Step 2.} Now we prove~\eqref{LowerSpreedingSpeed}. Let $g$ be  such that $\lim\limits_{s\to\infty} g(t)=\infty$. If $i\in\{\mathfrak{h},\mathfrak{p}\}$, by Step 1, and using also the monotonicity of the profile,
\begin{equation*}
    \lim\limits_{t\to\infty}\inf\limits_{\{0\leq\rho\leq R(t)-g(t)\}}u(\rho,t)\geq 	\lim\limits_{t\to\infty}\inf\limits_{\{0\leq\rho\leq R(t)-g(t)\}}\tilde{\Phi}_{c_0}(\rho-R(t)+\rho_3)\ge\lim\limits_{t\to\infty}\tilde{\Phi}_{c_0}(-g(t)+\rho_3)=m.
\end{equation*}
Since $u\in [0,1]$ and $m\in (0,1)$ arbitrary, we obtain~\eqref{LowerSpreedingSpeed} for $i\in\{\mathfrak{h},\mathfrak{p}\}$.
	
As for $i=\mathfrak{e}$, by \cite[Theorem 3.6, ii)]{Matano}, we know that
\begin{equation*}
	\lim\limits_{t\to\infty}\inf\limits_{\{0\leq\rho\leq \frac{c_*}{2}t\}}u(\rho,t)=1.
\end{equation*}
Then, since $u\in[0,1]$, using  Step 1 and proceeding as before,
\begin{equation*}
    \lim\limits_{t\to\infty}\inf\limits_{\{0\leq\rho\leq R(t)-g(t)\}}u(\rho,t)=\lim\limits_{t\to\infty}\inf\limits_{\{\frac{c_*}{2}t\leq\rho\leq R(t)-g(t)\}}u(\rho,t)\geq \lim\limits_{t\to\infty}\tilde{\Phi}_{c_0}(-g(t)+\rho_3)=m.
\end{equation*}
Once more, since the above holds for all $m\in (0,1)$, we get~\eqref{LowerSpreedingSpeed}.
\end{proof}

\begin{corollary}
    Under the hypotheses of Theorem~\ref{ThConvergenceTravellingWave}, formula~\eqref{u-to-tw-equation-main-theorem} holds true if we restrict the supremum to the  subdiffusive region.
\end{corollary}
\begin{proof}
Since $\Phi_{c_0}$ is a decreasing continuous function and $\displaystyle\lim_{s\to-\infty}\Phi_{c_0}(s)=1$,
\begin{equation*}
		\lim\limits_{t\to\infty}	\big\| 1-\Phi_{c_0}\big(\cdot-[c_*t-\tfrac{3}{c_0}\log t+\beta]\big)\big\|_{L^{\infty}(0, R(t)-(\rho_0+t^{\delta_0}))}=\lim\limits_{t\to\infty}\big(1-\Phi_{c_0}(-t^{\delta_{0}}-\tfrac{1}{\lambda_0}\log \alpha_{\infty})\big)=0.
\end{equation*}
Besides, by Proposition~\ref{ThmLowerSpredingSpeedBestia}, we have that
\begin{equation*}
    \lim\limits_{t\to\infty}\inf\limits_{\{x\in\Hd:\,0\leq \rho_i(x)\leq R(t)-(\rho_0+t^{\delta_0})\}} u(x,t)=1,
\end{equation*}
whence
\begin{equation*}
    \lim\limits_{t\to\infty}\sup\limits_{ \{x\in\Hd: \,0\leq \rho_i(x)\leq c_* t-\tfrac{3}{c_0}\log t-t^{\delta_0}\}}\big|u(\rho_i(x),t)
    -\Phi_{c_0}\big(\rho_i(x)-[c_*t-\tfrac{3}{c_0}\log t+\beta]\big)\big|=0.\qedhere
\end{equation*}
\end{proof}

%%%%%%%%%%%%%%%%%%%%%%%%%%%%%%%%%%%%%%%%%%%%%%%%%%
% ACKNOWLEDGMENTS
%%%%%%%%%%%%%%%%%%%%%%%%%%%%%%%%%%%%%%%%%%%%%%%%%%%%%%%%%%%%%
\section*{Acknowledgments}

This research was finantially supported by MICIU/AEI/10.13039/501100011033 (Spain) through grants CEX2023-001347-S, RED2022-134784-T, RED2024-153842-T (all three authors), PID2020-113596GB-I00, PID2023-150166NB-I00 (MG), PID2019-110712GB-I00 (IG), PID2020-116949GB-I00 (FQ), and PID2023-146931NB-I00 (IG and FQ).

All three authors received also financial support from the Madrid Government (Comunidad de Madrid – Spain), under the multianual Agreement with UAM in the line for the Excellence of the University Research Staff in the context of the V PRICIT (Regional Programme of Research and Technological Innovation).
	
%%%%%%%%%%%%%%%%%%%%%%%%%%%%%%%%%%%%%%%%%%%%%%%%%%%%%%
%          7. REFERENCES SECTION
%%%%%%%%%%%%%%%%%%%%%%%%%%%%%%%%%%%%%%%%%%%%%%%%%%%%%%

\end{document}